\documentclass[12pt]{article}
\usepackage[top=1in, bottom=1in, left=1in, right=1in]{geometry}
\usepackage{graphicx} 
\usepackage{comment}
\usepackage{geometry}
\usepackage{subcaption}
\usepackage{tkz-tab}
\usepackage{tikz}
\usepackage{cellspace} 
\usepackage{amsmath}
\usepackage{amssymb}
\usepackage[title]{appendix}
\usepackage[colorlinks=true, allcolors=blue]{hyperref}
\DeclareMathOperator{\R}{\mathbb{R}}
\DeclareMathOperator{\C}{\mathbb{C}}
\DeclareMathOperator{\F}{\mathcal{F}}
\DeclareMathOperator{\Sw}{\mathcal{S}}

\DeclareMathOperator{\FLs}{(-\Delta)^s}

\usepackage{amsthm}
\newtheorem{theorem}{Theorem}[section]
\newtheorem{lemma}[theorem]{Lemma}
\newtheorem{remark}[theorem]{Remark}
\newtheorem{proposition}[theorem]{Proposition}
\newtheorem{corollary}[theorem]{Corollary}
\theoremstyle{definition}
\newtheorem{definition}{Definition}[section]

\newcommand\restr[2]{{
  \left.\kern-\nulldelimiterspace 
  #1 
  \littletaller 
  \right|_{#2} 
  }}

\newcommand{\littletaller}{\mathchoice{\vphantom{\big|}}{}{}{}}

\begin{document}

\title{{Limiting Absorption Principle and Radiation Condition for the Fractional Helmholtz Equation}}
\author{Dana Zilberberg\footnote{Department of Mathematics, Rutgers University, New Brunswick, NJ, USA (dana.zilberberg@gmail.com)} \ \  Fioralba Cakoni \footnote{Department of Mathematics, Rutgers University, New Brunswick, NJ, USA (fc292@math.rutgers.edu)} \ \ and Michael S. Vogelius \footnote{Department of Mathematics, Rutgers University, New Brunswick, NJ, USA (vogelius@math.rutgers.edu)}}
\maketitle
\normalsize
\fontsize{12}{18}\selectfont
\begin{abstract}
  We investigate elliptic fractional equations in the whole space ${\mathbb R}^n$, $n=1,2,3$, involving zero-order perturbations of the fractional Laplacian $(-\Delta)^s$, $0<s<1$. Our main objective is to determine  appropriate radiation conditions at infinity that ensure existence and uniqueness of solutions to the fractional type Helmholtz equation. Extending classical scattering theory for the Helmholtz equation, we introduce and analyze suitable Sommerfeld-type radiation conditions for fractional orders. A central contribution is the explicit computation of the outgoing free-space Green’s function for the operator $(-\Delta)^s-k^{2s}$, for all $0<s<1$, any dimension $n=1,2,3$, and $k>0$,  obtained via contour integration and a limiting absorption principle. We show that its asymptotic behavior at infinity coincides with a rescaled version of the classical Helmholtz fundamental solution, thereby justifying the standard Sommerfeld radiation condition for compactly supported sources. In addition, using resolvent estimates and a limiting absorption framework, we establish existence and uniqueness of outgoing solutions for compactly supported data, and for weighted sources when $s\geq 1/2$. We further derive a convolution representation of the solution in terms of the outgoing fundamental solution. For inhomogeneous media with compactly supported perturbations, we reformulate the problem as a Lippmann–Schwinger integral equation of Fredholm type and prove unique solvability away from a discrete set of frequencies. Our analysis provides a rigorous foundation for scattering theory of fractional Helmholtz operators and offers a framework suitable for numerical implementation of these nonlocal wave propagation models. 
\end{abstract}

\section{Introduction}
Fractional partial differential equations (PDEs) generalize classical models by incorporating derivatives of non-integer order. In this work we consider the case of the fractional Laplacian of order $0<s<1$ in  ${\mathbb R}^n$, $n=1,2,3$. The corresponding operators are intrinsically nonlocal and admit several equivalent formulations; see, for instance, \cite{garofalo2017fractional}  for a basic selfcontained review of the subject. A fundamental characterization is provided by the Caffarelli-Silvestre extension \cite{CS-extesion}, which realizes the fractional Laplacian as the Dirichlet-to-Neumann map for a degenerate elliptic operator in one higher dimension.  More recently,  the attention has turned to inverse problems for fractional PDEs -- see for instance the monograph \cite{bill}. For elliptic fractional problems the recovery of a potential  term in the fractional Schr\"odinger equation from knowledge of the  Dirichlet-to-Neuman operator was initiated in \cite{inverse}. Owing to the strong unique continuation property associated with nonlocal operators, such inverse  problems often yield stronger uniqueness and stability results than their local counterparts (see \cite{inverse} and the references therein for state-of-the-art results). Most results for models with fractional derivatives in the spatial variable in Euclidean setting address interior problems, in which Dirichlet or Neumann data for the fractional Laplacian are prescribed on the exterior of a bounded domain. 

\noindent
In this paper we study elliptic fractional equations formulated  in the whole space ${\mathbb R}^n$, and the central issue that we address, for a particular class of elliptic fractional equations, is the formulation of the appropriate condition at infinity that guarantees uniqueness  and existence of a solution.  Our operators are zero order perturbations of the s'th power of the Laplacian in ${\mathbb R}^n$, for $s\in (0,1)$. The analysis of the forward problem for this model, serves as the starting point for properly defining the field at infinity (referred to as  scattering data). This data is essential for the solution of the corresponding inverse problem,   as demonstrated in the recent papers \cite{inverse-scat2,inverse-scat} for the fractional Schr\"odinger equation in ${\mathbb R}^n$, for restricted choices of $s$.
\noindent
We first consider the fractional Helmholtz equation with a source $f$
\begin{align}\label{fractional_Helmholtz}
    \FLs u - k^{2s} u = f \mbox{ in } \R^n
\end{align}
where $s\in (0,1), k\in (0,\infty)$, and the fractional Laplacian, $\FLs $, is a self adjoint operator with domain of definition $H^{2s}(\R^n)$(see for example \cite{di2012hitchhikers}, \cite{garofalo2017fractional}). Our main questions of interest are:  what is the range of  $\FLs - k^{2s}$ and under what condition at infinity is the solution $u$ to (\ref{fractional_Helmholtz}) unique. When $s=1$, the equation is the well-known Helmholtz equation whose unique solvability is classically obtained by adding a radiation condition, known as the Sommerfeld radiation condition (see e.g \cite{colton-book})
\begin{align}\label{def:SRC}
   r^{\frac{n-1}{2}}\left( \frac{\partial u}{\partial r} -ik u \right) \xrightarrow[r\to \infty]{}0,
\end{align}
uniformly in $\hat x = \frac{x}{|x|} \in \mathbb{S}^{n-1}$  which equivalently can be written in the integral form as
\begin{equation}
\lim\limits_{r\to \infty}\int_{|x|=r}\left|\frac{\partial u}{\partial r} -ik u \right|^2 \,ds \to 0. \nonumber
\end{equation}
The case when $s=1/2$ has been studied by Umeda in \cite{umeda1995radiation,umeda2003generalized}  where it is shown  that, $\sqrt{-\Delta}u-ku=f$ for $f\in L^{2,\delta}({\mathbb R}^n)$, $1/2<\delta<1$, has a unique solution $u\in L^{2,-\delta}({\mathbb R}^n)$,  provided one adds the radiation condition 
\begin{align}\label{gen}
  \int\limits_{{\mathbb R}^n\setminus B_R}  \frac{|\nabla u -ik\hat x u|^2}{(1+|x|^2)^{1-\delta}}dx <+\infty  \qquad \hbox{ with } \hat x = \frac{x}{|x|} .
\end{align}
The above radiation condition, known in the literature as the  Agmon--H\"ormander  or  Ikebe--Sait\={o} condition, in our paper  is  referred to  as the generalized Sommerfeld radiation condition, (as opposed to the standard Sommerfeld radiation condition (\ref{def:SRC})).  In fact, in the appendix we show that for a solution to the Helmholtz equation $\Delta u+k^2u=0$ outside a big ball, the standard Sommerfeld Radiation Condition (\ref{def:SRC}) is equivalent to the generalized Sommerfeld Radiation Condition.  The uniqueness result for $s=1/2$  is suggested by the fact that the free space Green's function, corresponding to the operator  $\sqrt{-\Delta}  - k$, as computed  in \cite{john1,umeda1995radiation}, exhibits the same asymptotic behavior at infinity as the fundamental solution of the classical Helmholtz operator. A goal of this paper is to extend these results to  any power $s\in (0,1)$.  Towards this goal, a main contribution is the explicit computation of what we refer to as the outgoing fundamental solution for the operator $\FLs u - k^{2s} u$ for any $s\in (0,\,1)$ and for $n=1,2,3$.

\noindent
Motivated by the scattering theory for inhomogeneous media in the case of Helmholtz equation (see e.g \cite{colton-book}), we also study the unique solvability of  the inhomogeneous fractional Helmholtz equation
\begin{align}\label{inhomogeneous_fractional_Helmholtz}
    \FLs u^{scat} - k^{2s}(1+q) u^{scat} = k^{2s}qu^{inc} \mbox{ in } \R^n
\end{align}
with a $L^\infty$- perturbation  $q$ of compact support, and a given probing wave  $u^{inc}$ which solves $\FLs u^{inc} - k^{2s} u^{inc}=0$ \footnote{Such a wave also satisfies $(\Delta+k^2)u^{inc}=0$ if $u^{inc} \in H^{2s,-\delta}(\R^n)$, as seen from Proposition \ref{from_fractional_Helmholtz_to_Helmholtz}}. In this case the solution is also uniquely characterized  by adding an appropriate radiation condition at infinity.  Note that  (\ref{inhomogeneous_fractional_Helmholtz}) is a particular case of 
\begin{equation} 
\label{inhomogeneous_fractional_Helmholtz2}
 \FLs u- k^{2s}(1+q) u = f \mbox{ in } \R^n
\end{equation}
with $f:=k^{2s}qu^{inc}$ (which has compact support).
\noindent
The fractional Helmholtz equation is used  to describe wave propagation in complex, attenuating media, or media with nonlocal properties  that cannot be accurately represented by the classical Helmholtz equation, such as wave propagation in lossy media, and  in complex geological formations, particularly in the context of nonlocal elasticity. It provides a framework for understanding and simulating wave propagation in fractal or inhomogeneous materials where the standard integer-order derivatives are insufficient.  Special choices of the perturbations of the fractional Helmholtz operator include the relativistic Schr\"odinger operator and the anomalous transport operator. The particular case of $s=1/2$ arises in quantum optics of a single photon interacting with a system of two-level atoms as discussed in \cite{john3,john4,john2,john1}.

\noindent 
The paper is organized as follows. Section 2 is devoted to computations of the outgoing free-space Green's function 
\(G_{n,s}^{k}(x)\) of the fractional Helmholtz 
operator \(\FLs - k^{2s}\), for every \(0<s<1\),$k>0$ and in dimension $n=1,2,3$. 
To the best of our knowledge, such computations are not available in the literature, 
except in the special case of \(s=\tfrac{1}{2}\). For \(k>0\), the outgoing and incoming 
free space Green's functions \(G_{n,s}^k(x)\)  (otherwise referred to as the outgoing and incoming fundamental solutions) are obtained based on the limiting absorption principle, that is
\[
    (G_{n,s}^{k})^{\pm}(x) = \lim_{\epsilon \to 0^+}  G_{n,s}^{k_\epsilon}(x) \qquad \mbox{where\;\;  $k^{2s}_\epsilon=k^{2s}\pm i \epsilon$}.
\] 
The calculations for complex parameters are carried out using contour integration  of Fourier-type integrals in the complex plane. Our focus here is the outgoing Green's function \( (G_{n,s}^{k})^+(x)\), which as \(|x|\to \infty\) for all $0<s<1$, turns  out  to behave asymptotically  like a rescaled version of the fundamental solution of the  classical Helmholtz operator  in the  respective dimension. This strongly indicates that the standard Sommerfeld Radiation Condition is the correct condition to impose at infinity in order to guaranty the uniqueness of the solution of (\ref{fractional_Helmholtz}), at least if $f$ has compact support. This leads us to another objective of our paper, namely to rigorously  establish existence and uniqueness of solutions to $(-\Delta)^{s}u-k^{2s}u=f$, which satisfy  a Sommerfeld radiation condition. We accomplish this objective in Section 3 for $f$ with compact support or subject to  decay assumptions at infinity. Specifically we study the resolvent operator
\[
   R_0(z) = (\FLs - z)^{-1}, \qquad z \in \mathbb{C}\setminus \mathbb{R}^+,
\]
employing a limiting absorption principle for real \(k\), in the spirit of the 
seminal work of Agmon~\cite{agmon1975spectral} for the Laplace operator, later extended  to the fractional Laplace operator
by Ben-Artzi, Devinatz and Nemirovski in ~\cite{ben1987limiting, ben1997remarks}.  In this approach the resolvent is extended by
$$R_0^{\pm}(z):=\lim_{\epsilon \to 0^+}R_0(z\pm i \epsilon)~,~~z \in (0,\infty)$$
where the $+$ sign corresponds to the outgoing solutions. We then show  that $R_0^{+}(z)f$, for any $0<s<1$, satisfies the Sommerfeld Radiation Condition  (\ref{def:SRC}) and is the only solution to (\ref{fractional_Helmholtz}) with this property, provided $f$ is compactly supported. In contrast for $f\in L^{2,\delta}({\mathbb R}^n)$ we prove a similar result using the Generalized Sommerfeld Radiation Condition (\ref{gen}), but for only $s\geq 1/2$. Furthermore, we derive an equivalent volume integral representation of this outgoing solution expressed as a convolution of the outgoing fundamental solution with the source term $f$.   In the first part of Section \ref{4} we show that the inhomogeneous Helmholtz problem (\ref{inhomogeneous_fractional_Helmholtz}) with the standard radiation condition (\ref{def:SRC}) is equivalent to a Lippmann--Schwinger volume  integral equation over a bounded region $D$ containing the support of $q$, a problem of Fredholm type. The kernel of the volume integral is exactly the outgoing Green's function that we previously calculated. This formulation allows us to conclude that (\ref{inhomogeneous_fractional_Helmholtz}) with the condition (\ref{def:SRC}) has a unique solution for all $k\in (0,\infty) \setminus \Lambda$, where $\Lambda$ is a discrete (possibly empty) set that can only accumulate at $\infty$.  Due to the fairly explicit form of the Green's function, and the boundedness of $D$, the Lippmann--Schwinger integral equation provides a suitable framework for the numerical solution of this non-local problem in the entire space. In the second part of Section 4 we finally extend our analysis to the solution of (\ref{inhomogeneous_fractional_Helmholtz2}), by considering the resolvent
\[
   R_k(z) = (\FLs - k^{2s}q -z)^{-1}, \qquad z \in \mathbb{C}\setminus \mathbb{R}^+,
\]
and its extension $R_k^{\pm}(z):=\lim_{\epsilon \to 0^+}R_k(z\pm i \epsilon)$. The extension is possible for every $z>0$ except for a discrete set (possibly empty) that can only accumulate  at $\infty$. These exceptional values would correspond to embedded eigenvalues of the operator $\FLs - k^{2s}q$. 

We note that for the sake of readability, some calculations and proofs of certain technical lemmas are deferred to the Appendix.
\section{The Outgoing Fundamental Solution of \texorpdfstring{$ \FLs  - k^{2s}$}{TEXT}}\label{fund}
The  outgoing  fundamental solution $G_{n,s}^{k}$ of the fractional Helmholtz operator for $k>0$ is defined as that solution of 
 \begin{equation*}
( \FLs  - k^{2s}) G_{n,s}^{k}(x) =\delta(x), \qquad  x\in {\mathbb R}^n,  \, n=1,2,3, \quad  \mbox{and}\quad s\in (0,\,1), 
\end{equation*}
which is obtained from the  limiting absorption principle, and where $\delta(\cdot)$ denotes the Dirac distribution. More specifically we carry out the calculations for complex values $k^{2s}_\epsilon:= k^{2s} + i\epsilon$ with $\epsilon>0$ by solving 
\begin{equation}\label{lap}
\left( \FLs  - (k^{2s}+i\epsilon)\right) G_{n,s}^{k_\epsilon}(x) =\delta(x)~,
\end{equation}
which is uniquely solvable. Then, for real $k>0$, the outgoing fundamental solution is obtained by taking the limit 
 $$G_{n,s}^{k}(x)= \lim_{\epsilon \to 0^+} G_{n,s}^{k_\epsilon}(x), \qquad x\in \R^n.$$
 We remark that taking $\epsilon <0$ in (\ref{lap}) and letting it go to zero will provide the incoming fundamental solution. Although the incoming fundamental solution can be obtained immediately from our calculations, throughout  this paper we will provide only the formulas for the outgoing fundamental solution. Note that throughout this paper,  we use interchangeably the terminology ``outgoing fundamental solution" and ``outgoing free space Green's function".
 
 \noindent
 \begin{remark}\label{example}
 {\em In the limiting absorption approach we perturb $k^{2s}$ by $i\epsilon$, as opposed to  perturbing $k$ by $i\epsilon$. The reason for this becomes apparent later on when we connect the extended resolvent to the convolution with the fundamental solution in Section~\ref{3}. The calculations below of the fundamental solutions are valid as long as $k_\epsilon$ lies in the first quadrant; this is clearly the case for 
 $$k_\epsilon =(k^{2s}+i\epsilon)^{1/2s}=(k^{4s}+\epsilon ^2)^{1/4s}\exp\left({\frac{i\arctan\frac{\epsilon}{k^{2s}}}{2s}}\right),$$ provided  $\arctan\frac{\epsilon}{k^{2s}}$ is in $(0, s\pi)$, which is guaranteed for $\epsilon>0$ sufficiently small.}
\end{remark}

\subsection{Computations in 1D for \texorpdfstring{$0<s<1$}{TEXT}}
For $\epsilon>0$, $k>0$ and $k^{2s}_\epsilon := k^{2s}+i\epsilon$, after taking the Fourier transform of \ref{lap},  we obtain for $x\neq 0$
$$ G_{1,s}^{k}(x) = \lim_{\epsilon \to 0^+} \frac{1}{2\pi} \int_{\R} \frac{e^{i\xi |x|}}{|\xi|^{2s}-k_\epsilon^{2s}}d\xi=\lim_{\epsilon \to 0^+}\frac{1}{2\pi} \int_0^\infty \frac{e^{i\xi |x|}}{\xi^{2s}-k_\epsilon^{2s}} d\xi+ \frac{1}{2\pi}\int_0^\infty \frac{e^{-i\xi |x|}}{\xi^{2s}-k_\epsilon^{2s}}d\xi$$
\begin{figure}
    \centering
    \includegraphics[width=0.25\linewidth]{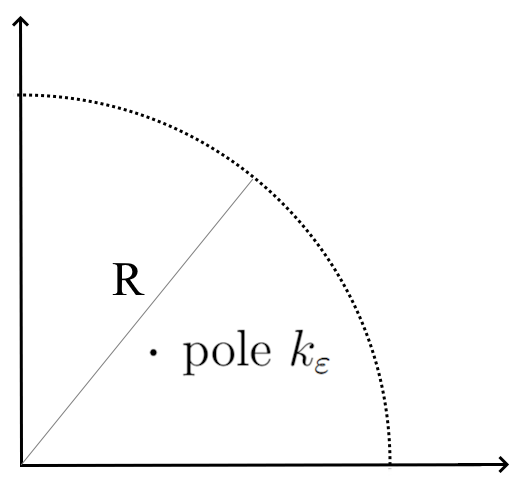}
   \caption{Contour integral in the complex plane}
\label{fig:contour}
\end{figure}
\noindent
For the first integral, we choose the contour in the upper half plane as in figure \ref{fig:contour}.  The branch cut for $\xi^{2s}$ is the negative real axis, i.e. we work with the natural logarithm. We compute each part of the contour integral :
\begin{itemize}
   \item On the quarter circle, we have $\xi = R e^{i \theta}$ for $\theta \in (0,\pi/2)$. We will prove that the integral on the quarter circle 
   \begin{align*}
       I_R :=\int_0^{\pi/2} \frac{e^{i R e^{i\theta}|x|}}{R^{2s} e^{i2s\theta}- k_\epsilon^{2s}} R e^{i\theta}id\theta
   \end{align*}
    goes to $0$ as $R\to \infty$. For $R$ large enough and using that $\sin(\theta) \geq \frac{2}{\pi} \theta$ for $\theta \in (0, \frac{\pi}{2})$, we obtain
     \begin{align*}
         |I_R| &\leq 2 R^{1-2s}\int_0^{\pi/2} e^{- R |x|\sin(\theta)} d\theta  \\
         &\leq 2 R^{1-2s}\int_0^{\pi/2} e^{- \frac{2}{\pi} R |x|\theta} d\theta \\
         &= \pi  R^{-2s} |x|^{-1}\left( 1 - e^{-R |x|} \right) \xrightarrow[ R \to \infty ]{} 0 ~. 
     \end{align*}

\item On the vertical segment $\xi = iy$ for $y \in [0, \infty)$ we have 
\begin{align*}
     \int_\infty^0 \frac{e^{i|x|iy}}{(y^{2s}e^{i\pi s}-k_\epsilon^{2s})} idy = -\int_0^\infty\frac{e^{-|x|y}}{(y^{2s}e^{i\pi s}-k_\epsilon^{2s})} idy
\end{align*}
\item Finally,  we compute the residue 
\begin{align*}
    2\pi i \lim_{z\to k_\epsilon}\frac{ e^{i|x|z}(z-k_\epsilon)}{z^{2s}-k_\epsilon^{2s}}= 2\pi i e^{i|x| k_\epsilon} \lim_{h\to 0} \frac{h}{2sk_\epsilon^{(2s-1)}h} = \frac{i\pi e^{i|x|k_\epsilon}}{s k_\epsilon^{2s-1}}.
\end{align*}
\end{itemize}
Combining the above, we conclude that the first integral is given by
\begin{align}\label{int_1_1D}
    \int_0^\infty \frac{e^{i|x|\xi}}{(\xi^{2s}-k_\epsilon^{2s})}d\xi =   \int_0^\infty\frac{e^{-|x|y}}{(y^{2s}e^{i\pi s}-k_\epsilon^{2s})} idy +  \frac{i\pi e^{i|x|k_\epsilon}}{s k_\epsilon^{2s-1}}
\end{align}
For the second integral, we follow the same procedure by taking a contour in the lower half plane. Similarly, we are left with the segment along the imaginary axis, which yields
\begin{align}\label{int_2_1D}
    \int_0^\infty \frac{e^{-i |x|\xi}}{(\xi^{2s}-k_\epsilon^{2s})}d\xi = - \int_{-\infty}^0 \frac{e^{-i|x|(iy)}}{(|y|^{2s}e^{-i\pi s}-k_\epsilon^{2s})} idy = - \int_0^{\infty} \frac{e^{-|x|y}}{(y^{2s}e^{-i\pi s}-k_\epsilon^{2s})} idy
\end{align}
Summation of the two integrals (\ref{int_1_1D}) and (\ref{int_2_1D}), and division by $2\pi$ yield
\begin{align*}
    G_{1,s}^{k}(x) &= \lim_{\epsilon\to 0}G_{1,s}^{k_\epsilon}(x)\\
    &=    \frac{i e^{i|x|k}}{2s k^{2s-1}}   +\frac{1}{2\pi}\int_0^\infty \frac{e^{-|x|y}}{(y^{2s}e^{i\pi s}-k^{2s})} idy - \frac{1}{2\pi}\int_0^{\infty} \frac{e^{-|x|y}}{(y^{2s}e^{-i\pi s}-k^{2s})} idy\\
    &= \frac{i e^{i|x|k}}{2s k^{2s-1}}   +\frac{1}{2\pi}\int_0^\infty ie^{-|x|y}\,2i\, \Im\left[\frac{1}{y^{2s}e^{i\pi s}-k^{2s}}\right]dy \\
     &= \frac{i e^{i|x|k}}{2s k^{2s-1}}   -\frac{1}{\pi}\int_0^\infty \Im\left[\frac{e^{-|x|y}}{y^{2s}e^{i\pi s}-k^{2s}}\right]dy
\end{align*}
Finally making a change of variable in the last integral we obtain the following fundamental solution in one dimension 
\begin{align}\label{Fundamental_sol_1D}
    G_{1,s}^{k}(x) =\frac{i e^{i|x|k}}{2s k^{2s-1}}   -\frac{1}{\pi |x|^{1-2s}}\int_0^\infty \Im\left[\frac{e^{-y}}{y^{2s}e^{i\pi s}-k^{2s}|x|^{2s}}\right]dy
\end{align}
We will later view this kernel as the sum of two parts, the first of which is a multiple of the fundamental solution of the Helmholtz equation and the second of which has a faster decay as $|x|\to +\infty$. We introduce the following notation for the Helmholtz part of the kernel for general $k_\epsilon^{2s}:=k^{2s}+i\epsilon$ for $\epsilon\geq 0$ 
\begin{align}\label{def:G_helm1}
    G_{1, helm}^{k_\epsilon}(x) = \frac{i e^{i|x|k_\epsilon}}{2s k_\epsilon^{2s-1}},
\end{align}
and denote by $J_1^{s,k_\epsilon}$ the integral term  
\begin{align}\label{def:J1}
    J_1^{s,k_\epsilon}(x) := \frac{i}{2\pi|x|^{1-2s}}\int_0^\infty e^{-y}\left[ \frac{1}{y^{2s}e^{i\pi s}-k_\epsilon^{2s}|x|^{2s}} - \frac{1}{y^{2s}e^{-i\pi s}-k_\epsilon^{2s}|x|^{2s}}\right] dy.
\end{align}
\subsection{Computations in 2D for \texorpdfstring{$0< s\leq 1/2$}{TEXT}}
In the two-dimensional case, we need to distinguish between the cases $0< s\leq 1/2$ and $1/2< s< 1$. We start with the more delicate case $0< s\leq 1/2$, where  a different analysis is needed, depending on whether  $\frac{1}{2s}$ is a natural number or not.
\subsubsection{The case of \texorpdfstring{$\frac{1}{2s} \notin \mathbb{N}$}{TEXT}}
For any $m\in \mathbb{N}$ we have
\begin{align}\label{identity_m}
    &\frac{|\xi|^{2sm}-k_\epsilon^{2sm}}{|\xi|^{2sm}} + \frac{k_\epsilon^{2sm}}{|\xi|^{2sm}} = 1 \nonumber\\
    \quad \implies& \frac{(|\xi|^{2s}-k_\epsilon^{2s}) \sum_{j=0}^{m-1} |\xi|^{2s(m-1-j)}k_\epsilon^{2sj}}{|\xi|^{2sm}(|\xi|^{2s}-k_\epsilon^{2s})} + \frac{k_\epsilon^{2sm}}{|\xi|^{2sm}(|\xi|^{2s}-k_\epsilon^{2s})} = \frac{1}{|\xi|^{2s}-k_\epsilon^{2s}}\nonumber\\
    \implies &\sum_{j=0}^{m-1} \frac{k_\epsilon^{2sj}}{|\xi|^{2s(j+1)}} + \frac{k_\epsilon^{2sm}}{|\xi|^{2sm}(|\xi|^{2s}-k_\epsilon^{2s})} = \frac{1}{|\xi|^{2s}-k_\epsilon^{2s}}
\end{align}
where we used Newton's formula in the second line. Therefore we obtain
\begin{align*}
  \frac{1}{|\xi|^{2s}-k_\epsilon^{2s}}&= \frac{s^{-1}k_\epsilon^{2-2s}}{|\xi|^2-k_\epsilon^2}+  \sum_{j=0}^{m-1} \frac{ k_\epsilon^{2sj}}{|\xi|^{2s(j+1)}} + \frac{k_\epsilon^{2sm}}{|\xi|^{2sm}(|\xi|^{2s}-k_\epsilon^{2s})} - \frac{s^{-1}k_\epsilon^{2-2s}}{|\xi|^2-k_\epsilon^2}   \\
  &=\frac{s^{-1}k_\epsilon^{2-2s}}{|\xi|^2-k_\epsilon^2}+  \sum_{j=0}^{m-1} \frac{k_\epsilon^{2sj}}{|\xi|^{2s(j+1)}} + \frac{k_\epsilon^{2sm}(|\xi|^2-k_\epsilon^{2}) - s^{-1}k_\epsilon^{2-2s}|\xi|^{2sm}(|\xi|^{2s}-k_\epsilon^{2s})}{|\xi|^{2sm}(|\xi|^{2s}-k_\epsilon^{2s})(|\xi|^2-k_\epsilon^2)}
\end{align*}
Define the following function corresponding to the last term in the previous formula
\begin{align}\label{Fm}
  F_m(|\xi|, k_\epsilon):=  \frac{k_\epsilon^{2sm}(|\xi|^2-k_\epsilon^{2}) - s^{-1}k_\epsilon^{2-2s}|\xi|^{2sm}(|\xi|^{2s}-k_\epsilon^{2s})}{|\xi|^{2sm}(|\xi|^{2s}-k_\epsilon^{2s})(|\xi|^2-k_\epsilon^2)}
\end{align}
We observe that the expression $F_m(|\xi|, k_\epsilon)$ has a singularity of order 2 at $k_\epsilon$ that is removable. The terms in the sum correspond to the Fourier transforms of the free space Green's functions for  $(-\Delta)^{s(j+1)}$. These Green's functions are in dimension 2 given by $\frac{c_{2,j}}{|x|^{2-2s(j+1)}}$, for exponents $2s (j+1)\leq 2sm < 2$. The coefficients $c_{2,j}$ are the constants of the Riesz Potential which  in dimension $n$ are given by 
\begin{align}\label{def:c_n_s_j}
    c_{n,j} := \frac{\Gamma(\frac{n}{2} - s(j+1))}{4^{s(j+1)}\pi^{n/2} \Gamma(s(j+1))}~.
\end{align}
see for example definition 2.3 in \cite{garofalo2017fractional}.
Therefore, taking the inverse Fourier transform, we obtain from the first term the fundamental solution of Helmholtz equation, from the second term a sum of fundamental solutions for the fractional Laplacian of various orders and from the last term an integral,  which we write in radial coordinates
\begin{align}\label{fundamental_solution_2D_not_integer}
   G_{2,s}^{k_\epsilon}(x) &= s^{-1}k^{2-2s} \frac{i}{4} H_0^{(1)}(k_\epsilon|x|)+  \sum_{j=0}^{m-1} \frac{c_{2,j} k_\epsilon^{2sj}}{|x|^{2-2s(j+1)}} + \frac{1}{4\pi^2} \int_{\R^2} e^{i\xi\cdot x}F_m(|\xi|, k_\epsilon)d\xi\nonumber\\
    &=s^{-1}k_\epsilon^{2-2s} \frac{i}{4} H_0^{(1)}(k_\epsilon|x|)+  \sum_{j=0}^{m-1} \frac{c_{2,j} k_\epsilon^{2sj}}{|x|^{2-2s(j+1)}} + \frac{1}{2\pi} \int_0^\infty J_0(r|x|)rF_m(r,k_\epsilon)dr,
\end{align}
where $J_0$ is the Bessel function of the first kind  of order zero, which is given by
\begin{align*}
    J_0(r) =\frac{1}{2\pi} \int_0^{2\pi} e^{ir \cos(\theta)} d\theta,
\end{align*}
and $H_0^{(1)}(z)$ is the Hankel function of order zero of the first kind given by $H_0^{(1)}(z):=J_0(z)+iY_0(z)$.  Note that the Bessel function $J_0$ is analytic in ${\mathbb C}$, whereas the  Neumann function $Y_0(z)$ (a Bessel function of the second kind)  has  a logarithmic singularity at $z=0$ and a branch cut at $(-\infty, 0]$. Thus, $H_0^{(1)}(z)$ is well-defined and analytic in ${\mathbb C}\setminus(-\infty, 0]$.
 
\noindent
We choose $m \in \mathbb{N}$ so that $r^{\frac12}F_m(r,k_\epsilon)$ is an $L^2$ function of $r$ which implies that the  last integral  in (\ref{fundamental_solution_2D_not_integer}) makes sense in $L^2$. Indeed, $r^{\frac12}F_m(r,k_\epsilon) \sim_{r\to 0} r^{\frac12-2sm}$, hence it is $L^2$-integrable at $r=0$ if $4sm<2$, that is 
    $$ \frac{1}{2s}>m ~.$$
Furthermore, $F_m(r,k_\epsilon) \sim_{r\to \infty} r^{-\min(2, 2s(m+1))}$. From the previous condition the minimum should always be $2s(m+1)$. Therefore $r^{\frac12}F_m(r,k_\epsilon)$ decays like $r^{\frac12-2s(m+1)}$ which is $L^2$ integrable at infinity  provided 
    $$
    2s(m+1)>1 \iff m> \frac{1}{2s} - 1 ~.
    $$
If we select $m = \lfloor \frac{1}{2s}\rfloor$, then both conditions are satisfied, except integrability at zero when $m= \frac{1}{2s}$, which is excluded by our assumption. Thus, for all $0<s \leq 1/2$ such that $\frac{1}{2s}\notin \mathbb{N}$, we obtain that the fundamental solution is given by  (\ref{fundamental_solution_2D_not_integer}) with $m = \lfloor \frac{1}{2s}\rfloor$.  For this $m (=\lfloor \frac{1}{2s}\rfloor)$,  we define
\begin{align}\label{def:J2}
    J_2^{s,k_\epsilon} (x) =  \frac{1}{2\pi} \int_0^\infty J_0(r|x|)rF_m(r,k_\epsilon)dr.
\end{align}
Now, passage to the limit $\epsilon=0^+$ yields the fundamental solution for $k>0$ 
\begin{align}\label{fundamental_solution_2D_not_integer_k>0}
G_{2,s}^{k}(x)=\lim\limits_{\epsilon\to 0^+} G_{2,s}^{k_\epsilon}(x)=s^{-1}k^{2-2s} \frac{i}{4} H_0^{(1)}(k|x|)+  \sum_{j=0}^{m-1} \frac{c_{2,j} k^{2sj}}{|x|^{2-2s(j+1)}} + J_2^{s,k} (x). 
\end{align}
When $m := \lfloor \frac{1}{2s}\rfloor = \frac{1}{2s}$, we need further modifications to take care of the singularity. 

\subsubsection{\texorpdfstring{The case of $\frac{1}{2s} = m \in \mathbb{N}$}{TEXT}}\label{s2.2.2}
It follows that is $2sm = 1$. We modify $F_m$ with a corrector function in the following way 
$$ F_m(|\xi|, k_\epsilon)= F_m(|\xi|,k_\epsilon) + \frac{k_\epsilon^{2-2s}}{|\xi|(|\xi|+k_\epsilon)}- \frac{k_\epsilon^{2-2s}}{|\xi|(|\xi|+k_\epsilon)}=\tilde{F}_m(|\xi|,k_\epsilon)- \frac{k_\epsilon^{2-2s}}{|\xi|(|\xi|+k_\epsilon)}$$
where we define 
     \begin{align*}
     \tilde{F}_m(|\xi|, k_\epsilon) &:=  F_m(|\xi|,k_\epsilon)+ \frac{k_\epsilon^{2-2s}}{|\xi|(|\xi|+k_\epsilon)} \\
     &= \frac{k_\epsilon(|\xi|^2-k_\epsilon^{2}) - s^{-1}k_\epsilon^{2-2s}|\xi|(|\xi|^{2s}-k_\epsilon^{2s}) + k_\epsilon^{2-2s}(|\xi|^{2s}-k_\epsilon^{2s})(|\xi|-k_\epsilon)}{|\xi|(|\xi|^{2s}-k_\epsilon^{2s})(|\xi|^2-k_\epsilon^2)}. 
\end{align*}
Notice that for $s=1/2$, $\tilde F_m=\tilde F_1$ equals $0$.
Now the leading singularity around zero is  'transferred' to the added term, which we simplify by  introducing the Struve function of the second kind of order zero  
\begin{align*}
K_0(z) = \frac{2}{\pi} \int_0^\infty \frac{J_0(t)}{t+z}dt 
\end{align*}
see for example \cite{watson1922treatise}.
Note that $K_0(z)$ is analytic for $z\in \mathbb{C} \setminus (-\infty,0]$. Evaluated at $z= k_\epsilon |x|$ and after a change of variable we get
\begin{align*}
    K_0(k_\epsilon|x|) = \frac{2}{\pi}\int_0^\infty \frac{J_0(|x|r)}{r+k_\epsilon} dr.
\end{align*}
Taking the inverse Fourier transform of the additional term we have
$$\frac{1}{4\pi^2}\int_{\R^2} \frac{e^{i\xi \cdot x}k_\epsilon^{2-2s}}{|\xi|(|\xi|+k_\epsilon)}d\xi = \frac{k_\epsilon^{2-2s}}{2\pi} \int_0^\infty \frac{J_0(|x|r)}{r+k_\epsilon}dr= \frac{k_\epsilon^{2-2s}}{4} K_0(k_\epsilon |x|).$$
Finally, we note that $\tilde F_m$ for $m>1$ is an $L^2$ function, by first observing that the singularity around zero is square integrable  
\begin{align*}
    \tilde F_m (|\xi|, k_\epsilon) \sim_{|\xi|\to 0} \frac{-k_\epsilon^{3-2s}|\xi|^{2s}}{k_\epsilon^{2+2s}|\xi|}  = \frac{-k_\epsilon^{1-4s}}{|\xi|^{1-2s}}~, \qquad 1-2s <1
\end{align*}
and secondly by observing that the decay at infinity is sufficiently fast 
\begin{align*}
    \tilde F_m (|\xi|, k_\epsilon) \sim_{|\xi|\to \infty} \frac{k_\epsilon|\xi|^2}{|\xi|^{1+2s+2}}  = \frac{k_\epsilon}{|\xi|^{1+2s}}~, \qquad 1+2s >1 ~. 
\end{align*}
We obtain that the fundamental solution when $\frac{1}{2s}=m \in \mathbb{N}$ is 
\begin{align}\label{fundamental_solution_2D_integer}
   G_{2,s}^{k_\epsilon}(x) 
    &=s^{-1}k_\epsilon^{2-2s} \frac{i}{4} H_0^{(1)}(k_\epsilon|x|)+  \sum_{j=0}^{m-1} \frac{c_{2,j} k_\epsilon^{2sj}}{|x|^{2-2s(j+1)}} + J_{2}^{s,k_\epsilon}(x), 
\end{align}
where  we define  $J_2^{s,k_\epsilon}$ by 
\begin{align}\label{def:J2-int}
    J_2^{s,k_\epsilon}(x)=\frac{1}{2\pi} \int_0^\infty J_0(r|x|)r\tilde F_m(r,k_\epsilon)dr  - \frac{k_\epsilon^{2-2s}}{4} K_0(k_\epsilon |x|).
\end{align}
For $k>0$ the outgoing limiting expression of the fundamental solution  as $\epsilon\to 0^+$  is 
\begin{equation}
G_{2,s}^{k}(x)=s^{-1}k^{2-2s} \frac{i}{4} H_0^{(1)}(k|x|)+  \sum_{j=0}^{m-1} \frac{c_{2,j} k^{2sj}}{|x|^{2-2s(j+1)}} + J_{2}^{s,k}(x) ~.
\end{equation}
Note that the case of $s=1/2$ corresponds to $m=1$ and our expression for the outgoing fundamental solution coincides with the one obtained in \cite{john1,umeda1995radiation}

\subsection{Computations in 2D for \texorpdfstring{$1/2<s<1$}{TEXT}}
In the case when $s>1/2$ the calculations can be carried out directly using
$$G_{2,s}^{k_\epsilon}(x) = \frac{1}{(2\pi)^2} \int_{\R^2} \frac{e^{i\xi \cdot x}}{|\xi|^{2s}-k_\epsilon^{2s}}d\xi.$$
In other words the calculations for $0<s\leq 1/2$ remain valid without the formula involving $m$, which can be viewed as taking  $m=0$ with convention that $\sum_{j=0}^{-1} = 0$. Hence we have the following expression for the fundamental solution 
\begin{align}\label{def:fundamental_sol_2D_s_3/4-ep}
    G_{2,s}^{k_\epsilon}(x):= s^{-1}k_\epsilon^{2-2s} \frac{i}{4} H_0^{(1)}(k_\epsilon|x|)+  \frac{1}{2\pi} \int_0^\infty J_0(r|x|)rF_0(r,k_\epsilon)dr
\end{align}
with the outgoing limit as $\epsilon \to 0^+$ for $k>0$ given by
\begin{align}\label{def:fundamental_sol_2D_s_3/4}
    G_{2,s}^k(x):= s^{-1}k^{2-2s} \frac{i}{4} H_0^{(1)}(k|x|)+J_{2}^{s,k}(x)~,
\end{align}
where
\begin{align}
J_{2}^{s,k}(x)=  \frac{1}{2\pi} \int_0^\infty J_0(r|x|)rF_0(r,k)dr
\end{align}
and where $F_0(r,k)$  is given by (\ref{Fm}) with $m=0$, i.e.,
\begin{align*}
  F_0(|\xi|, k):=  \frac{1}{|\xi|^{2s}-k^{2s}}- s^{-1}k^{2-2s}\frac{1}{|\xi|^2-k^2}
\end{align*}
\subsection{Computations in 3D for \texorpdfstring{$0<s\leq 1/2$}{TEXT}}
Similarly as in the two-dimensional case for $0<s\leq 1/2$, we make use of the identity (\ref{identity_m}) with $m := \lfloor \frac{1}{2s}\rfloor$. It remains to compute  
\begin{align*}
  I_\epsilon = \frac{1}{(2\pi)^3}\int_{\R^3} \frac{e^{i \xi \cdot x}}{|\xi|^{2sm}(|\xi|^{2s}-k_\epsilon^{2s})} d\xi
\end{align*}
To this end, we  rewrite $I_\epsilon$ using spherical coordinates : 
\begin{align*}
   I_\epsilon &= \frac{1}{(2\pi)^3}\int_{\R^3} \frac{e^{i\xi \cdot x}}{|\xi|^{2sm}(|\xi|^{2s}-k_\epsilon^{2s})} d\xi= \frac{1}{(2\pi)^3}\int_0^\infty \int_0^\pi \int_0^{2\pi} \frac{r^2 e^{i r |x| cos(\phi)}}{r^{2sm}(r^{2s}-k_\epsilon^{2s})} \sin(\phi) d\theta d\phi dr \\
    &= \frac{1}{(2\pi)^3} \int_0^\infty  2\pi \left[ \frac{e^{ir|x|} - e^{-ir|x|}}{ir|x|}\right] \frac{r^2}{r^{2sm}(r^{2s}-k_\epsilon^{2s})} dr \\
    &=\frac{1}{i(2\pi)^2|x|}\int_0^\infty  \frac{re^{ir|x|}}{r^{2sm}(r^{2s}-k_\epsilon^{2s})} dr- \frac{1}{i(2\pi)^2|x|}\int_0^\infty  \frac{re^{-ir|x|}}{r^{2sm}(r^{2s}-k_\epsilon^{2s})} dr.
\end{align*}
For simplicity, we let 
$$I^1_\epsilon := \int_0^\infty  \frac{te^{it|x|}}{t^{2sm}(t^{2s}-k_\epsilon^{2s})} dt \qquad \mbox{and} \qquad   I^2_\epsilon := \int_0^\infty  \frac{te^{-it|x|}}{t^{2sm}(t^{2s}-k_\epsilon^{2s})} dt.$$
To compute $I^1_\epsilon$, we use the same contour integral as in Figure  \ref{fig:contour}. In that case 
\begin{itemize}
    \item The term $\frac{1}{t^{2sm-1}}$ in the integral represents a removable singularity.
    \item On the quarter circle, we have $t = R e^{i\theta}, \theta\in (0,\pi/2)$, we will show that  
    \begin{align*}
       I_{\epsilon,R}^1 &:= \int_0^{\pi/2} \frac{e^{i|x|R e^{i\theta}} }{(R^{2s}e^{i2s\theta} - k_\epsilon^{2s})}  R^{2-2sm} e^{i (2-2sm)\theta} id\theta
    \end{align*}
    goes to $0$ as $R\to \infty$. By a similar argument as in dimension 1, we have
    \begin{align*}
        |I_{\epsilon,R}^1| &\leq 2 R^{2-2s(m+1)} \int_0^{\pi/2} e^{-R |x| \sin(\theta)}d\theta \\
        &\leq  2 R^{2-2s(m+1)} \int_0^{\pi/2} e^{-\frac{2}{\pi}R|x| \theta}d\theta  \\
        &= \pi R^{1-2s(m+1)} |x|^{-1}\left(1- e^{-R |x|} \right) \xrightarrow[R \to \infty]{} 0
    \end{align*}
    since $1< 2s(m+1)$. 
\item The vertical segment along the imaginary axis gives rise to the following integral
\begin{align}\label{vertical3D-2}
    \int_\infty^0 \frac{e^{-|x|y}}{y^{2s} e^{i\pi s} - k_\epsilon^{2s}} y^{1-2sm} e^{i\pi\frac{1-2sm}{2}} idy = \int_0^\infty \frac{e^{-|x|y}}{y^{2s} e^{i\pi s} - k_\epsilon^{2s}} y^{1-2sm} e^{-i\pi sm} dy.
\end{align}
\item Finally we need to  compute the residue at $k_\epsilon$. To this end, letting  $z = k_\epsilon+h$ for some $h\in \C$ small, we have the expansion 
$z^{2s} = (k_\epsilon + h)^{2s} = k_\epsilon^{2s} + 2s k_\epsilon^{2s-1}h + o(h)$. Therefore
\begin{align} \label{residue3D-2}
    \lim_{z\to k_\epsilon} \frac{e^{i|x|z }(z-k_\epsilon)z^{1-2sm}}{(z^{2s}-k_\epsilon^{2s})} = \lim_{h\to 0} \frac{e^{i|x|(k_\epsilon+h) }h (k_\epsilon+h)^{1-2sm}}{2s k_\epsilon^{2s-1}h} = \frac{e^{ik_\epsilon |x|}k_\epsilon^{1 - 2sm}}{2sk_\epsilon^{2s-1}} 
\end{align}
\end{itemize}
Summing up $2\pi i$ times (\ref{residue3D-2}) and (\ref{vertical3D-2}) we obtain 
\begin{align}\label{int_1_3D-2}
    I^1_\epsilon = \frac{i \pi e^{i|x|k_\epsilon}}{sk_\epsilon^{2s(m+1)-2}} -\int_0^\infty \frac{e^{-|x|y}}{y^{2s} e^{i\pi s} - k_\epsilon^{2s}} y^{1-2sm} e^{-i\pi sm} dy.
\end{align}
Note that $y^{1-2sm}$ is integrable near $0$ since $m = \lfloor\frac{1}{2s}\rfloor \le \frac{1}{2s} \implies 1-2sm\ge  0 $. 

\noindent
For the second integral $I^2_\epsilon$ we take a similar contour in the lower half plane.  The integral on the quarter circle will vanish (as $R\rightarrow \infty$) and we are left with the vertical segment along the negative imaginary axis 
\begin{align}
   I^2_\epsilon &=  - \int_{-\infty}^0 \frac{iye^{y|x|}}{(|y|^{2s}e^{-i\pi s}-k_\epsilon^{2s})|y|^{2sm}e^{-i\pi sm}} i dy =  \int_{\infty}^0 \frac{-iye^{-y|x|}}{(y^{2s}e^{-i\pi s}-k_\epsilon^{2s})y^{2sm}e^{-i\pi sm}} i dy\nonumber\\
    &=-\int_0^\infty \frac{e^{-y|x|}}{(y^{2s}e^{-i\pi s}-k_\epsilon^{2s})}y^{1-2sm} e^{i\pi sm} dy.\label{int_2_3D}
\end{align}
Finally, adding the above expressions for  $I_1^\epsilon$ and $I_2^\epsilon$ we get
\begin{align*}
    &\frac{1}{i(2\pi)^2|x|}(I_1^\epsilon -I_2^\epsilon) \\
   &\qquad  = \frac{1}{i(2\pi)^2|x|}\left(\frac{i \pi e^{i|x|k_\epsilon}}{sk_\epsilon^{2s(m+1)-2}} +\int_0^\infty e^{-|x|y} y^{1-2sm} \left[ \frac{e^{i\pi sm}}{y^{2s} e^{-i\pi s} -k_\epsilon^{2s}} - \frac{e^{-i\pi sm}}{y^{2s}e^{i\pi s}-k_\epsilon^{2s}}\right]dy\right)\\
     &\qquad = \frac{ e^{i|x|k_\epsilon}}{4\pi sk_\epsilon^{2s(m+1)-2}|x|} + \frac{1}{i(2\pi)^2|x|}\int_0^\infty e^{-|x|y} y^{1-2sm} \left[ \frac{e^{i\pi sm}}{y^{2s} e^{-i\pi s} -k_\epsilon^{2s}} - \frac{e^{-i\pi sm}}{y^{2s}e^{i\pi s}-k_\epsilon^{2s}}\right]dy
\end{align*}
We see from this calculation that the fundamental solution for $k_\epsilon = k+ i\epsilon$ is given by 
\begin{align}
     G_{3,s}^{k_\epsilon}(x) =&\sum_{j=0}^{m-1}\frac{c_{3,j}k_\epsilon^{2sj}}{|x|^{3-2s(j+1)}} +\frac{k_\epsilon^{2-2s}}{s} \frac{ e^{i|x|k_\epsilon}}{4\pi |x|} \nonumber\\
     &+ \frac{k_\epsilon^{2sm}}{4i\pi^2 |x|}\int_0^\infty e^{-|x|y} y^{1-2sm} \left[ \frac{e^{i\pi sm}}{y^{2s} e^{-i\pi s} -k_\epsilon^{2s}} - \frac{e^{-i\pi sm}}{y^{2s} e^{i\pi s} -k_\epsilon^{2s}}\right]dy, \label{green}
\end{align}
with $m= \lfloor \frac{1}{2s}\rfloor$, and $c_{3,j}$ given by (\ref{def:c_n_s_j}). As before we can decompose this into three  parts 
\begin{align}\label{def:G3helm_J3}
G_{3,s}^{k_\epsilon}(x)= G_{3,helm}^{k_\epsilon}(x) + \sum_{j=0}^{m-1}\frac{c_{3,j}k_\epsilon^{2sj}}{|x|^{3-2s(j+1)}} + J_{3}^{s,k_\epsilon}(x)
\end{align}
where 
\begin{align*}
    G_{3,helm}^{k_\epsilon}(x) &:=\frac{k_\epsilon^{2-2s}}{s} \frac{ e^{i|x|k_\epsilon}}{4\pi |x|} \qquad \mbox{ and }
 \end{align*}
 \begin{align}   
 &J_{3}^{s,k_\epsilon}(x):=  \frac{k_\epsilon^{2sm}}{4i\pi^2 |x|}\int_0^\infty e^{-|x|y} y^{1-2sm} \left[ \frac{e^{i\pi sm}}{y^{2s} e^{-i\pi s} -k_\epsilon^{2s}} - \frac{e^{-i\pi sm}}{y^{2s} e^{i\pi s} -k_\epsilon^{2s}}\right]dy\nonumber\\
    &=\frac{k_\epsilon^{2sm}}{4i\pi^2 |x|^{3-2s(m+1)}}\int_0^\infty e^{-y} y^{1-2sm} \left[ \frac{e^{i\pi sm}}{y^{2s} e^{-i\pi s} -k_\epsilon^{2s}|x|^{2s}} - \frac{e^{-i\pi sm}}{y^{2s} e^{i\pi s} -k_\epsilon^{2s}|x|^{2s}}\right]dy. \label{def:J3-l1}
\end{align}
Note that we performed a change of variable in the last integral in (\ref{green}) to see the dependence in $|x|$  more clearly. Taking the limit as $\epsilon \to 0^+$, we obtain the outgoing fundamental solution for real $k>0$ 
\begin{align}\label{3D}
     G_{3,s}^k(x) =&\sum_{j=0}^{m-1}\frac{c_{3,j}k^{2sj}}{|x|^{3-2s(j+1)}} +\frac{k^{2-2s}}{s} \frac{ e^{i|x|k}}{4\pi |x|} +J_{3}^{s,k}(x), \qquad m=\left\lfloor \frac{1}{2s}\right\rfloor.
\end{align}
\subsection{Computations in 3D for \texorpdfstring{$1/2<s<1$}{TEXT}}
When $s>1/2$, just as in the two dimensional case, we do not need to introduce $m = \lfloor \frac{1}{2s}\rfloor$ and the term involving a sum. The contour integral stays the same, and the final expression corresponds to formally  taking $m = 0$ (with the convention that $\sum_0^{-1}=0$)  in the result from the previous section. We have
\begin{align*}\label{3D_s_greater_than_half}
     G_{3,s}^{k_\epsilon}(x) &= \frac{1}{sk_\epsilon^{2s-2}} \frac{ e^{i|x|k_\epsilon}}{4\pi |x|} + \frac{1}{4\pi^2 |x|^{3-2s}}\int_0^\infty e^{-y} y\, \left[ \frac{1}{y^{2s} e^{-i\pi s} -k_\epsilon^{2s}|x|^{2s}} -\frac{1}{y^{2s} e^{i\pi s} -k_\epsilon^{2s}|x|^{2s}}\right]dy.
\end{align*}
Observe that we can define $J_3^{s,k_\epsilon}$ as in (\ref{def:J3-l1}) by taking $m=0$. The fundamental solution for real $k>0$ is then given by
\begin{align}
G_{3,s}^{k}(x) &= \frac{1}{sk^{2s-2}} \frac{ e^{i|x|k}}{4\pi |x|} + J_3^{s,k}(x).
\end{align}
\subsection{Asymptotics at infinity of the outgoing fundamental solution}
We summarize in the table below the expressions of  the outgoing  fundamental solution  $G_{n,s}^{k}$ for $k>0$ computed in the  above sections.
 \begin{table}[h]
    \centering
    \begin{tabular}{Sc|Sc|Sc}
         &$0<s\leq 1/2$  &$1/2<s<1$\\
         \hline
         $n=1$& \multicolumn{2}{c}{$G_{1,helm}^k + J_1^{s,k}$}\\
         \hline
         $n=2$& {$G_{2,helm}^k + J_2^{s,k} + \sum_{j=0}^{\lfloor \frac{1}{2s}\rfloor-1}\frac{c_{2,j}k^{2sj}}{|x|^{2-2s(j+1)}}$} &  $G_{2,helm}^k + J_2^{s,k} $\\
         \hline
         $n=3$& $G_{3,helm}^k + J_3^{s,k} + \sum_{j=0}^{\lfloor \frac{1}{2s}\rfloor-1}\frac{c_{3,j}k^{2sj}}{|x|^{3-2s(j+1)}}$ & $G_{3,helm}^k + J_3^{s,k}$
    \end{tabular}
    \caption{Fundamental Solution of the Fractional Helmholtz Equation}
    \label{tab:fund_sol}
\end{table}

\noindent
In this table, $G_{n,helm}^k$, $n=1,2,3$,  are functions of $|x|$  given by
$$G_{1,helm}^k(|x|):=\frac{i}{2sk^{2s-1}}e^{i|x|k},\; G_{2,helm}^k(|x|):=\frac{i}{4sk^{2s-2}}H_0^{(1)}(k|x|), \;G_{3,helm}^k(|x|):=\frac{1}{sk^{2s-2}}\frac{e^{i|x|k}}{4\pi|x|}$$
which up to the $s$ and $k$-dependent constant,  coincide with the fundamental solution  of the Helmholtz equation.  Furthermore,  $J_1^{s,k}$ is given by (\ref{def:J1}). $J_2^{s,k}$ is given by (\ref{def:J2}) with $m=\lfloor \frac{1}{2s}\rfloor$ when $0<s\leq 1/2$ and $\frac{1}{2s}\notin{\mathbb N}$ , and given by (\ref{def:J2}) with m=0 when $1/2<s<1$. When $0<s\leq 1/2$ and $\frac{1}{2s}\in{\mathbb N}$ then $J_2^{s,k}$ is given by (\ref{def:J2-int}). Finally $J_3^{s,k}$ is given by (\ref{def:J3-l1}) with  $m=\lfloor \frac{1}{2s}\rfloor$ when $0<s\leq 1/2$  and with $m=0$ when  $1/2<s<1$.  Separation of the terms $|x|^{-n+2s(j+1)}$ from $J_n^{s,k}$ for $0<s\le 1/2$, $n=2,3$, is convenient for the analysis in the appendix \ref{appendixB}.

\noindent
From classical scattering theory we know that  the behavior at infinity of the Helmholtz part, $G_{n,helm}^k(|x|)$, is such  that
$$ |x|^{\frac{n-1}{2}}\left( \frac{\partial u}{\partial |x|} -ik u \right) \xrightarrow[|x|\to \infty]{}0, \qquad \mbox{uniformly in } \hat x=x/|x|.$$
\noindent The above behavior occurs due to cancelation, and the minus sign distinguishes the outgoing solution. All the additional terms  that appear in the  fundamental solution $G_{n,s}^{k}$ of the fractional Helmholtz along with their derivatives decay faster than $1/|x|^{\frac{n-1}{2}}$ as $|x|\to \infty$.  To see this, we refer the reader to  Appendix \ref{appendixB} which presents the asymptotic behavior of the fundamental solution $G_{n,s}^{k}$ at infinity, as well as at zero. Summarizing we have proven:
\begin{theorem}\label{som}
The outgoing fundamental solution $G_{n,s}^{k}$ of the fractional Helmholtz operator $(-\Delta)^s - k^{2s}$ for $k>0$ and $0<s<1$ in $\mathbb{R}^n$, $n=1,2,3$, behaves asymptotically as $|x|\to \infty$ like the outgoing fundamental solution of the Helmholtz operator $\Delta + k^2$. In particular, it satisfies the Sommerfeld radiation condition  (\ref{def:SRC}).
\end{theorem}
\noindent
Note that  $G_{n,s}^{k}$ also satisfies (\ref{gen}). Indeed, from the equivalence proof  in Appendix \ref{equiv} (cf. Remark \ref{equivrem} in Section \ref{uniqsec}) the Helmholtz part satisfies this, and the remaining terms also do since they decay sufficiently fast. 
\section{Extended Resolvent of the Fractional Laplacian and the Radiation Condition at Infinity}\label{3}
We start by introducing some notation that will be used throughout the paper.  We define the following weighted Sobolev spaces for $\delta  \in \R$ and $s\in (0,1)$: 
\begin{align}
    L^{2,\delta}(\R^n) &:= \left\{ f : \R^n \to \R \mbox{ such that } \|f\|_{L^{2,\delta}} < \infty \right\}\\
    H^{2s, \delta}(\R^n) &:=  \left\{ f : \R^n \to \R \mbox{ such that } (\|f\|_{L^{2,\delta}}^2+\| \FLs f\|_{L^{2,\delta}}^2)^{1/2} < \infty \right\},
\end{align}
equipped with their respective weighted norm
\begin{align*}
    \|f\|_{L^{2,\delta}(\R^n)} &:=\left( \int_{\R^n} (1+|x|^2)^{\delta} |f(x)|^2 dx\right)^{1/2}\\
    \|f\|_{H^{2s,\delta}(\R^n)} &:=(\|f\|_{L^{2,\delta}}^2+\| \FLs f\|_{L^{2,\delta}}^2)^{1/2} ~. 
\end{align*}
Note that $L^{2,\delta}(\R^n)$ is the dual of $L^{2,-\delta}(\R^n)$. The upper-half complex plane is denoted by $\C^+ := \{z \in \C | \; \Im(z)>0\}$.  We use the Japanese bracket  notation $\left<x\right>^{\alpha}:= (1+|x|^2)^{\alpha/2}$, $\alpha \in \R$. 
\subsection{Properties of the resolvent in weighted Sobolev spaces}
Consider the self-adjoint operator ${\mathcal H}_0:=\FLs$ with domain $H^{2s}(\R^n)$ and denote its resolvent operator by $R_0(z)$ : 
\begin{align*}
    R_0(z) := (\FLs - z)^{-1}
    \end{align*}
which is defined for all $z\in \C \setminus \R^+$ since ${\mathcal H}_0$ has continuous spectrum $\R^+ := [0, \infty)$.  We first notice  that $x\mapsto x^s$ is a $C^2(0,\infty)$ function that satisfies Assumption 1.1 in \cite{ben1997remarks} (i.e. continuously differentiable with a positive locally H{\"o}lder continuous derivative). Therefore Theorem 2A and Corollary 2.1 in \cite{ben1997remarks} provide  a limiting absorption principle for the resolvent of ${\mathcal H}_0$ which we state in Theorem \ref{reso} for the reader's convenience. One has
$$H^{2s,-\delta}(\R^n) = \overline{D({\mathcal H}_0)}^{\|\cdot\|_{L^{2,-\delta}} + \|{\mathcal H}_0 \cdot\|_{L^{2,-\delta}}} = \{ f\in L^{2,-\delta}(\R^n),\;  \FLs f \in L^{2,-\delta}(\R^n)\}$$ 
which is the completion of  the domain of definition $D({\mathcal H}_0)$ of ${\mathcal H}_0$ with respect to its graph norm in $L^{2,-\delta}$, i.e.  $\|\cdot\|_{L^{2,-\delta}} + \|{\mathcal H}_0 \cdot\|_{L^{2,-\delta}}$.
\begin{theorem}[Ben-Artzi, Nemirovski \cite{ben1997remarks}]\label{th:ben-artzi,Nemirovski}\label{reso}
    For every $\delta>1/2$, the operator valued function $$
    z \longrightarrow R_0(z), \quad z\in \C^+$$
    can be extended continuously to $(0, \infty)$ in the uniform operator topology of ${\mathcal B}(L^{2,\delta}, H^{2s,-\delta})$ and the limiting operator valued function $R_0^+(\lambda) = \lim_{\epsilon \to 0^+} R_0(\lambda + i \epsilon)$ is locally H{\"o}lder  continuous in ${\mathcal B}(L^{2,\delta}, H^{2s,-\delta})$.
\end{theorem}
\noindent
In our notation $\lambda:=k^{2s}$ for $k>0$. Next we prove that the limiting absorption principle is consistent with the differential equation, that is, the  resolvent will satisfy the fractional Helmholtz equation. 
\begin{lemma}
    Let $\delta>1/2$, $k\in (0,\infty)$ and let $f\in L^{2,\delta}(\R^n)$. Then $u_0^+:= R_0^+(k^{2s})f\in H^{2s,-\delta}({\mathbb R}^n)$ defined in Theorem \ref{th:ben-artzi,Nemirovski} satisfies 
    \begin{align*}
        \FLs u_0^+ - k^{2s} u_0^+ = f \;\; \mbox{ in } L^{2,-\delta}(\R^n).
    \end{align*}
\end{lemma}
\begin{proof}
We know that $u_\epsilon := R_0^+(k^{2s} + i\epsilon)f$  as $\epsilon\to 0^+$ approaches  $u_0^+:= R_0^+(k^{2s})f$ in $H^{2s,-\delta}$ for any $\delta>1/2$. Furthermore $u_\epsilon \in H^{2s}(\R^n)$ for all $\epsilon >0$ and satisfies the fractional Helmholtz equation in the $L^{2,-\delta}$ sense
\begin{align*}
    \left< (-\Delta)^su_\epsilon -(k^{2s}+i\epsilon)u_\epsilon, \phi \right>_{L^{2,-\delta}, L^{2,\delta}} = \left<f,\phi\right>_{L^{2,-\delta}, L^{2,\delta}}.
\end{align*}
Passing to the limit in the above equation  we obtain
\begin{align*}
&\left|\left<f,\phi\right>_{L^{2,-\delta}, L^{2,\delta}}- \left< (-\Delta)^su_0^+ -k^{2s}u_0^+, \phi \right>_{L^{2,-\delta}, L^{2,\delta}}\right| \\
   &=\left|\left< (-\Delta)^s(u_\epsilon - u_0^+) -k^{2s}(u_\epsilon - u_0^+) +i\epsilon u_\epsilon, \phi \right>_{L^{2,-\delta}, L^{2,\delta}}\right| \\
   &\leq (\| \FLs (u_\epsilon - u_0^+)\|_{L^{2,-\delta}} + |k|^{2s} \|u_\epsilon - u_0^+\|_{L^{2,-\delta}} + \epsilon \|u_\epsilon\|_{L^{2,-\delta}} )\|\phi\|_{L^{2,\delta}} \xrightarrow[\epsilon \to 0]{} 0. 
\end{align*}
\end{proof}     
\noindent
Next our goal is to show that the solution provided by the resolvent operator indeed satisfies  the generalized Sommerfeld radiation condition. To this end we need the following technical lemma.
\begin{lemma}\label{lemma:decomposition_resolvent}
 Let  $s\in (0,1)$ and $k>0$. Then the resolvent admits a decomposition 
    \begin{align}\label{formula:resolvent_decomposution}
        R_0^+(k^{2s}) = \Gamma_0^+(k^2)(k^{2-2s}I + A(k)) + B_1(k) + B_2(k),
    \end{align}
 where $\Gamma_0^+(k^2) := (-\Delta -k^2)^{-1} \in {\mathcal B}(L^{2,\delta}, H^{2,-\delta})$ for $1/2<\delta $ is obtained by the limiting absorption principle,  $A(k)\in {\mathcal  B}(L^{2,\delta}, L^{2,\delta})$, $B_1(k)\in {\mathcal B}(L^{2,\delta}, H^{2s,\delta})$ for any $\delta\in{\mathbb R}$, and $B_2(k) \in {\mathcal B}(L^{2,\delta}, H^{2})$  for any $\delta\geq 0$. 
\end{lemma}
 \begin{proof} We adapt the argument used by Umeda in \cite{umeda1995radiation} for the case $s=1/2$. However, the factorization of the resolvent of the negative Laplace operator as a product of resolvents of the square root of the negative Laplace operator used in \cite{umeda1995radiation}  is not available in the general case $0 < s < 1$, and additional terms must be analyzed. 
 \vspace*{-0.3cm}
    \begin{figure}[hh]
\begin{center}
\begin{tabular}{ccc}
\hspace*{-0.7cm}\includegraphics[width=5cm,height=3cm]{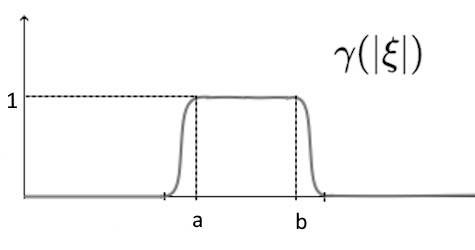} &
\hspace*{-0.6cm}\includegraphics[width=5cm,height=3cm]{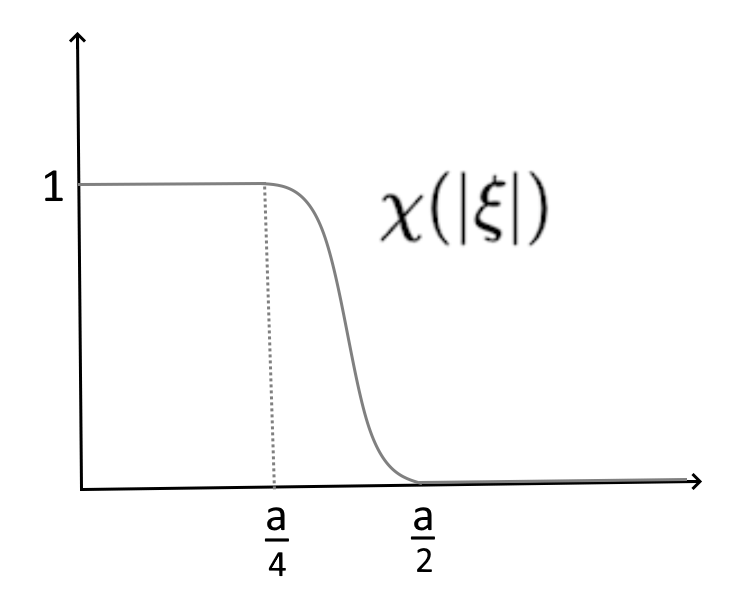} &
\includegraphics[width=5cm,height=4cm]{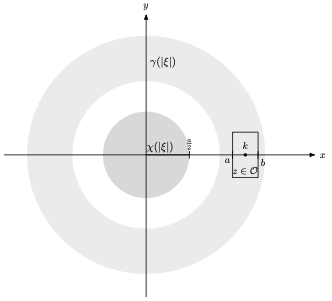} 
\end{tabular}
\end{center}
\vspace*{-0.6cm}
\caption{\small Sketch  of the profile of the  cut-off functions and the estimate regions of ${\mathbb C}$}
\label{cut_off_chi}
\end{figure}   
 We first note that for $z \in \C^+$ 
\begin{align}\label{identity1}
        |\xi|^2 -z^2 = (|\xi|^{2s}-z^{2s}) \left( |\xi|^{2-2s} + z^{2-2s} + \frac{z^{2s}|\xi|^{2-2s}-z^{2-2s}|\xi|^{2s}}{|\xi|^{2s}-z^{2s}}\right).
        \end{align}
      \noindent  
We denote the last term in this factorization, referred to as the Fourier multiplier, by
    \begin{align}\label{Mxi_def}
        M(\xi;z):= \frac{z^{2s}|\xi|^{2-2s}-z^{2-2s}|\xi|^{2s}}{|\xi|^{2s}-z^{2s}}.
        \end{align}
For the reader's convenience, the properties of $M(\xi;z)$ are collected in Appendix \ref{appendixA}, Lemma \ref{boundM}, where it is shown in particular that the singularity at $z$ is removable to order one. For a fixed $k$ take  $a,b \in \R^+$ such that  $k \in (a,b)$. We consider two radial cut-off $C_c^\infty$ functions, namely  $0\le \gamma(|\xi|)$ with compact support taking the value $1$ in $(a,b)$,  and  $0\le \chi(|\xi|)$ such that $\chi(|\xi|) = 0$ outside $B_{a/2}(0)$ and $\chi(|\xi|)=1$
on $B_{a/4}(0)$. Furthermore, we consider a rectangular region ${\mathcal O}\subset{\mathbb C}$ symmetric in the real axis, containing $k$, and included in the region where $\gamma(|\xi|)=1$ (see Figure  \ref{cut_off_chi}). In the following calculations we keep $z\in {\mathcal O} $. From (\ref{identity1}) we obtain the following decomposition of the resolvent of the fractional Helmholtz operator in the Fourier domain
        \begin{align*}
         \frac{1}{|\xi|^{2s}-z^{2s}} &= \frac{1}{|\xi|^2 - z^2} \left( z^{2-2s} + \underbrace{\gamma(|\xi|) |\xi|^{2-2s} + \gamma(|\xi|)M(\xi;z) }_{A(\xi;z)}\right)\\
         & +  \underbrace{\frac{(1- \gamma(|\xi|))|\xi|^{2-2s} + (1- \gamma(|\xi|))M(\xi;z)}{|\xi|^2 - z^2} }_{B(\xi;z)}.
    \end{align*}
  \noindent
We define the following operators 
    \begin{align*}
        A(z) &:= \F^{-1} A(\xi;z) \F := \F^{-1}\left(\gamma(|\xi|) |\xi|^{2-2s} + \gamma(|\xi|)M(\xi;z)\right)\F\\
        B(z) &:= \F^{-1} B(\xi;z) \F := \F^{-1} \frac{(1- \gamma(|\xi|))|\xi|^{2-2s} + (1- \gamma(|\xi|))M(\xi;z)}{|\xi|^2 - z^2} \F.
    \end{align*}
Note that since $z\in {\mathcal O}$, which is a compact  set in ${\mathbb C}$, the following estimates on the symbols of  all the pseudo-differential operators are uniform with respect to $z$. Since  $A(\xi;z)$ is a smooth compactly supported symbol, according to  \cite[Lemma 3.1]{umeda2003generalized}, (see also {\cite[Chapter~18.1]{HormanderIII}}), it defines a linear operator valued function  $A(z)$ which is  continuous on $z\in  {\mathcal O}$  with values in ${\mathcal B}(L^{2,\delta}, L^{2, \delta})$ for any $\delta\in {\mathbb R}$.  For the reader's convenience we formulate this result in Lemma \ref{lemma_31_Umeda}  in Appendix \ref{appendixA}. Note that  continuity in $z$ follows from the continuous dependence of the norm of the operator given by (\ref{norm}) (see the calculations for  $M(\xi;z)$  in Appendix \ref{appendixA}).

\noindent
 Next we decompose  the symbol $B(\xi;z)$ 
    \begin{align*}
         B(\xi;z)  &=\underbrace{\frac{(1-\gamma(|\xi|))}{|\xi|^2-z^2}(1-\chi(|\xi|))(|\xi|^{2-2s}+M(\xi;z))}_{B_1(\xi;z)} + \underbrace{\frac{(1-\gamma(|\xi|))}{|\xi|^2-z^2}}_{B_{2,1(\xi;z)}}\underbrace{\chi(|\xi|)(|\xi|^{2-2s}+M(\xi;z))}_{B_{2,2}(\xi;z)}
    \end{align*}
and investigate all three operators 
    \begin{align*}
        B_1(z):= \F^{-1} B_1(\xi;z) \F, \qquad   B_{2,1}(z):= \F^{-1} B_{2,1}(\xi;z) \F, \qquad  B_{2,2}(z):= \F^{-1} B_{2,2}(\xi;z) \F,
    \end{align*}
by analyzing their symbols. To this end  $B_1(\xi;z)$ is a smooth symbol, more specifically  it has no singularity at $z$ because of the term $(1-\gamma)$ and it has no differentiability issues at $0$ because of $(1-\chi)$. Moreover, when differentiating  $B_1(\xi;z)$  the decay at infinity is not affected by the term containing the  derivatives of the compactly supported cut-off functions. Therefore the rate of decay of the derivatives as $|\xi|$ tends to  $\infty$ follows from the following calculations
    \begin{align*}
       \left| \left(\frac{\partial}{\partial \xi} \right)^{\alpha} B_1(\xi;z) \right|&\leq C (1-\gamma(|\xi|))(1-\chi(|\xi|))\sum_{\beta \leq \alpha} \left|\frac{\partial^\beta}{\partial\xi^\beta} (|\xi|^2-z^2)^{-1}\right| \left| \frac{\partial^{\alpha-\beta}}{\partial\xi^{\alpha-\beta}} (|\xi|^{2-2s}+M(\xi;z))\right|\\
        &\leq  C \sum_{\beta \leq \alpha} \left<\xi\right>^{-2-|\beta|} \left<\xi\right>^{2-2s-|\alpha|+|\beta|}\leq C \left<\xi\right>^{-2s-|\alpha|}
    \end{align*}
for $|\xi|$ sufficiently large. The constant $C$ is also independent of $z\in {\mathcal O}$. We deduce that for a fixed $z \in {\mathcal O}$, $B_1(z) \in {\mathcal B}(L^{2,\delta}, H^{2s,\delta})$ for any $\delta\in {\mathbb R}$ by  Remark \ref{hoer}  in Appendix \ref{appendixA},  and in particular $B_1(z) \in  {\mathcal B}(L^{2,\delta}, H^{1,\delta})$  for $1/2\leq s<1$. Furthermore, $B_1(z)$ is a continuous operator-valued function of $z \in {\mathcal O}$. Similar analysis applied  to $B_{2,1}(\xi;z)$ yields that $B_{2,1}(z) \in {\mathcal B}(L^{2,\delta}, H^{2,\delta})$  for any $\delta\in {\mathbb R}$, and all $0<s<1$,  with continuity in $z\in {\mathcal O}$.

\noindent
Finally, using Plancherel theorem directly on $B_{2,2}(z):=\F^{-1} B_{2,2}(\xi;z) \F$ we have that
$$\|B_{2,2}(z)\phi\|_{L^{2,-\delta}}\leq \|B_{2,2}(z)\phi\|_{L^2}\leq C(z)\|\phi\|_{L^2}\leq C(z)\|\phi\|_{L^{2,\delta}}, \qquad \delta\geq 0,$$
where $C(z)$ depends continuously  on  $z\in {\mathcal O}$. Thus, we have that $B_{2,2}(z)\in {\mathcal B}(L^{2}, L^2)$ (and hence $B_{2,2}(z)\in {\mathcal B}(L^{2,\delta}, L^{2,-\delta})$  for $\delta \geq 0$). Combining this with   $B_{2,1}(z)$ we obtain that $B_2(z)= (B_{2,1} \circ B_{2,2})(z)$ is a continuous function of $z\in {\mathcal O}$ with operator-values in ${\mathcal B}(L^{2}, H^{2})$ (and hence in ${\mathcal B}(L^{2,\delta}, H^{2})$ for any $\delta\geq 0$). 

\noindent
Combining everything, we  obtain 
    \begin{align}\label{formula:resolvent_decomposition_complex}
        R_0(z^{2s}) = \Gamma_0(z^2) (z^{2-2s}I+A(z)) + B_1(z) + B_2(z).
    \end{align}
 \noindent 
 The first term is the resolvent of the Helmholtz operator studied by Agmon in \cite{agmon1975spectral}, where the limiting absorption principle justifies extending the resolvent $\Gamma_0(z)$  to $(0,\infty)$  by  $\Gamma_0^+\in {\mathcal B}(L^{2,\delta}, H^{2,-\delta})$ for $\delta>1/2$. Thus, passing to the limit in (\ref{formula:resolvent_decomposition_complex}) with $z^{2s} = k^{2s} + i \epsilon$  as $\epsilon \to 0^+$  and noting that  the operators $A,B_1,B_2$ are  continuous in $z\in \mathcal O$, we obtain the desired result. 
\end{proof}

\begin{definition}[GSRC] We say that $u\in L^{2,-\delta}({\mathbb R}^n)\cap H^1_{loc}({\mathbb R}^n)$ satisfies the Generalized Sommerfeld Radiation Condition for $k>0$ if
\begin{equation}\label{def:generalized_SRC-I}
\int\limits_{{\mathbb R}^n}  \frac{|\nabla u -ik\hat x u|^2}{(1+|x|^2)^{1-\delta}}\,dx <+\infty, \qquad  \mbox{ for some $1/2<\delta<1$}.
\end{equation}
\end{definition}
\noindent
The above decomposition of the resolvent enables us to prove that the resolvent of the fractional Helmholtz operator, satisfies the generalized Sommerfeld radiation condition, provided $s\geq 1/2$. 

\begin{lemma}\label{lemma:Resolvent_SRC_s_geq_half} Let $1/2<\delta<1$ and $s\geq 1/2$. Then,  for $f\in L^{2,\delta}(\R^n)$  the solution of $\FLs u - k^{2s} u = f$ given by the extended resolvent $u_0^+:= R_0^+(k^{2s})f\in H^{2s,-\delta}(\R^n)$ satisfies the generalized Sommerfeld radiation condition (\ref{def:generalized_SRC-I}).
 \end{lemma}
\begin{proof}
We use the decomposition (\ref{formula:resolvent_decomposution}) and the fact that  $f\in L^{2,\delta}(\R^n)$.   As a consequence of \cite[Theorem 6.1]{umeda2003generalized}, a result  first due to Ikebe-Saito  in \cite{MR312066} (see also\cite{agmon1975spectral}), the first term in (\ref{formula:resolvent_decomposution})
satisfies the generalized Sommerfeld Radiation Condition 
    \begin{align*}
       \left(\frac{\partial}{\partial x_j}-ik \hat x_j\right) \Gamma_0^+(k) (z^{2-2s} I+ A(k))f  \in L^{2,\delta-1} \quad \mbox{ for all } 1\leq j\leq n ~.
    \end{align*}
Here we use that $F:=(k^{2-2s} I+ A(k))f \in   L^{2,\delta}(\R^n)$  and  $ U_0=\Gamma_0^+(k) F$ is an outgoing solution of $-\Delta U-k^2U=F$. Concerning  the other terms, we observe that 
    \begin{align*}
         \left(\frac{\partial}{\partial x_j}-ik \hat x_j\right) B_1(k)f  \in L^{2,\delta} \subset L^{2,\delta-1} \quad \mbox{ for all } 1\leq j\leq n\\
         \left(\frac{\partial}{\partial x_j}-ik \hat x_j\right) B_2(k)f  \in H^{1} \subset L^{2,\delta-1} \quad \mbox{ for all } 1\leq j\leq n
    \end{align*}
where we used the inclusions $L^{2,\delta} \subset L^2 \subset L^{2,\delta-1}$ for $\frac{1}{2}<\delta<1$.
\end{proof}
\noindent
Note that for $0<s<1/2$ the solution $u^+_0$ is only in $H^{2s,-\delta}(\R^n)$,  hence we cannot directly conclude that its gradient is an $L^2$-function.
\subsection{Uniqueness given by the generalized Sommerfeld radiation condition}
\label{uniqsec}
We now want to show that the generalized Sommerfeld Radiation Condition uniquely  characterizes the outgoing solutions to the fractional Helmholtz equation. We know this to be true for solutions to the Helmholtz equation thanks to the results of Ikebe and Saito in \cite{MR312066}, therefore one approach is to show that if $u\in H^{2s,-\delta}$ for $1/2<\delta<1$  is a solution to the homogeneous fractional Helmholtz equation, then $u$ is a solution to the homogeneous Helmholtz equation. This is our approach in what follows. 

\noindent
The classical Schwarz space is too small to be stable under the action of $\FLs$. Therefore we define the larger space $S_\alpha$, $\alpha \in \R$ to be  
\begin{align*}
    \mathcal{S}_\alpha(\R^n) :=\left\{ \phi \in C^\infty(\R^n) | \sup_{x\in \R^n} (1+|x|^\alpha)|\phi^{(\beta)}(x)|<\infty \mbox{ for all multi-index } \beta \right\} 
\end{align*}
Note that $\mathcal{S}_\alpha (\R^n)\subset \mathcal{S}_{\alpha'}(\R^n)$ for all $0\le \alpha' \le \alpha$.
The action of the fractional Laplacian on this  Schwarz space  can be understood by the following lemma.
\begin{lemma}\label{spaces}
    Let $\phi \in \mathcal{S}(\R^n)$.  Then $\FLs \phi \in \mathcal{S}_{n+2s}(\R^n)$ for $0<s<1$. 
\end{lemma}
\begin{proof}
  A proof can be found in the paper of Bucur \cite{bucur2015some}. We include it here for the reader's convenience. Clearly $\sup_{|x|\leq 1} (1+|x|^{n+2s}) |\FLs \phi(x) |<\infty$. Next we study the decay   $\FLs\phi$ as $|x|\to +\infty$.   To this end
    \begin{align*}
       | \FLs \phi (x) | &= \left|\frac{1}{2}\int_{\R^n} \frac{2\phi(x) - \phi(x-y) -\phi(x+y)}{|y|^{n+2s}} dy\right| \\
       &\leq  \frac{1}{2} \int_{B_{|x|/2}(0)}\frac{|2\phi(x) - \phi(x-y) -\phi(x+y)|}{|y|^{n+2s}} dy \\
       &+ \frac{1}{2} \int_{\R^n \setminus B_{|x|/2}(0)} \frac{|\phi(x) -\phi(x-y)|}{|y|^{n+2s}} dy + \frac{1}{2} \int_{\R^n \setminus B_{|x|/2}(0)} \frac{|\phi(x) -\phi(x+y)|}{|y|^{n+2s}} dy 
    \end{align*}
    In the first integral, we use $ \phi(x-y) = \phi(x) - y \cdot \nabla \phi (x) + y^T\cdot R_2 \phi (x,y) \cdot y$, where 
    \begin{align*}
        R_2\phi(x,y) = \int_0^1 (1-t) D^2\phi(x-ty)dt
    \end{align*}
    and a similar formula for $\phi(x+y)$. Hence we  obtain
    \begin{align*}
        &\frac{1}{2} \int\limits_{B_{|x|/2}(0)}\frac{|2\phi(x) - \phi(x-y) -\phi(x+y)|}{|y|^{n+2s}} dy \leq C(n) \int\limits_{B_{|x|/2}(0)}\frac{\sup_{z \in B_{|x|/2}(x)}|D^2\phi(z)||y|^2}{|y|^{n+2s}} dy\\
        &\qquad \leq  C(n) \sup_{z \in B_{|x|/2}(x)}|D^2\phi(z)| \int_0^{|x|/2}r^{1-2s}dr \leq  C(n)\left|D^2\phi\left(x +t\frac{|x|}{2}e_\theta\right)\right| |x|^{2-2s}\\
        &\qquad \leq  C(n)\sup_{z\in \R^n}(1+|z|^{n+2})|D^2\phi(z)| |x|^{-n-2s},    
    \end{align*}
where we use that the supremum is attained at $x+t\frac{|x|}{2}e_\theta$ for some $0\leq t\leq 1$, $e_\theta \in S_1$.  For the remaining integrals we have 
    \begin{align*}
        &\int_{\R^n\setminus B_{|x|/2}(0)} \frac{|\phi(x) - \phi(x-y)|}{|y|^{n+2s}} dy \leq C(n,s) (|\phi(x)| |x|^{-2s} + \|\phi\|_{L^1(\R^n)} |x|^{-n-2s})\\
        &\qquad \qquad \leq C(n,s) \left(\sup_{z \in \R^n} (1 + |z|^n)|\phi(z)| + \|\phi\|_{L^1(\R^n)}\right) |x|^{-n-2s}
    \end{align*}
    and a similar formula for $\phi(x) -\phi(x+y)$. 
 Combining the above we obtain the desired estimate $|\FLs\phi(x)| \leq C(\phi, n,s) |x|^{-n-2s} $.
\end{proof}
\begin{remark}\label{Salpha_L2delta} {\em We have  the following inclusion $\Sw_{\alpha} \subset L^{2,\delta}({\mathbb R}^n)$ provided $\alpha > \frac{n+2\delta}{2}$. Indeed, we observe that
    \begin{align*}
        \int_{\R^n} (1+|x|^{2})^{\delta} |\phi(x)|^2 dx \leq \int_{\R^n} \frac{(1+|x|^2)^{\delta}}{(1+|x|^\alpha)^2}dx \left(\sup_{z \in \R^n} (1+|z|^\alpha)|\phi(z)|\right)^2 
        \end{align*}
 is finite if and only if $2\alpha - 2\delta >n$ i.e. $\alpha > \frac{n+2\delta}{2}$.}
\end{remark}
\begin{proposition}\label{from_fractional_Helmholtz_to_Helmholtz}
    Let $1/2<\delta<1$ and $ s\in (0,1)$. If $u\in H^{2s,-\delta}({\mathbb R}^n)$ is a solution of the homogeneous fractional Helmholtz equation in the sense  
    \begin{align}\label{fls_L2delta}
       \left< \FLs u -k^{2s}u, \phi\right>_{L^{2,-\delta}, L^{2,\delta}} = 0 \mbox{  for all } \phi \in L^{2,\delta}(\R^n),
    \end{align}
    then $u \in \mathcal{S}'(\R^n)$ (the space of tempered distributions) is a solution of the homogeneous Helmholtz equation in the sense of distributions:  
    \begin{align*}
        \left<(-\Delta-k^2)u, \phi \right>_{\mathcal{S}', \mathcal{S}} = 0 \mbox{ for all  } \phi \in \mathcal{S}(\R^n).
    \end{align*}
 Here $\left<\cdot, \cdot\right>$ denotes the duality between indicated spaces.    
\end{proposition}
\begin{proof}
Let us first consider  the case when  $n>1$.  We use the following factorization of the operator $-\Delta - k^2$ : 
\begin{align*}
    -\Delta - k^{2} &= \F^{-1} (|\xi|^2 - k^2) \F \\
    &=\F^{-1} \left((|\xi|^{2s} - k^{2s})(|\xi|^{2-2s}+k^{2-2s}) + k^{2s}|\xi|^{2-2s} - k^{2-2s} |\xi|^{2s} \right)\F\\
    &=(\FLs -k^{2s}) \left(((-\Delta)^{1-s} +k^{2-2s}) +   \F^{-1} \frac{k^{2s}|\xi|^{2-2s}-k^{2-2s}|\xi|^{2s}}{|\xi|^{2s}-k^{2s}}\F\right) \\
    &=(\FLs -k^{2s}) \left(((-\Delta)^{1-s} +k^{2-2s}) +  \F^{-1}M(\xi;k)\F(\xi)\right),
\end{align*}
where $M(\xi;k)$ is the Fourier multiplier  defined by (\ref{Mxi_def}) and we define the operator $M$ by $M := \mathcal{F}^{-1} M(\xi;k) \mathcal{F}$. For $\phi \in \Sw(\R^n)$, Lemma \ref{boundM} item (iii) in Appendix \ref{appendixA}, implies  $M\phi \in L^{2,\delta}(\R^n)$. By Lemma \ref{spaces} and Remark \ref{Salpha_L2delta} we have that $(-\Delta)^{1-s} \phi \in \Sw_{n+2-2s}(\R^n) \subset L^{2,\delta}(\R^n)$,   for $\phi \in \Sw(\R^n)$, and the inclusion holds since $\delta<1$ and  $1 < \frac{n}{2} + 2-2s $ for $n\geq 2$. Hence, we have that for any $\phi \in \Sw$, $[(-\Delta)^{1-s}+k^{2-2s}]\phi  + M\phi\in L^{2,\delta}(\R^n)$, and  furthermore for $u\in H^{2s, -\delta}(\R^n)$ satisfying (\ref{fls_L2delta}), we obtain
\begin{align*}
    0 = \left< \FLs u -k^{2s} u , [(-\Delta)^{1-s} +k^{2-2s}]\phi + M\phi \right>_{L^{2,-\delta}, L^{2,\delta}} =\left<  u , [(-\Delta) -k^{2}]\phi \right>_{\Sw', \Sw},
\end{align*}
concluding that $u$ is a solution to the Helmholtz equation in the sense of distributions. \\
For $n=1$, we use a similar argument as in \cite[Lemma 2.3]{guan2023helmholtz} where the authors show that any solution $u$ to the fractional Helmholtz equation must have the form 
\begin{align*}
    u(x) = a(x) \sin(kx) + b(x) \cos(kx)
\end{align*}
where $a(x)$ and $b(x)$ are polynomials. By assumption, $u \in L^{2,-\delta}({\mathbb R}^n)$ for $\delta<1$, therefore we see that if $a(x)$ and $b(x)$ are non-constant, there needs to be a cancellation for the $L^{2,-\delta}$-norm of $u$ to be finite
\begin{align*}
    \|u\|_{L^{2,-\delta}} = \int_{\R}(1+|x|^2)^{-\delta}\left[a(x)^2 \sin(kx)^2 +b(x)^2\cos(kx)^2) +2a(x) \sin(kx)b(x)\cos(kx)\right].
\end{align*}
Because of the sin function  in the last term, the even powers will have no contributions in the last term, consequently cannot cancel the leading powers in $a(x)^2$ and $b(x)^2$ which are always even. Thus $a(x)$ and $b(x)$ must be constant and $u(x) = a \sin(kx) + b\cos(kx) $ is a solution of the homogeneous Helmholtz equation. This completes the proof.
\end{proof}
\begin{remark}
\label{equivrem}
{\em If $u\in H^1_{loc}({\mathbb R}^n)$ is a solution to $\Delta u+k^2u=0$ in the exterior of a bounded region $D$, then the Generalized Sommerfeld Radiation Condition (GSRC)
\begin{align}
\int\limits_{{\mathbb R}^n}  \frac{|\nabla u -ik\hat x u|^2}{(1+|x|^2)^{1-\delta}}dx <+\infty, \qquad 1/2<\delta<1, \label{GSRC} \tag{GSRC}
\end{align}
and the classical Sommerfeld Radiation Condition (SRC) 
\begin{align*}
   r^{\frac{n-1}{2}}\left( \frac{\partial u}{\partial r} -ik u \right) \xrightarrow[r\to \infty]{}0,\qquad \mbox{uniformly in $\hat x = \frac{x}{|x|}$}
\end{align*}
are equivalent. While this may be a well known fact, we have not been able to find a proof in the literature. We include  a short proof of this equivalence in Appendix \ref{equiv}. For $1/2\le s<1$ Theorem \ref{prop:representation_formula_s_geq_half} and Theorem \ref{prop:representation_formula_s_less_half} in the following section show that this equivalence extends to functions $u \in H^{2s,-\delta}(\mathbb{R}^n)$ which satisfy $(-\Delta)^su-k^{2s}u=0$ in the exterior of a bounded region.}
\end{remark}
\begin{corollary}\label{uniqueness}
   Let $0<s<1$ and $1/2<\delta<1$. If $u\in H^{2s,-\delta}({\mathbb R}^n)$ is a solution to $\left< \FLs u -k^{2s}u, \phi\right>_{L^{2,-\delta}, L^{2,\delta}} = 0$ for all $\phi \in L^{2,\delta}({\mathbb R}^n)$ then $u\in H^1_{loc}({\mathbb R}^n)$. Furthermore if $u$ satisfies either (GSRC) or (SRC), then $u\equiv 0$ in $\R^n$.  
\end{corollary}
\begin{proof}
From Proposition \ref{from_fractional_Helmholtz_to_Helmholtz} we know that $u$ is a distributional solution of the Helmholtz equation inside any open ball $B_R$. It is well-known that the homogenous Helmholtz equation is hypoelliptic, i.e. a distributional solution the Helmholtz equation in any open set is $C^\infty$  (see e.g \cite[Theorem 8.12]{rudin}). Thus $u$ is an entire solution to the homogeneous Helmholtz equation, and in particular is in $H^1_{loc}({\mathbb R}^n)$. It then  follows that $u=0$ if it satisfies either (GSRC) or (SCR).
\end{proof}
\subsection{Volume Integral Representation of the Resolvent}\label{volume}
Our goal in this section is to relate the outgoing fundamental solution $G_{n,s}^k(x-y)$ (see Table \ref{tab:fund_sol} for a summary) to the extended resolvent. For this purpose, for $f \in C_c^\infty(\R^n)$, $s\in (0,1)$ and $k\in (0,\infty)$ we define the volume integral operator $\mathcal{G}_{n,s}^k$ as the convolution of $f$  with $G_{n,s}^k(x-y)$, that is   
 \begin{align}\label{vol}
        (\mathcal{G}_{n,s}^k) f(x) &:= \int_{\R^n} G_{n,s}^k(x-y) f(y) dy.
    \end{align}
Obviously $\mathcal{G}_{n,s}^k f$ solves $ \FLs u - k^{2s} u = f$. Given the structure of the fundamental solution of the fractional Helmholtz operator, it is convenient to also define   the volume integral operator with kernel being the fundamental solution of the Helmholtz operator
\begin{equation}\label{volh}
        (\mathcal{G}_{n,helm}^k)f(x) = \int_{\R^n} G_{n,helm}^k(x-y) f(y) dy 
\end{equation}        
 and the remainder       
 \begin{align}\label{volr}
 (\mathcal{J}_n^{s,k})f(x) &= \int_{\R^n}\left[G_{n,s}^k-G_{n,helm}^k\right](x-y) f(y) dy ~.
    \end{align} 

\begin{remark}\label{remark:bound_Gk_constant_suppf}
 {\em While the operator $\mathcal{G}_{n,s}^k$ is extendable as a bounded linear operator from $L^{2,\delta}(\R^n)$ to $L^{2,-\delta}(\R^n)$, it is not easy to see this from direct estimates. From Proposition \ref{prop:Gconvf_L2delta} in Appendix \ref{appendix:estimate_L2delta} we see that for a compactly supported $f$ the following estimate holds
 $$  \|\mathcal{G}_{n,s}^k f\|_{L^{2,-\delta}(\R^n)} \leq C(n,s,k,\mbox{Supp}(f))\|f\|_{L^{2,\delta}(\R^n)}$$
where the constant  depends on the size of the $\mbox{Supp}(f)$. Hence we cannot from this estimate extend $\mathcal{G}_{n,s}^k$ as an operator from $L^{2,\delta}(\R^n)$ to $L^{2,-\delta}(\R^n)$. A way to extend $\mathcal{G}_{n,s}^k$ is to consider the extended resolvent operator $R_0^+(k^{2s}) \in {\mathcal B}(L^{2,\delta}, L^{2,-\delta})$ noting that  $\mathcal{G}_{n,s}^{k_\epsilon}$ coincides with $R_0^+(k_\epsilon^{2s})$ and then take the limit as $\epsilon\to 0$.}
\end{remark}

\begin{lemma}\label{lemma:Gks=Rk}
    For all $k\in (0,\infty)$, $s\in (0,1)$, $n\in \{1,2,3\}$ and $f\in L^{2,\delta}(\R^n)$ with $\delta>1/2$, the operator $\mathcal{G}_{n,s}^k$ coincides with the extended resolvent in the following sense 
    \begin{align*}
        \mathcal{G}_{n,s}^kf = R_0^+(k^{2s})f, \qquad \mbox{\em a.e. in ${\R^n}$}.
    \end{align*}
\end{lemma}
\begin{proof}
Let $k_\epsilon^{2s} = k^{2s} + i\epsilon$ with $\epsilon >0$ and $f\in L^{2,\delta}(\R^n)$. We have shown in  Appendix \ref{appendix:estimate_L2delta}  that for $k_\epsilon \in \C^+$ with $\epsilon>0$, $\mathcal{G}_{n,s}^{k_\epsilon}f$ is in $L^{2,-\delta}(\R^n)$  for every $f\in L^{2,\delta}(\R^n)$. We also know that it satisfies the same equation as the resolvent $R_0(k_\epsilon^{2s})f$, thus $\mathcal{G}_{n,s}^{k_\epsilon}f = R_0(k_\epsilon^{2s})f$  for any $f$.  Furthermore, $R_0(k_\epsilon^{2s})f \xrightarrow{\epsilon \to 0} R_0^+(k^{2s})f$ in $L^{2,-\delta}(\R^n)$, hence in particular it converges also almost everywhere up to a subsequence. From Remark \ref{remark:bound_Gk_constant_suppf}, we can not directly estimate the $L^{2,-\delta}$ norm of $\mathcal{G}_{n,s}^k$ but rather we can show that $\mathcal{G}_{n,s}^{k_\epsilon}f(x)$ converges almost everywhere to $\mathcal{G}_{n,s}^{k}f(x)$, which is proven in Proposition \ref{prop-conv} in Appendix \ref{5D}. Thus, we conclude that $\mathcal{G}_{n,s}^kf = R_0^+(k^{2s})f$  almost everywhere  in ${\mathbb R}^n$ for every $f\in L^{2,\delta}(\R^n)$.
\end{proof}
\noindent
The above lemma in combination with Lemma \ref{lemma:Resolvent_SRC_s_geq_half} allows us to derive an existence and uniqueness result as well as an intergal representation formula for solutions to the fractional Helmholtz equation when $s\geq 1/2$. 
\begin{theorem}\label{prop:representation_formula_s_geq_half}
    Let $1/2<\delta<1$, $1/2\leq s<1$, $n\in \{1,2,3\}$ and let $f\in L^{2,\delta}(\R^n)$. The $H^{2s,-\delta}(\R^n)$ function $u$, given by
    \begin{equation}\label{Formula}
        u(x) = \int_{\R^n} G_{n,s}^k(x-y) f(y) dy =\mathcal{G}_{n,s}^kf(x) = R_0^+(k^{2s})f(x) \quad \mbox{ for a.e. } x\in \R^n~,
    \end{equation}
    is a solution to 
    \begin{equation}\label{PDE}
    \left<(\FLs - k^{2s}) u, \phi\right>_{L^{2,-\delta}, L^{2,\delta}} = \left<f,\phi\right>_{L^{2,-\delta},L^{2,\delta} } \qquad  \mbox{ for all } \phi \in L^{2,\delta}(\R^n)~,
 \end{equation}
and $u$ satisfies the generalized Sommerfeld Radiation condition (GSRC) (\ref{def:generalized_SRC-I}). Furthermore this $u$ is the unique $H^{2s,-\delta}(\R^n)$ solution to (\ref{PDE}), which satisfies (\ref{def:generalized_SRC-I}). 
    \end{theorem}
\begin{proof}
The fact that the function given by the formula (\ref{Formula}) satisfies (\ref{PDE}) and (\ref{def:generalized_SRC-I}) follows immediately from a combination of Lemma \ref{lemma:Resolvent_SRC_s_geq_half} and Lemma \ref{lemma:Gks=Rk}. It remains to establish the uniqueness. Let $v$ be any  $H^{2s,-\delta}(\R^n)$ solution to (\ref{PDE}), which satisfies (\ref{def:generalized_SRC-I}), then
$w=\mathcal{G}_{n,s}^kf -v$ is a homogeneous $H^{2s,-\delta}(\R^n)$ solution satisfying the (GSRC) (\ref{def:generalized_SRC-I}). Therefore, by the uniqueness result  in Corollary \ref{uniqueness}, we conclude that $w=0$, {\it i.e.}, $v=\mathcal{G}_{n,s}^kf$.   
\end{proof}
\noindent
We now investigate the case of $s<1/2$, where although $\mathcal{G}_{n,s}^kf$ coincides with the resolvent almost everywhere, the latter does not a priori satisfy the (GSRC) (\ref{def:generalized_SRC-I}). We instead show directly that $\mathcal{G}_{n,s}^kf$ satisfies the Sommerfeld radiation condition  (\ref{def:SRC}),  which we manage to prove only for square integrable and compactly supported functions $f$. In fact the following results hold true for all $s\in (0,\,1)$.
\begin{lemma}\label{lemma:SRC_G_s_less_than_half} Let $0<s<1$, $\delta>1/2$, $n\in \{1,2,3\}$ and $k\in(0,\infty)$. Then $\mathcal{G}_{n,s}^kf$ defined by (\ref{vol}) satisfies the standard Sommerfeld Radiation Condition  (\ref{def:SRC}) for all $f\in L^2(\R^n)$ with compact support, as well as  the equation 
       \begin{align*}
        \left<(\FLs - k^{2s}) \mathcal{G}_{n,s}^kf, \phi\right>_{L^{2,-\delta}, L^{2,\delta}} = \left<f,\phi\right>_{L^{2,-\delta},L^{2,\delta} } & \mbox{ for all } \phi \in L^{2,\delta}(\R^n).
    \end{align*}
\end{lemma}
\begin{proof}  By Lemma \ref{lemma:Gks=Rk} $\mathcal{G}_{n,s}^kf$ and  $R_0^+(k^{2s})f$ coincides, hence the $\mathcal{G}_{n,s}^kf$ satisfies the above equation since  $R_0^+(k^{2s})f$ does. The proof of the radiation condition relies on the fact that $f$ has compact support, hence the behavior as $|x|\to \infty$ is carried by the fundamental solution $G_{n,s}^k(x-y)$  which is the kernel of the integral operator $\mathcal{G}_{n,s}^k$. The proof is completed by noting that $G_{n,s}^k(x-y)$ satisfies the standard Sommerfeld radiation condition (see Theorem \ref{som} and the estimates in Appendices \ref{appendixB}).
\end{proof}
\begin{theorem}\label{prop:representation_formula_s_less_half} Let $1/2<\delta<1$, $0<s<1$, $k\in (0,\infty)$ and $n\in \{1,2,3\}$. Let $f\in L^2(\R^n)$ with compact support. The $H^{2s,-\delta}(\R^n)$ function $u$, given by
    \begin{equation}\label{Formula2}
        u(x) = \int_{\R^n} G_{n,s}^k(x-y) f(y) dy =\mathcal{G}_{n,s}^kf(x) = R_0^+(k^{2s})f(x) \quad \mbox{ for a.e. } x\in \R^n~,
    \end{equation}
    is a solution to 
    \begin{equation}\label{PDE2}
    \left<(\FLs - k^{2s}) u, \phi\right>_{L^{2,-\delta}, L^{2,\delta}} = \left<f,\phi\right>_{L^{2,-\delta},L^{2,\delta} } \qquad  \mbox{ for all } \phi \in L^{2,\delta}(\R^n)~,
 \end{equation}
and $u$ satisfies the standard Sommerfeld Radiation condition (SRC) (\ref{def:SRC}). Furthermore this $u$ is the unique $H^{2s,-\delta}(\R^n)$ solution to (\ref{PDE2}), which satisfies (\ref{def:SRC}). 
\end{theorem}
\begin{proof}
The existence part of this theorem (the first part) is already established in Lemma \ref{lemma:SRC_G_s_less_than_half}. It only remains to establish uniqueness. Let $v$ be any  $H^{2s,-\delta}(\R^n)$ solution to (\ref{PDE2}), which satisfies (\ref{def:SRC}), then
$w=\mathcal{G}_{n,s}^kf -v$ is a homogeneous $H^{2s,-\delta}(\R^n)$ solution satisfying the standard Sommerfeld Radiation Condition (SRC) (\ref{def:SRC}). Therefore, by the uniqueness result  in Corollary \ref{uniqueness}, we conclude that $w=0$, {\it i.e.}, $v=\mathcal{G}_{n,s}^kf$.   
\end{proof}

\section{The inhomogeneous fractional Helmholtz equation and the radiation condition}\label{4}
\subsection{The Lippmann-Schwinger Equation}

We turn our attention to an integral equation formulation for the  forward scattering problem, where we again assume that $q\in L^{\infty}(\R^n)$ has support in a compact region $D\subset \R^n$. Our goal is to study the well-posedeness of the direct scattering problem for the inhomogeneous fractional Helmholtz equation 
\begin{eqnarray}\label{inhomogeneous-final}
    &\FLs u^{scat} - k^{2s}(1+q) u^{scat} = k^{2s}q u^{inc}  \qquad \mbox{ in } \R^n& \label{in1}\\
    &\displaystyle{\lim\limits_{|x|\to 0} |x|^{\frac{n-1}{2}}\left( \frac{\partial u^{scat}}{\partial |x|} -ik u^{scat} \right)=0} \qquad \mbox{uniformly in $\hat x = \frac{x}{|x|}$}&\label{in2}.
\end{eqnarray}
given an incident field $u^{inc}$, which is a solution to 
$$\FLs u^{inc} - k^{2s} u^{inc}=0   \qquad \mbox{ in } \R^n. $$ 

Note that the total field $u:=u^{scat}+u^{inc}$ solves $\FLs u- k^{2s}(1+q) u =  0$ in $\R^n$. 
Introducing $u$ we may now rewrite the equation (\ref{in1}) as follows
$$\FLs u^{scat}-k^{2s}u^{scat} = k^{2s}q u \qquad \R^n~.$$
Using the volume integral representation in Section \ref{volume}, more precisely Theorem \ref{prop:representation_formula_s_less_half}, we obtain the following volume integral equation for the total field $u:=u^{scat}+u^{inc}$
 \begin{equation}\label{inth}
 u(x)=u^{inc}(x)+k^{2s}\int_{D} G_{n,s}^k(x-y)q(y) u(y) dy \qquad x\in \R^n~.
\end{equation}

\noindent
Let the bounded linear operator $T_k:L^2(D)\to L^2(D)$ be defined by 
$$T_k: w \mapsto  \int_{D} G_{n,s}^k(x-y)q(y)w(y) dy \qquad \mbox{for }x\in D.$$
Clearly $T_k$ is compact, since $T_kw\in H^{2s}(D)$. We have the following equivalence result.
\begin{theorem}\label{equiv0}
Let $0<s<1$ and $n \in \{1,2,3\}$. If $u\in H^{2s}_{loc}(\R^n)$ is a solution to (\ref{in1})-(\ref{in2})  then $u|_D$ satisfies  the Lippmann-Schwinger equation
\begin{equation}\label{LSh}
(I-k^{2s}T_k)u=u^{inc}|_{D}.
\end{equation}
Conversely if $u\in L^2(D)$ satisfies (\ref{LSh}) then $u^{scat}$  defined by  
 \begin{align}
 u^{scat}(x)=k^{2s} \int_{D} G_{n,s}^k(x-y) q(y) u(y) dy
\end{align}
is in $H^{2s,-\delta}(\R^n)$ for any $1/2<\delta<1$ and satisfies  (\ref{in1})-(\ref{in2}).
\end{theorem}
\begin{proof}
It remains to show only the converse.  Assume $u\in L^2(D)$  is a solution to (\ref{LSh}). Then we can extend it to the whole of $\R^n$ by (\ref{inth}).  Note that $qu$ is supported in $D$. 
Hence $u^{scat}:=u-u^{inc}$ is in $H^{2s,-\delta}(\R^n)$ satisfies the standard Sommerfeld  Radiation Condition (\ref{in2}), and by Lemma \ref{lemma:Gks=Rk} it coincides with $R^+_0(k^{2s})(qu)$, that is, it satisfies 
$$\FLs u^{scat}-k^{2s}u^{scat} = k^{2s}q u \qquad \mbox{in} \;\;L^{2,-\delta}({\mathbb R}^n).$$
Adding the equation for $u^{inc}$ we get $\FLs u- k^{2s}(1+q) u =  0$, and the proof is now complete by noticing that $u^{scat}:=u-u^{inc}\in H^{2s,-\delta}(\R^n)$ is a solution to (\ref{in1})-(\ref{in2}) thanks to  Theorem \ref{prop:representation_formula_s_less_half}.
\end{proof}
The Lippmann-Schwinger equation (\ref{LSh}) is a Fredholm equation thanks to the compactness of the operator $T_k$. Furthermore, the operator is compact for any $k \in \R^+$, and with similar arguments as in the proof of proposition \ref{prop-conv}, one can show that $T_k$ is differentiable in $k$ in the sense that $\lim_{k \to k_0} \frac{T_k - T_{k_0}}{k-k_0} \in \mathcal{B}(L^2(D))$. The estimates in Appendix \ref{appendix:estimate_L2delta}, show that for  $k^{2s}\|q\|_{\infty}$ sufficiently small, $I-k^{2s}T_k$ is invertible. By the Analytic Fredholm Theorem 8.26 in \cite{colton-book}, equation (\ref{LSh}) is uniquely solvable for all $k>0$ except for a discrete set $\Lambda$ (possibly empty) accumulating only at $+\infty$. We have thus established

\begin{theorem}
\label{LSlam}
Let $0<s<1$ and $n \in \{1,2,3\}$. Given $q\in L^\infty(\R^n)$ with compact support there exists a discrete set  $\Lambda \subset \R^+$ (possibly empty) such the forward scattering problem (\ref{in1})-(\ref{in2}) has a unique solution $u^{scat} \in H^{2s,-\delta}$ for any $k \in \R^+\setminus \Lambda $. The set $\Lambda$ depends on $q$ and $s$ and has $+\infty$ as only possible accumulation point.
\end{theorem}
 
\noindent 
It is typical that integral equation formulations equivalent to PDE models of wave propagation give rise to an exceptional set of frequencies. Hence, it is highly desirable to determine whether genuinely  $\Lambda\neq \emptyset$, that is, whether there exist $k>0$ for which a nonzero outgoing scattered field arises in the absence of an incident field. This corresponds precisely to the definition of scattering poles or resonances (see \cite{john1} for  a related discussion when $s=1/2$). In the classical case of the Helmholtz equation, i.e., when $s=1$ one has  $\Lambda=\emptyset$, and the scattering poles occur only for complex values of  $k$ with negative imaginary part. To shed light on this question in the fractional setting, in the next section we analyze the extended resolvent of the operator $\FLs  - k^{2s}(1+q)$, an approach that is independent of the integral representation of the solution.

\subsection{The extended Resolvent }
In this last section, we extend our analysis from the scattering problem (\ref{in1}) to the more general setting
\begin{align}\label{inhomogeneous_fractional_Helmholtz22}
    \FLs u - k^{2s}(1+q) u= f \mbox{ in } \R^n 
\end{align}
where $f$ is now in $ L^{2,\delta}(\R^n)$ but not necessarily compactly supported. We take an approach, where the solution is obtained by means of an extended resolvent operator. The multiplication operator $u\mapsto q(\cdot) u$ is clearly a map from $L^{2, -\delta}(\R^n)$ to $L^{2, \delta}(\R^n)$, since $q$ has compact support. Thus, for $k>0$, we can define the operator ${\mathcal H}_{k}: H^{2s}(\R^n)\to L^2(\R^n)$
\begin{align*}
    {\mathcal H}_{k} := \FLs -k^{2s}q,
\end{align*}
and  its  resolvent 
\begin{align*}
    R_k(z) := (\FLs - k^{2s}q -z))^{-1} \qquad \qquad z\in {\mathbb C}^+.
\end{align*}
The limiting  absorption principle as $z=k^{2s}+i\epsilon \to k^{2s}$ will allow us to analyze the solvability of (\ref{inhomogeneous_fractional_Helmholtz22}).

\medskip 
\noindent
The extension of  the resolvent $R_k$ to the real axis is studied by  Ben-Artzi and Nemirovski \cite{ben1997remarks}, more precisely Theorem 4A  provides  a limiting absorption principle for the resolvent of ${\mathcal H}_{k}$ under some conditions on $q$, and for a family of functions of $-\Delta$. Our case  of $(-\Delta)^{s}$ corresponds to power function for $s\in (0,1)$ and hence the growth condition (4.1) in \cite{ben1997remarks} is satisfied for $\gamma = 2s$.  The other the assumption in \cite[Theorem 4A]{ben1997remarks} is  that the multiplication by $(1+|x|)^{1+\epsilon}q(x)$ is a compact operator from $H^{2s}(\R^n)$  into $L^2(\R^n)$ for some $\epsilon >0$,  which for our case of $q$ compactly supported is always satisfied.

\noindent
In the following we state the results of  \cite[Theorem 4A]{ben1997remarks}  for our particular  case of ${\mathcal H}_k := \FLs -k^{2s}q$ (see also \cite[Theorem 7.1]{umeda2003generalized}  for  $s=1/2$).
\begin{theorem}
Let $1/2<\delta<1$, $0<s<1$, $n\in \{1,2,3\}$, and suppose $q\in L^\infty(\R^n)$ with compact support in $D$. Then: 
 \begin{itemize}
     \item[(i)]The continuous spectrum  $\sigma_c({\mathcal H}_k)=[0,\infty)$ of ${\mathcal H}_{k}$   is absolutely continuous except for  possibly  a discrete set of embedded  eigenvalues $\Lambda_k:=\{\lambda_j\}_{j\in {\mathbb N}}$ which can accumulate only at $0$ or $+\infty$. 
     \item[(ii)] The resolvent $R_k(z)$ can be extended continuously to $\C^{\pm}\cup\left((0,\infty)\setminus \Lambda_k\right)$ with respect to the operator norm topology of ${\mathcal B}(L^{2,\delta}(\R^n),H^{2s, -\delta}(\R^n))$ for $\delta>1/2$. That is, for all $k^{2s} \in (0,\infty)\setminus \Lambda_k $, 
     \begin{align*}
 \mbox{the limit} \qquad   R_k^\pm(k^{2s}) = \lim_{\epsilon \to 0^+} R_k(k^{2s} \pm i \epsilon) \qquad \mbox{exists in ${\mathcal B}(L^{2,\delta}, H^{2s, -\delta})$.}
     \end{align*}
 Furthermore, the operator valued functions $$
     R_k^\pm(z) := \begin{cases}
         R_k(z) & \mbox{ for } z \in \C^{\pm}\\
         R_k^\pm(z) & \mbox{ for } z \in (0,\infty) \setminus \Lambda_k
     \end{cases}$$ are ${\mathcal B}(L^{2,\delta}(\R^n),H^{2s, -\delta}(\R^n))$ valued continuous functions.
 \end{itemize}
\end{theorem}
\begin{remark} {\em
 It is not guaranteed in general that for a given $q\in L^\infty(D)$, there exists $k \in \R^+$ such that $k^{2s} \not \in \Lambda_k$, since the set $ \Lambda_k$ depends on $k$, even though for  each fixed $k$ it is a most countable. This is the main weakness of the arguments that follow: in order to be able to extend the resolvent to $R^+_k(k^{2s})$, we will a priori have to make the assumption that $k^{2s}\not \in \Lambda_k$ for that particular $k$.  }
\end{remark}
\noindent
Provided $k^{2s} \not \in \Lambda_k$, $R_k^\pm(k^{2s})$ gives the outgoing  and the incoming solution of (\ref{inhomogeneous_fractional_Helmholtz22}). Our main goal is to determine a radiation condition that distinguishes the outgoing solution $R_k^+(k^{2s})$. To this end, we first  show that the extended resolvent $R_k(z)$ satisfies the fractional Helmholtz equation with potential $q$ and source $f\in L^{2,\delta}(\R^n)$ in the $L^{2,-\delta}$ sense. 
\begin{lemma}
Let $1/2<\delta<1$, $0<s<1$, $n\in \{1,2,3\}$ and let $f\in L^{2,\delta}(\R^n)$. Suppose $k^{2s} \in (0,\infty)\setminus \Lambda_k$. Then $u^+:= R_k^+(k^{2s})f$ satisfies :
    \begin{align*}
        \FLs u^+ - k^{2s}(1+q) u^+ = f \mbox{ in } L^{2,-\delta}(\R^n)
    \end{align*}
\end{lemma}
\begin{proof}
We know that $u_\epsilon := R_k^+(k^{2s} + i\epsilon)f$ goes to $u^+:= R_k^+(k^{2s})f$ in $H^{2s,-\delta}(\R^n)$ for any $\delta>1/2$. Furthermore, $u_\epsilon \in H^{2s}(\R^n)$ for all $\epsilon >0$ and satisfies the fractional Helmholtz equation in the $L^{2,-\delta}$ sense, that is 
\begin{align*}
    \left< (-\Delta)^su_\epsilon - k^{2s}qu_\epsilon-(k^{2s}+i\epsilon)u_\epsilon , \phi \right>_{L^{2,-\delta}, L^{2,\delta}} = \left<f,\phi\right>_{L^{2,-\delta}, L^{2,\delta}}
\end{align*}
Passing to the limit in the above equation, we obtain
\begin{align*}
   &|\left< (-\Delta)^s(u_\epsilon - u^+) -k^{2s}(1+q)(u_\epsilon - u^+) + i\epsilon u_\epsilon, \phi \right>_{L^{2,-\delta}, L^{2,\delta}}| \\
   &\leq \left((1+k^{2s}(1+\|q\|_{\infty})) \|u_\epsilon - u^+\|_{H^{2s,-\delta}} + \epsilon \|u_\epsilon\|_{L^{2,-\delta}} \right)\|\phi\|_{L^{2,\delta}} \xrightarrow[\epsilon \to 0]{} 0,
\end{align*}
which completes the proof.
\end{proof}

\begin{lemma} \label{radh} Let  $1/2<\delta<1$ and  suppose that $k^{2s} \in (0,\infty)\setminus \Lambda_k$. If  $1/2\le s <1$, then $u^+ := R_k^+(k^{2s})f$ satisfies the generalized Sommerfeld Radiation condition (GSRC) (\ref{def:generalized_SRC-I}) for all $f\in L^{2,\delta}(\R^n)$. If $0<s<1$,  then $u^+ := R_k^+(k^{2s})f$ satisfies the standard Sommerfeld Radiation Condition (SRC) (\ref{def:SRC})  for all $f\in L^2(\R^n)$ with compact support. 
\end{lemma}
 \begin{proof}
 We proceed by writing the resolvent $R_k^+$ in terms of the extended unperturbed resolvent $R_0^+$  discussed in Section \ref{3} (see \cite{umeda1995radiation, umeda2003generalized}  for the case of $s=1/2$). Indeed, we have 
 \begin{align*}
 R_k^+(z) = R_0^+(z)(I+k^{2s}qR_k^+(z))
\mbox{ in } {\mathcal B}(L^{2,\delta},H^{2,-\delta}) \mbox{ for all } z \in \C^+\cup(0,\infty)\setminus \Lambda_k,
 \end{align*}
which follows from 
 \begin{align*}
({\mathcal H}_0 -z)R^+_k(z) = (\FLs-z-k^{2s}q))R^+_k(z) + k^{2s}q R^+_k(z) =  I+k^{2s}qR_k^+(z)  \mbox{ in }L^{2, -\delta}(\R^n)
 \end{align*}
 and the fact that multiplication by $q$ is in ${\mathcal B}(L^{2,-\delta}, L^{2,\delta})$, so that $qR_k^+(z) \in {\mathcal B}(L^{2,\delta}, L^{2,\delta})$. Then we conclude using the previous results  in Lemma \ref{lemma:Gks=Rk}, Lemma \ref{lemma:SRC_G_s_less_than_half}, and Lemma \ref{lemma:Resolvent_SRC_s_geq_half}. 
 \end{proof}
\noindent 
We are almost ready to establish a uniqueness result for the solution of (\ref{inhomogeneous_fractional_Helmholtz}). However, we need the following auxiliary result.
\begin{lemma}\label{inverse_id+zR}
  Suppose that $1/2 < \delta < 1$ and $0<s<1$. Then
   \begin{align*}
       (I + R_k^+(z)k^{2s}q)(I-R_0^+(z)k^{2s}q) = I \mbox{ in } {\mathcal B}(L^{2,-\delta}(\R^n))\\
       (I - R_0^+(z)k^{2s}q)(I+R_k^+(z)k^{2s}q) = I \mbox{ in } {\mathcal B}(L^{2,-\delta}(\R^n))
   \end{align*}
for every $z\in \C^+ \cup (0,\infty) \setminus \Lambda_k$.
\end{lemma}
\begin{proof}
The proof  follows the lines  of \cite[Lemma 7.3]{umeda2003generalized} for $s=1/2$. Let $v\in H^{2s}(\R^n) = \mbox{dom}({\mathcal H}_0)$. Then we write for $z\in {\mathcal C}^+$
    \begin{align*}
        (I + R_k(z) k^{2s}q) v = R_k(z) \left( \FLs -z -k^{2s}q +k^{2s}q\right)v=R_k(z) ({\mathcal H}_0-z)v,
    \end{align*}
and similarly    
    \begin{align*}
        (I - R_0(z) k^{2s}q) v &= R_0(z) \left( \FLs -z -k^{2s}q\right)v=R_0(z) ({\mathcal H}_k-z)v.
    \end{align*}
Thus we  obtain the desired identities by taking $v =(I-R_0^+(z)k^{2s}q)u$ and $v =(I+R_k^+(z)k^{2s}q) u $ respectively, for $u  \in H^{2s}(\R^n)$. Finally, we use that we can extend the resolvents to operators in ${\mathcal B}(L^{2,\delta}, L^{2,-\delta})$ in order to conclude.
\end{proof}
\begin{theorem}
 Let $1/2<\delta<1$, $0<s<1$ and suppose that $k^{2s} \in (0,\infty)\setminus \Lambda_k$. If $u\in H^{2s,-\delta}(\R^n)$ is a solution of the homogeneous fractional Helmholtz equation with potential $q$ in the sense
    \begin{align}\label{eq_L2delta}
       \left< \FLs u -k^{2s}(1+q)u, \phi\right>_{L^{2,-\delta}, L^{2,\delta}} = 0 \mbox{  for all } \phi \in L^{2,\delta}(\R^n)
    \end{align}
    and $u$ satisfies the standard  Sommerfeld radiation condition  (\ref{def:SRC}), then $u \equiv 0$. 
\end{theorem}
\begin{proof}
 Since $k^{2s} \in (0,\infty)\setminus \Lambda_k$ we have that by the  hypothesis on $u$ and definition of $R_0^+$ :
\begin{align*}
    (\FLs -k^{2s} )u = k^{2s} qu \mbox{ in } L^{2,-\delta}(\R^n)\\
    (\FLs -k^{2s} )R_0^+(k^{2s})(-k^{2s}qu) = -k^{2s} qu \mbox{ in } L^{2,-\delta}(\R^n)
\end{align*}
since the multiplication by $q$ is in ${\mathcal B}(L^{2,-\delta}, L^{2,\delta})$. From Lemma \ref{radh} and the fact that $qu \in L^{2,\delta}$ and has compact support  we have that $-R_0^+(k^{2s})(k^{2s}qu)$ satisfies the standard Sommerfeld radiation  condition. Adding the equations we obtain that 
\begin{align*}
    (\FLs -k^{2s})w = 0 \mbox{ in } L^{2,-\delta}(\R^n)
\end{align*}
and $w:=u-R_0^+(k^{2s})k^{2s}qu$ satisfies the standard Sommerfeld radiation  condition. Thus Corollary \ref{uniqueness} implies that 
\begin{align*}
    u-R_0^+(k^{2s})k^{2s}qu \equiv 0 \mbox{ in } \R^n~.
\end{align*}
Based on Lemma \ref{inverse_id+zR}, we can invert the operator $I- R_0^+(k^{2s})k^{2s}q$, and  hence $u\equiv 0$. 
\end{proof}
\begin{remark}
{\em Although not included here, many of the results in this Section work under more general assumption on $q$, namely under the assumption that the multiplication by $(1+|x|)^{1+\epsilon}q(x)$ is compact from $H^{2s}(\R^n)$  into $L^2(\R^n)$ for some $\epsilon >0$. In this case we must assume that $1/2\leq s<1$ and that the outgoing solution is selected by the generalized Sommerfeld Radiation condition (GSRC) (\ref{def:generalized_SRC-I}).}
\end{remark}

\noindent
To our knowledge, it is not known in general whether the set of embedded eigenvalues $\Lambda_k$ of the operator ${\mathcal H}_{k} := \FLs - k^{2s}q$ is empty, even for compactly supported $q$. We refer the reader to \cite{MR4491134} for a discussion of related questions. Theorems \ref{equiv0} and \ref{prop:representation_formula_s_less_half}  show that the (essential) exceptional set $\Lambda$ (from Theorem \ref{LSlam}) gives rise to the following implication: $k\in \Lambda \implies k^{2s} \in \Lambda_k$. More precisely, if $k$ is such that $I-k^{2s} T_k$ is non injective, then there exists a nontrivial $u\in L^2(D)$ such that $u = k^{2s} T_k u$ in $D$, and hence $u$ can be extended to a function in $H^{2s,-\delta}(\R^n)$ by $k^{2s}T_ku$. Furthermore, this extended  $u$ is a non trivial $H^{2s, -\delta}(\R^n)$ solution of $\FLs u-k^{2s}(q+1)u = 0$, therefore it is an eigenfunction of $\mathcal{H}_k$ with eigenvalue $k^{2s}$. Furthermore we see that these values correspond to embedded eigenvalues of ${\mathcal H}_{k}$ with outgoing eigenfunctions, i.e., eigenfunctions that satisfy the standard Sommerfeld radiation condition (\ref{def:SRC}). To establish the well-posedness of the scattering problem (\ref{in1}) for every $k$, it would therefore be very helpful to determine whether ${\mathcal H}_{k}$ admits embedded eigenvalues with outgoing eigenfunctions; this question remains open. To illustrate the subtlety of this issue, we refer the reader to \cite{sphere}, where it is shown that $\Lambda = \emptyset$ if $q(r)$ is a spherically symmetric function satisfying $q(r) \leq 1$.

 \newpage
\section{Appendices}

\begin{subappendices}
This section collects the proofs of some technical results that are used throughout the paper.
\subsection{Asymptotics and singularity estimates of the Fundamental solution}\label{appendixB}
The following theorem describes the behavior of the fundamental solution both at infinity and near the origin. It does not give the precise rates of decay at infinity or blow-up at zero of the fundamental solution, but rather provides bounds, which are sufficient for the purposes of this paper.\begin{theorem}\label{th:asymptotics-all} Let $k\in (0,\infty)$ and $n=1,2,3$. Then  the following estimate for the fundamental solution hold. 
In dimension $n=1$ we have
\begin{align*}
    |G_{1,s}^k(x) - G_{1,helm}^k(x)| \leq \begin{cases}
         \frac{C(s,k)}{\left<x\right>^{1+2s}} & \mbox{ for } s \in (\frac{1}{2},1)\\
        \frac{C(s,k)(-\ln(|x|)\chi_{(0,1)}(|x|)+1)}{\left<x\right>^{1+2s}} & \mbox{ for } s=\frac{1}{2}\\
         \frac{C(s,k)}{|x|^{1-2s}\left<x\right>^{4s}} & \mbox{ for } s \in (0,\frac{1}{2}).
    \end{cases} 
    \end{align*}
    where we recall that  $G_{1,helm}^k(x) = \frac{i e^{i |x| k}}{2s k^{2s-1}}$. 
In dimension $n=2$, we have
\begin{align*}
    |G_{2,s}^k(x) - G_{2, helm}^k(x)| &\leq   \frac{C(s,k)}{|x|}  \mbox{ for } s \in (\frac{1}{2}, 1)\\
     \left|G_{2,s}^k(x) - G_{2, helm}^k(x) -\sum_{j=0}^{m-1}  \frac{c_{2,j}k^{2sj}}{|x|^{2-2s(j+1)}}\right| &\leq  \frac{C(s,k)}{|x|}  \mbox{ for } s \in (0, \frac{1}{2}].
\end{align*}
where we recall that $G_{2,helm}^k (x) = \frac{ik^{2-2s}}{4s } H^{(1)}_0(k |x|)$, which has a logarithmic singularity at $0$ and decays like $|x|^{-1/2}$ at $\infty$. 
In dimension $n=3$, we have
\begin{align*}
    |G_{3,s}^k(x) - G_{3,helm}^k(x) | &\leq   \frac{C(s,k)}{|x|^{3-2s}\left<x\right>^{4s}} & \mbox{ for } s \in (\frac{1}{2}, 1)\\
    \left|G_{3,s}^k(x) - G_{3,helm}^k(x) - \sum_{j=0}^{m-1} \frac{c_{3,j}k^{2sj}}{|x|^{3-2s(j+1)}} \right|& \leq   \frac{C(s,k)}{|x|^{3-2sm}} & \mbox{ for } s \in (0, \frac{1}{2}].
\end{align*}
where we recall that $G_{3,helm}^k (x) = \frac{k^{2-2s}}{s}\frac{e^{ik |x|}}{4\pi|x|}$. 
Here $C(s,k)>0$ is a constant depending on $s$ and $k$, $\chi_{(0,1) }$ denotes the indicator function of the interval $(0,1)$, $m=\lfloor \frac{1}{2s} \rfloor$, and we use the Japanese bracket notation $\left<x\right>^\alpha:=(1+|x|^2)^{\alpha/2}$.
\end{theorem}
\begin{proof}
We consider the different cases in separate lemmas, more specifically the asymptotic behavior of the fundamental solution  for each dimension as  $|x|\to \infty$ is analyzed  in Lemmas \ref{1D-infinity}, \ref{2D-infinity}, \ref{3D-infinity}, and the behavior as $|x|\to 0$ is analyzed  in Lemmas \ref{1D-sg}, \ref{2D-sg}, \ref{3D-sg}. 
\end{proof}
\subsubsection{Asymptotic behavior  at infinity}
\begin{lemma}[1D - Asymptotics at infinity]\label{1D-infinity} For $k\in(0,\infty)$, the fundamental solution  for $n=1$  as $|x|\to \infty$ satisfies
    $$
   \left| G_{1,s}^k (x) - \frac{ie^{i|x|k}}{2sk^{2s-1}} \right| \leq \frac{C(s,k)}{|x|^{1+2s}}.
    $$
\end{lemma}
\begin{proof}
 We only need to show the rate of decay of the term $J_1^{s,k}$ defined by (\ref{def:J1}) as $|x|\to \infty$. To this end, we explicitly observe that
\begin{align}\label{Im_computation_1D}
   |J_1^{s,k}(x)|= &\frac{1}{\pi |x|^{1-2s}} \left|\int_0^\infty \Im\left[\frac{e^{-y}}{y^{2s}e^{i\pi s}-k^{2s}|x|^{2s}}\right]dy\right| \nonumber\\
    &\leq \frac{1}{\pi |x|^{1-2s}} \int_0^\infty \frac{e^{-y} y^{2s}\sin(s\pi)}{|y^{4s} +k^{4s}|x|^{4s}-2y^{2s}\cos(s\pi)k^{2s}|x|^{2s} |} dy \nonumber\\
    &\leq \frac{1}{\pi |x|^{1+2s}} \int_0^\infty \frac{e^{-y} y^{2s}\sin(s\pi)}{P(y^{2s}/|x|^{2s}, k)} dy\nonumber\\
    &\leq \frac{1}{\pi |x|^{1+2s}} \int_0^\infty \frac{e^{-y} y^{2s}\sin(s\pi)}{k^{4s}(1-\cos(s\pi)^2)} dy,
\end{align}
where in the last step we used that the polynomial
\begin{align}\label{Poly}
    P(X,k) :=X^2 + k^{4s} - 2k^{2s}\cos(s\pi)X
\end{align} 
has discriminant $\Delta_P = 4k^{4s}(\cos(s\pi)^2 - 1) <0$, and therefore $P$ is bounded below as
\begin{align}\label{Lower_bound_P} P(X,k) \geq P(k^{2s}\cos(s\pi),k)= k^{4s}(1-\cos(s\pi)^2).\end{align} 
\end{proof}
\begin{lemma}[2D - Asymptotics at infinity]\label{2D-infinity}For $k\in(0,\infty)$, the fundamental solution  for $n=2$  as $|x|\to \infty$ satisfies      
\begin{align*}
    \left|G_{2,s}^k(x)- \frac{ik^{2-2s}}{4s } H^{(1)}_0(k |x|)\right|  &\leq \frac{C(s,k)}{|x|} \quad & \mbox{ when } s \in (\frac{1}{2},1)\\
     \left|G_{2,s}^k(x)- \frac{ik^{2-2s}}{4s } H^{(1)}_0(k |x|) - \sum_{j=0}^{m-1}\frac{c_{2,j}k^{2sj}}{|x|^{2-2s(j+1)}}\right|  &\leq \frac{C(s,k)}{|x|} \quad & \mbox{ when } s \in (0,\frac{1}{2}]
\end{align*}
\end{lemma}
\begin{proof} Since the fundamental solution is given by different formulas depending on the value of $s$, we need to consider several cases separately. We use the notation $r=|x|$.
    \begin{itemize}
        \item For $s\in (\frac{1}{2}, 1)$, we see that the function $F_0$ defined  by (\ref{Fm}) with $m=0$  is in $L^2(\R^2)$, has a removable singularity at $r=k$,  and furthermore 
        \begin{align}\label{derivativeF_half-1}
            \frac{\partial}{\partial r} F_0(r, k)  &= \frac{-2s r^{2s-1}}{(r^{2s}-k^{2s})^2} + \frac{2s^{-1}k^{2-2s}r}{(r^2-k^2)^2}.
        \end{align}
Note that this expression at infinity is $O(\frac{1}{r^{1+2s}})$, and using the Taylor expansion we get
        $$
       \lim_{r\to k} \frac{\partial}{\partial r}F_0(r,k) = 0.
        $$
which implies that $ \frac{\partial}{\partial r}F_0$ is an $L^1(\R^2)$ function.  We have 
        \begin{align*}
            |x_iJ_2^{s,k}(|x|)| = |\mathcal{F}^{-1}\partial_iF_0(|x|)| \leq \| \nabla F_0\|_{L^1}, 
        \end{align*}
which implies that $ |x| J_2^{s,k}(|x|)$ is bounded, hence $J_2^{s,k}(|x|) = O(|x|^{-1})$ as $|x|\to \infty$.
 \item For $s\in (0, \frac{1}{2})$ and $\frac{1}{2s}\not\in \mathbb{N}$, and $F_m$ defined  by (\ref{Fm}) with $m=\lfloor \frac{1}{2s}\rfloor$, we have 
    \begin{align}\label{derivativeF_0_half_notinteger}
        F_m(r, k) &= \frac{k^{2sm}}{r^{2sm}(r^{2s}-k^{2s})} - \frac{s^{-1}k^{2-2s}}{(r^2-k^2)}\nonumber\\
        \frac{\partial }{\partial r} F_m(r,k) &= \frac{-k^{2sm}[2sm(r^{2s}-k^{2s}) + 2sr^{2s}]}{r^{2sm+1}(r^{2s}-k^{2s})^2} + \frac{2s^{-1}k^{2-2s}r}{(r^2-k^2)^2}.
    \end{align}
The singularity at $k$ is removable, and  as $r\to k$
    \begin{align*}
        \lim_{r\to k} \frac{\partial}{\partial r} F_m(r, k ) = \frac{m(2s-1)-1}{2sk^2},
    \end{align*}
whereas the singularity at 0 is integrable in $\R^2$, since $2sm<1$. Finally, the behavior at infinity of $\frac{\partial }{\partial r}F_m$ is $O(r^{-2s(m+1)-1})$,  hence it is in $L^1(\R^2)$. By the same argument as in the previous item, it follows that $J_2^{k,s}(|x|) = O(|x|^{-1})$ as ${|x|\to \infty}$. 
\item For $s\in (0,\frac{1}{2})$ and $\frac{1}{2s}\in \mathbb{N}$, and $ \tilde F_m (r, k)$ and $ F_m(r,k)$ defined in Section \ref{s2.2.2}, we have 
    \begin{align}\label{derivative_tildeF}
        \tilde F_m (r, k) &=   F_m(r,k)+ \frac{k^{2-2s}}{r(r+k)}\nonumber\\
        \frac{\partial}{\partial r}\tilde F_m(r,k) &=\frac{\partial}{\partial r} F_m(r, k ) -\frac{k^{2-2s}}{r^2(r+k)} - \frac{k^{2-2s}}{r(r+k)^2}.
    \end{align}
    According to the previous analysis, $\frac{\partial}{\partial r}\tilde F_m(r,k)$ does not have a singularity at $k$, and the behavior at infinity is still a $O(r^{-2-2s})$. Finally, around zero we have 
    \begin{align*}
        \frac{\partial }{\partial r} \tilde F_m(r, k) \sim \frac{k^{1+2s}}{r^2(r^{2s}-k^{2s})^2} - \frac{k^{2-2s}}{r^2(r+k)} - \frac{k^{2-2s}}{r(r+k)^2} \sim  - \frac{k^{-2s}}{r}
    \end{align*}
    Hence we conclude once again that $\frac{\partial }{\partial r}\tilde F_m\in L^1(\R^2)$, and so as before $J_2^{k,s}(|x|) = O(|x|^{-1})$ as ${|x|\to \infty}$. For $s=\frac12$ the function $\tilde F_m$ vanishes, as observed earlier.

    \item For $s \in (0, \frac{1}{2}]$, we have to consider the contribution of the Struve Function of the second kind. Its asymptotic for large $|x|$ is given by  
    $$
   K_0(k |x|) = \frac{2}{\pi k |x|} + O(|x|^{-2})~, 
    $$
    see for example formula C.38 in \cite{hiltunen2024nonlocal}. 
    \end{itemize}
\end{proof}

\begin{lemma}[3D - Asymptotics at infinity]\label{3D-infinity}For $k\in(0,\infty)$, the asymptotic behavior of the fundamental solution  for $n=3$  as $|x|\to \infty$ satisfies
    \begin{align*}
    \left| G_{3,s}^k(x) - \frac{k^{2-2s}}{s} \frac{e^{i k |x|}}{4\pi |x|} \right| &\leq \frac{C(s,k)}{|x|^{3+2s}} &\mbox{ when } s\in (\frac{1}{2}, 1]\\
    \left| G_{3,s}^k(x) - \frac{k^{2-2s}}{s} \frac{e^{i k |x|}}{4\pi |x|} - \sum_{j=0}^{m-1} \frac{c_{3,j}k^{2sj}}{|x|^{3-2s(j+1)}}\right| &\leq \frac{C(s,k)}{|x|^{3-2sm}} &\mbox{ when } s \in (0, \frac{1}{2}].
    \end{align*}
\end{lemma}
\begin{proof}  Again we consider different ranges of $s$.
\begin{itemize}
    \item  For  $s\leq1/2$, the asymptotic behavior of the term $J_3^{s,k}$ as $|x|\to \infty$ can be obtained in a similar way as in dimension  one, using the lower bound (\ref{Lower_bound_P}) of the corresponding $P$. Indeed
\begin{align}\label{asymptotics:J3_s_less__half}
    J_3^{s,k}(x) = &\frac{k^{2sm}}{2\pi^2 |x|^{3-2s(m+1)}}\left|\int_0^\infty e^{-y} y^{1-2sm} \Im\left[ \frac{e^{i\pi sm}}{y^{2s} e^{-i\pi s} -|x|^{2s}k^{2s}} \right]dy\right| \nonumber\\
    \nonumber\\
     &\leq \frac{C(k,s)}{ |x|^{3-2s(m-1)}}\int_0^\infty \frac{e^{-y} y^{1-2sm}(y^{2s}+|x|^{2s})}{P(|x|^{-2s}y^{2s},k)}dy=O(|x|^{-3+2sm})
\end{align}
\item When $s>1/2$ the expression for $J_3^{s,k}$ differs from that of  $J_1^{s,k}$ only by a factor of $y$ inside the integral and a factor of $|x|^{-2}$ outside, and so the desired estimate follows directly as in (\ref{Im_computation_1D}) . 
\end{itemize}
   \end{proof}

\subsubsection{Singularity at zero}
\begin{lemma}[1D - Singularity as $|x|\to 0$]\label{1D-sg}
Let $k\in(0,\infty)$. For $|x|$ sufficiently small, the fundamental solution for $n=1$ satisfies
\begin{align*}
    \left|G_{1,s}^k(x) - \frac{i e^{ik|x|}}{2s k^{2s-1}}\right|\leq \begin{cases}
      C(s,k) & \mbox{ when } s \in (\frac{1}{2}, 1)\\
      -C(s,k)\ln(|x|) &\mbox{ when } s = \frac{1}{2}\\
      C(s,k) |x|^{-1+2s} &\mbox{ when } s\in (0,\frac{1}{2}),
    \end{cases}
    \end{align*}
where $C(s,k)>0$ is a constant depending on $s$ and $k$.    
\end{lemma}
\begin{proof}
Similarly to the behavior at infinity, we consider several cases  of $s$.
    \begin{itemize}
        \item For $s> 1/2$, by making a change of variable $y \to |x|y$, and since $\cos(s\pi)< 0$ we obtain 
    \begin{align*}
        |J_1^{s,k}(x)|  &\leq C(k,s) \int_0^\infty \frac{e^{-|x|y} y^{2s}\sin(s\pi)}{y^{4s} +k^{4s}-2y^{2s}\cos(s\pi)k^{2s}} dy \leq C(k,s)\int_0^\infty \frac{y^{2s}\sin(s\pi)}{y^{4s} +k^{4s}} dy,    \end{align*}
where the last integral is finite since $2s>1$. 
\item For $s<1/2$, we use the following inequality
\begin{align*}
    (1-\cos(s\pi)^2) y^{4s} \leq P(|x|^{2s}k^{2s},|y|)= y^{4s} -2k^{2s}|x|^{2s}\cos(s\pi)y^{2s} + k^{4s}|x|^{4s}~,
    \end{align*}
to obtain
    \begin{align*}
         J_1^{s,k}(x)  &= \frac{C(k,s)}{|x|^{1-2s}}\int_0^\infty \frac{e^{-y} y^{2s}\sin(s\pi)}{y^{4s} +k^{4s}|x|^{4s}-2y^{2s}\cos(s\pi)k^{2s}|x|^{2s}} dy\\
        &\leq \frac{C(k,s)}{|x|^{1-2s}} \int_0^\infty \frac{e^{-y}}{y^{2s}} dy = O\left(\frac{1}{|x|^{1-2s}}\right),
    \end{align*}
    where the last integral is finite since $2s<1$. 
    \item We need to handle the case $s=1/2$ separately, as none of the above argument works in this case. Instead we perform an integration by parts 
    \begin{align*}
         J_1^{s,k}(x)  &= \frac{1}{\pi} \int_0^\infty \frac{e^{-y} y}{y^{2} +k^{2}|x|^{2}} dy\\
         &= \frac{1}{2\pi}\int_0^\infty e^{-y} \ln\left( y^2 + k^2|x|^2\right)dy + \left[ \frac{e^{-y}}{2\pi} \ln(y^2+k^2|x|^2)\right]_0^\infty\\
         &\leq -\ln(|x|) + C(k),
    \end{align*}
   which gives the desired result. 
    \end{itemize}
\end{proof}
\begin{lemma}[2D - Singularity as $|x|\to 0$]\label{2D-sg}
   For $|x|$ sufficiently small, the fundamental solution for $n=2$ satisfies
    \begin{align*}
    \left|G_{2,s}^k(x)- \frac{ik^{2-2s}}{4s } H^{(1)}_0(k |x|)\right|  &\leq \frac{C(s,k)}{|x|} \quad & \mbox{ when } s \in (\frac{1}{2},1) \\
     \left|G_{2,s}^k(x)- \frac{ik^{2-2s}}{4s } H^{(1)}_0(k |x|) - \sum_{j=0}^{m-1}\frac{c_{2,j}k^{2sj}}{|x|^{2-2s(j+1)}}\right|  &\leq \frac{C(s,k)}{|x|}\quad & \mbox{ when } s \in (0,\frac{1}{2}]
\end{align*}
 where $C(s,k)>0$ is a constant depending on $s$ and $k$.       
\end{lemma}

\begin{proof}
\begin{itemize}

\item The term $J_2^{s,k}$ for every case of $s$ always contains an $L^2$ function, as it is the inverse Fourier transform of an $L^2$ function. The singularity at zero can be handled by the analysis done in Lemma \ref{2D-infinity}, which yields
$$
|x| |J^{s,k}_2(|x|)| \leq 2 \|\nabla F\|_{L^1(\R^2)} <\infty.
$$ 

 \item For  $s\in (0,\frac{1}{2}]$ and $\frac{1}{2s}\in \mathbb{N}$, we have the extra term $K_0(|x|k)$ which has a logarithmic singularity around $0$. 
\end{itemize}
\end{proof}
\begin{lemma}[3D - Singularity as $|x|\to 0$]\label{3D-sg}
 For $|x|$ sufficiently small, the fundamental solution for $n=3$ satisfies
 \begin{align*}
    \left| G_{3,s}^k(x) - \frac{k^{2-2s}}{s} \frac{e^{i k |x|}}{4\pi |x|} \right| &\leq \frac{C(s,k)}{|x|^{3-2s}} &\mbox{ when } s\in (\frac{1}{2}, 1)\\
    \left| G_{3,s}^k(x) - \frac{k^{2-2s}}{s} \frac{e^{i k |x|}}{4\pi |x|} - \sum_{j=0}^{m-1} \frac{c_{3,j}k^{2sj}}{|x|^{3-2s(j+1)}}\right| &\leq \frac{C(s,k)}{|x|^{3-2sm}} &\mbox{ when } s \in (0, \frac{1}{2}].
    \end{align*}
 where $C(s,k)>0$ is a constant depending on $s$ and $k$.       
\end{lemma}
\begin{proof}
Again we go through all the different cases for $s$.
\begin{itemize}
\item For $s> 1/2$, $\cos(s\pi)<0$  so for $J_3^{s,k}$  as $|x| \to 0$, we obtain
\begin{align*}
        |J_3^{s,k}(x)|  &= \frac{C(s,k)}{|x|^{3-2s}}\int_0^\infty \frac{e^{-y} y^{1+2s}\sin(s\pi)}{y^{4s} +k^{4s}|x|^{4s}-2y^{2s}\cos(s\pi)k^{2s}|x|^{2s}} dy\\
        &\leq \frac{C(s,k)}{|x|^{3-2s}} \int_0^\infty \frac{e^{-y} \sin(s\pi)}{y^{2s-1}} dy.  
    \end{align*}
    \item For $s<1/2$ on the other hand, we obtain 
\begin{align}\label{asymptotics:J3_s_less__half_near0}
    |J_3^{s,k}(x)| = &\frac{k^{2sm}}{2\pi^2 |x|^{3-2s(m+1)}}\left|\int_0^\infty e^{-y} y^{1-2sm} \Im\left[ \frac{e^{i\pi sm}}{y^{2s} e^{-i\pi s} -k^{2s}|x|^{2s}} \right]dy\right| \nonumber\\
    \nonumber\\
     &\leq \frac{k^{2sm}}{2\pi^2 |x|^{3-2s(m+1)}}\int_0^\infty e^{-y} y^{1-2sm}\frac{|y^{2s}\sin(s(m+1)\pi) -k^{2s} |x|^{2s}\sin(sm\pi)|}{P(y^{2s},k|x|)}dy
   \nonumber \\
    \nonumber\\
    &\leq  \frac{C(k,s)}{|x|^{3-2s(m+1)}}\int_0^\infty e^{-y} y^{1-2sm}\frac{y^{2s}+1}{P(y^{2s}, k|x|)}  dy,
\end{align}
where
$$ P(X,k|x|) :=X^2 + k^{4s}|x|^{4s} - 2k^{2s}|x|^{2s}\cos(s\pi)X.$$
Next, we observe that
    \begin{align*}
        P(y^{2s}, k|x| )&= (y^{2s}- k^{2s}|x|^{2s})^2 + 2(1-\cos(s\pi))y^{2s}|x|^{2s} \\
        &\geq C(s,k)k^{2s} y^{2s}|x|^{2s}
    \end{align*}
 where $C(s,k) = 2(1-\cos(s\pi))>0$. Hence, we deduce that 
    $$
   |J_3^{s,k}(x)| =O_{|x|\to 0}(|x|^{-3+2sm}),
    $$
    since $y^{1-2s(m+1)}$ is integrable when $s<1/2$. 
    \item For $s=1/2$, we directly estimate 
    \begin{align*}
        J_3^{s,k}(x) = &\frac{k}{2\pi^2 |x|}\left|\int_0^\infty e^{-y} \Im\left[ \frac{i}{-iy -k|x|}  \right]dy\right|\nonumber\\\nonumber \\
        &\leq \frac{k}{2\pi^2 |x|}\left|\int_0^\infty e^{-y} \frac{k|x|}{y^2 +k^2|x|^2}  dy\right| \leq \frac{C(k,s)}{|x|^2} = O(|x|^{-3+2sm}).
    \end{align*}
    \end{itemize}
\end{proof}

\begin{remark}
{\em  In the case $s=1/2$ the fundamental solution was previously computed in \cite{hiltunen2024nonlocal}, using the limiting absorption principle and contour integral. Although the formulas look different, they coincide. It is however interesting to note that using the expression of the 3D fundamental solution obtained in \cite{hiltunen2024nonlocal}, namely
\begin{align*}
    G_{3,1/2}^k (x) = \frac{1}{2\pi^2 |x|^2} - \frac{ik}{4\pi^2|x|}\left( e^{ik|x|}E_1(ik|x|) - e^{-ik|x|}E_1(-ik|x|) \right) + \frac{k e^{ik|x|}}{2\pi |x|},
\end{align*}
one obtains  a better order of decay at infinity than the one we obtained. Indeed, according to the asymptotic expression for the exponential integral
\begin{align}\label{E1_asymptotic}
E_1(z) = \frac{e^{-z}}{z}\left(1 + \frac{1}{z} + O(|z|^{-2})\right).
\end{align}
we see  that 
\begin{align*}
    &\frac{1}{2\pi^2 |x|^2} - \frac{ik}{4\pi^2|x|}\left( e^{ik|x|}E_1(ik|x|) -e^{-ik|x|}E_1(-ik|x|) \right) \\
    \\
    &\sim \frac{1}{2\pi^2 |x|^2} - \frac{ik}{4\pi^2|x|}\left( \frac{1}{ik|x|} + \frac{1}{-k^2|x|^2} + O(|x|^{-3}) - \frac{1}{-ik|x|} - \frac{1}{-k^2|x|^2} + O(|x|^{-3}) \right) \\
    \\
    &=\frac{1}{2\pi^2|x|^2} - \frac{2ik}{4\pi^2|x|^2ik} + O(|x|^{-4}) = O(|x|^{-4}).
\end{align*}
With our expression for the fundamental solution for $s=1/2$ we obtained the bound $O(|x|^{-2})$ on the decay at infinity. There clearly is a cancellation at infinity which our estimation disregards. However, our asymptotic bounds are sufficient to conclude that the fundamental solution satisfies the Sommerfeld radiation condition, as already noted earlier.}
 \end{remark}
\subsubsection{Asymptotic at infinity of \texorpdfstring{$\partial_r J_n^{s,k}(r)$}{TEXT}}
In order to verify that the fundamental solution of the fractional Helmholtz operator satisfies the Sommerfeld radiation condition, we need to check that the derivatives of all terms other than the one corresponding to the Helmholtz equation have the appropriate decay at infinity.
\begin{lemma} Let $k>0$ and  $n=1,2,3$. Then 
    \begin{align*}
        \lim_{r\to \infty} r^{\frac{n-1}{2}}\left|\partial_r J_n^{s,k}(r)\right| = 0.
    \end{align*}
\end{lemma}
\begin{proof}
We check this for each dimension.
    \begin{itemize}
        \item For $n=1$, taking the derivative under the integral  yields  
        \begin{align*}
            \left| \partial_r J_1^{s,k}(r)\right| &\le \left| \frac{(2s-1)}{2\pi r^{2-2s}} \int_0^\infty e^{-y}y^{2s}\frac{c(s)}{P(y^{2s}, kr)}dy \right|\\
            &\quad +\left|  \frac{1}{2\pi r^{1-2s}} \int_0^\infty e^{-y}y^{2s}\frac{c(s)\partial_r P(y^{2s}, rk)}{P(y^{2s}, kr)^2}dy \right| \leq  \frac{c(s,k)}{r^{2-2s}} \int_0^\infty\frac{ e^{-y}y^{2s}}{P(y^{2s}, kr)}dy,
        \end{align*}
        where we used that for some constant depending on $s$, one may calculate 
\begin{align}\label{partial_rP_less_P}
    |\partial_r P(y^{2s}, kr) | = |-4sk^{2s}\cos(s\pi) y^{2s}r^{2s-1} + 4sk^{4s}r^{4s-1}| \leq c(s) r^{-1} P(y^{2s}, kr).
\end{align}
The polynomial $P(\cdot,\cdot)$  is that defined by (\ref{Poly}). Applying the analysis in the proof  of Lemma \ref{1D-infinity}, we conclude that $|\partial_r J_1^{s,k}(r)| = O(r^{-2-2s})$, hence $ \lim\limits_{r\to \infty}|\partial_r J_1^{s,k}(r)| =0$.
    \item For $n=2$, we need to verify that the radial derivative of the inverse Fourier transform of $F_m$ (or $\tilde F_m$, which the calculations below apply to), decays faster then $1/\sqrt{|x|}$. To this end, we have 
    \begin{align*}
        \frac{\partial}{\partial r}J_2^{s,k}(r) &= \frac{\partial}{\partial r} \mathcal{F}^{-1} F_m(|\xi|, k)(r) = \frac{1}{4\pi^2}\sum_{i=1}^2 \frac{\partial}{\partial x_i}\int_{\R^2} e^{i \xi \cdot x}F_m(|\xi|, k) d\xi \frac{x_i}{|x|}\\
        &= \frac{1}{4\pi^2 |x|}\sum_{i=1}^2 \int_{\R^2} \frac{\partial}{\partial \xi_i}e^{i \xi \cdot x}\xi_i F_m(|\xi|, k) d\xi \\
        &= - \frac{1}{4\pi^2|x|}\sum_{i=1}^2 \int_{\R^2} e^{i \xi \cdot x}[\xi_i \frac{\partial}{\partial \xi_i}F_m(|\xi|, k) + F_m(|\xi|, k)]d\xi \\
        &= - \frac{1}{4\pi^2|x|}\int_{\R^2} e^{i \xi \cdot x}[|\xi| \frac{\partial}{\partial |\xi|}F_m(|\xi|, k) + F_m(|\xi|, k)]d\xi\\
        &= - \frac{1}{|x|}\mathcal{F}^{-1} [|\xi|\frac{\partial}{\partial |\xi|}F_m(|\xi|, k) + F_m(|\xi|, k)].
    \end{align*}
Applying the analysis  in proof of Lemma \ref{2D-infinity}, we  observe that we only need to check the limit of the first term, that is we need to prove
\begin{align*}
    \frac{1}{|x|^{1/2}} \left|\mathcal{F}^{-1}\left(|\xi|\frac{\partial}{\partial |\xi|}F_m(|\xi|, k) \right)(x) \right|\xrightarrow{|x|\to \infty} 0. 
\end{align*}
For this, it is enough to prove that 
    \begin{align*}
|x|\left|\mathcal{F}^{-1}\left(|\xi|\frac{\partial}{\partial |\xi|}F_m(|\xi|, k) \right)(x) \right|&\leq2 \left\Vert \nabla\left(  |\xi|\frac{\partial}{\partial |\xi|}F_m(|\xi|, k)\right) \right\Vert_{L^1(\R^2)}\\
&\leq C \left\Vert r \frac{\partial}{\partial r}\left(  r\frac{\partial}{\partial r}F_m(r, k)\right) \right\Vert_{L^1(\R)} <\infty~.
    \end{align*}
We  show  this for each of the cases of the range of $s$.  For $s\in(0,\frac{1}{2})$ and $\frac{1}{2s}\not \in \mathbb{N}$, the derivative is computed in (\ref{derivativeF_0_half_notinteger}). We have 
\begin{align*}
    r\frac{\partial}{\partial r}\left(  r\frac{\partial}{\partial r}F_m(r, k)\right) &= \frac{4s^2k^{2sm}[(m+1)^2r^{4s} - (2m^2+2m-1)k^{2s}r^{2s} + m^2 k^{4s}]}{r^{2sm}(r^{2s}-k^{2s})^3}\\
    &\quad - \frac{4s^{-1}k^{2-2s}r^2(r^2+k^2)}{(r^2-k^2)^3} \sim_{r\to \infty} r^{-2sm - 2s },
\end{align*}
which is integrable at infinity since $2sm + 2s >1$. Furthermore, the singularity at $k$ is removable, as seen from a Taylor expansion,  and around $0$ the function is a $O(r^{-2sm})$ with $2sm<1$, hence we conclude that  $\frac{\partial}{\partial r}\left(  r\frac{\partial}{\partial r}F_m(r, k)\right)\in L^1(\R^2)$.   For $s\in (\frac{1}{2}, 1)$, the same computations and results hold with $m=0$. 
\noindent
 Finally, when $s\in (0,\frac{1}{2})$ and $\frac{1}{2s} \in \mathbb{N}$, we have
\begin{align*}
     r \frac{\partial}{\partial r}\left( r\frac{\partial}{\partial r}\tilde F_m(r,k) \right)&= r \frac{\partial}{\partial r}\left(r \frac{\partial}{\partial r} F_m(r, k )\right) +\frac{k^{2-2s}(4r^2+3kr+k^2)}{r(r+k)^3}
\end{align*}
As before, we see that the singularity at $k$ is removable and the function is integrable at infinity. One can easily verify that the singularity at $0$ is also  removable. Hence, we conclude that $r \frac{\partial}{\partial r}\left( r\frac{\partial}{\partial r}\tilde F_m(r,k) \right) \in L^1(\R)$. For the term $K_0(|x|k)$, we can use the formula
\begin{align*}
    \frac{\partial}{\partial |x|} K_0(k|x|) = \frac{2 k }{\pi} - k K_1(k |x|)~,
\end{align*}
where $K_1$ denotes the Struve function of the second kind of order $1$. Since 
$$
K_1(z) = \frac{2}{\pi} + O_{|z|\to \infty}(|z|^{-2})
$$
(see \cite{lozier2003nist}), we get that $\frac{\partial}{\partial |x|}K_0( k |x|) = O_{|x| \to \infty} (|x|^{-2})$.
    \item For $n=3$, the reasoning is similar to the case of $n=1$. We only need to check  the term coming from transferring  the derivative inside the integral, 
\begin{align*}
    &\left|\frac{\partial}{\partial r}J_3^{s,k}(r) \right|\leq  \frac{c(k,s)}{ r} |J_3^{s,k}(r)| \\
    &+ \frac{c(s,k)}{r^{3-2s(m+1)}}\left|\int_0^\infty e^{-y} y^{1-2sm}  \frac{|r^{2s-1} P(y^{2s}, kr) |+|y^{2s}+r^{2s}|| \partial_r P(y^{2s}, kr)| }{|P(y^{2s}, kr)|^2} dy\right|=O(r^{-4+2sm}),\nonumber
    \end{align*}
where we use (\ref{partial_rP_less_P}) and the analysis  in the proof Lemma \ref{3D-infinity}. 
 \end{itemize}
  \end{proof}

\subsection{\texorpdfstring{$L^{2,-\delta}$}{TEXT}-estimates of \texorpdfstring{$\mathcal{G}_{n,s}^k$}{TEXT}}\label{appendix:estimate_L2delta}
In this appendix we prove the following two  propositions.
\begin{proposition}\label{prop:Gepsconvf_L2delta}
 Let $1/2<\delta<1$, $k_\epsilon^{2s}=k^{2s}+i\epsilon$ for $\epsilon >0$, and $n=1,2,3$. Then for any $g\in L^2(\R^n)$ we have 
  \begin{align*}
  \|G_{n,s}^{k_\epsilon}*g\|_{L^{2}(\R^n)}\leq C(\epsilon,\delta, s)\|g\|_{L^2(\R^n)}, 
  \end{align*}
  where $C(\epsilon,\delta, s)>0$ is a constant depending on $\epsilon$, $\delta$ and $s$.
\end{proposition}
\begin{proposition}\label{prop:Gconvf_L2delta}
    Let $1/2<\delta<1$, and $n=1,2,3$. Then for any $g\in L^2(\R^n)$ with compact support $D$, we have 
  \begin{align*}
      \|G_{n,s}^k*g\|_{L^{2,-\delta}(\R^n)} \leq C(D,\delta, s)\|g\|_{L^2(\R^n)}, 
  \end{align*}
 where $C(D,\delta, s)>0$ is a constant depending on $D$, $\delta$ and $s$. 
\end{proposition}
\noindent
It is sufficient  to show that the bound holds for each term of the fundamental solution. We will need the following lemma 
\begin{lemma}\label{lemma:L2delta_conv_alpha_beta}
Let $1/2<\delta<1$. Assume that $f$  satisfies $|f(x)| \leq |x|^{-\alpha}\left<x\right>^{-\beta}$ with
$$2\alpha + 2\beta >n - 2\delta \qquad \mbox{ and }\qquad 2\alpha <n.$$
Then for any $g\in L^2(\R^n)$ with compact support $D$, we have  that
  \begin{align*}
      \|f*g\|_{L^{2,-\delta}(\R^n)} \leq C(D,\delta, \alpha, \beta) \|g\|_{L^2(\R^n)}, 
  \end{align*}
 where $C(D,\delta, \alpha, \beta)>0$ is a constant depending on  $D$, $\delta$, $\alpha$ and  $\beta$.
\end{lemma}

\begin{proof} First we observe that 
\begin{align*}
    \int_{\R^n} \left|\int_D f(x-y) g(y) dy\right|^2 \frac{dx}{(1+|x|^2)^\delta} \leq \int_{\R^n} \int_D \frac{dy}{|x-y|^{2\alpha}<x-y>^{2\beta}} \frac{dx}{(1+|x|^2)^\delta} \| g\|_{L^2(\R^n)}^{2}.
\end{align*}
Let $R:= diam(D)$. Then we split the latter integral over $\R^n$ into two parts, namely one in $B_{2R}$, and the other in $\R^n\setminus B_{2R}$. Note that when $x\in \R^n\setminus B_{2R}$, we have  that 
$$
|x-y| \geq |x| - |y| \geq \frac{|x|}{2} + \frac{2R}{2} - |y|\geq \frac{|x|}{2}.
$$
Therefore
\begin{align*}
   \int_{\R^n\setminus B_{2R}} \int_D \frac{dy}{|x-y|^{2\alpha}<x-y>^{2\beta}} \frac{dx}{(1+|x|^2)^\delta} \leq C(\alpha, \beta) \int_{\R^n\setminus B_{2R}}  \frac{|D|}{|x|^{2\alpha+2\beta+2\delta}} dx,
\end{align*}
which is is finite from the assumption that $2\alpha + 2\beta + 2\delta > n$. \\
For the second integral, we notice that $D\subset B_{3R}(x)$ for any $x\in B_{2R}$, and so 
\begin{align*}
   \int_{ B_{2R}} \int_D \frac{dy}{|x-y|^{2\alpha}<x-y>^{2\beta}} \frac{dx}{(1+|x|^2)^\delta} \leq \int_{ B_{2R}}  \int_{B_{3R}(x)}\frac{dy}{|x-y|^{2\alpha}} dx.
\end{align*}
The last integral  is finite thanks to the second assumption $2\alpha < n$. 
\end{proof}

\begin{corollary}\label{cor-above}
    Let $1/2<\delta<1$, and $n=1,2,3$.  Then for any $g\in L^2(\R^n)$ with compact support $D$, we have  that
  \begin{align*}
      \|G_{helm}^k*g\|_{L^{2,-\delta}} \leq  C(D,\delta, s,k)\|g\|_{L^2}(\R^n). 
  \end{align*}
\end{corollary}
\begin{proof}
By the previous lemma, for $n=1$ take $\alpha = \beta = 0$, for $n=2$  take $\alpha = \epsilon$ for some small $\epsilon>0$ (since a logarithmic singularity is better than $|x|^{-\epsilon}$ for any $\epsilon>0$) and  $\beta = \frac{1}{2}-\epsilon$, and for $n=3$ we take  $\alpha =1, \beta = 0$. 
\end{proof}
\noindent
Now we are ready to prove both propositions.
\begin{proof}[Proof of Proposition \ref{prop:Gconvf_L2delta}] For each dimension the Helmholtz part is taken care of by  Corollary \ref{cor-above}.  Hence we prove the estimate for the remaining terms of the fundamental solution in each case. 
\begin{itemize}
    \item For $n=1$, the remaining term is simply $J_1^{s,k}$. From Theorem \ref{th:asymptotics-all}, we see that $J_1^{s,k} \in L^1(\R)$, therefore using Young's inequality for the convolution, we obtain
    \begin{align*}
        \| J_1^{s,k}*g\|_{L^{2,-\delta}(\R)} \leq  \| J_1^{s,k}*g\|_{L^{2}(\R)} \leq C(s,k)  \| g\|_{L^{2}(\R)}.
    \end{align*}

    \item For $n=2$, we have terms of the form $|x|^{-2+2s(j+1)}$, $0\leq j\leq m-1$ and $|x|^{-1}$, which all are  treated the same way. We make use of the Hardy-Littlewood-Sobolev Theorem (see for example \cite[Theorem 4.3]{LiebLoss}) 
    \begin{align*}
         \| |x|^{-\lambda}*g\|_{L^2(\R^2)} \leq C(\lambda, n) \|g\|_{L^p(\R^2)},
    \end{align*}
  where $p$ satisfies 
  $$
 \frac{1}{p} + \frac{\lambda}{n} = 1 + \frac{1}{2}. 
  $$
  Choosing  $\lambda = 2 - 2s(j+1)$ we get $p = \frac{2}{1+2s(j+1)}\geq 1$. Then, an application of H\"older's inequality, and the fact that $g$ has compact support $D$ gives
  $$
  \| |x|^{-2+2s(j+1)}*g\|_{L^2(\R^2)} \leq C(s)\|g\|_{L^p(\R^2)} \leq C(D,s)\|g\|_{L^2(\R^2)}.
  $$
\item For $n=3$ and $s>1/2$, since $J_3^{s,k}$ is an $L^1(\R^3)$ function by Theorem \ref{th:asymptotics-all}, we can simply use Young's Inequality for convolution to obtain the result. For $s\leq 1/2$, we have additional terms of the form $|x|^{3-2s(j+1)}$, $0\leq j\leq m-1$, and $|x|^{-3+2sm}$ which all are treated the same way, and so we analyze only the latter. To this end, we use again Hardy-Littlewood-Sobolev Theorem and H\"older inequality to obtain 
    \begin{align*}
         \| |x|^{-3+2sm}*g\|_{L^2(\R^3)} \leq C(s) \|g\|_{L^p(\R^3)} \leq C(D,s) \|g\|_{L^2(\R^3)},
    \end{align*}
    where  we take $p = \frac{6}{3+4sm}\geq 1$. 
\end{itemize}  
\end{proof}

\begin{proof}[Proof of Proposition \ref{prop:Gepsconvf_L2delta}]
Using  the Hardy-Littlewood-Sobolev Theorem and Young's inequality for convolution we already saw that the terms other than the one corresponding to the Helmholtz fundamental solution define operators from $L^2(\R^n)$ to $L^2(\R^n)$, uniformly in $\epsilon$. Hence, it remains to check that the Helmholtz fundamental solution also defines an operator from $L^2(\R^n)$ to $L^2(\R^n)$ with a constant depending on $\epsilon$. This follows by an application of  Young's inequality for convolution, since $G_{n,helm}^{k_\epsilon} \in L^1(\R^n)$ due to the exponential decay for $\epsilon>0$. 
\end{proof}

\subsection{Almost Everywhere Estimates of the Fundamental Solution with Absorption }\label{5D}
This section is dedicated to proving point-wise convergence results of the fundamental solution to the fractional Helmholtz with $k_\epsilon^{2s}:=k^{2s}+i\epsilon$ as the  absorption $\epsilon\to 0^+$.
\begin{proposition}\label{prop-conv}
    Let $\frac{1}{2}<\delta<1$,  $n=1,2,3$,  $k_\epsilon^{2s}:=k^{2s}+i\epsilon$  with $k>0, \epsilon>0$ and $f\in L^{2,\delta}(\R^n)$. Then  for almost every $x\in \R^n$ we have
    $$
   |\mathcal{G}_{n,s}^{k_\epsilon}f(x) - \mathcal{G}_{n,s}^k f(x)| \to 0 \quad \mbox{ as } \epsilon \to 0. 
    $$
\end{proposition}
\begin{proof} For each dimension, we  consider  separately the part of the fundamental solution that corresponds to the Helmholtz equation and the remainder.
    \begin{itemize}
        \item For $n=1$ we have 
\begin{align}\label{estimate_Ghelm1}
&|\mathcal{G}_{1,helm}^{k_\epsilon}f(x) - \mathcal{G}_{1,helm}^{k}f(x) |=\left|\int_{\R} \frac{e^{i|x-y|k_\epsilon}}{2 s k_\epsilon^{2s-1}} f(y) dy-\int_{\R} \frac{e^{i|x-y|k}}{2s k^{2s-1}} f(y) dy\right| \nonumber\\
    &\leq \frac{1}{2 s} \left(\int_{\R} \left|\frac{e^{i|x-y|k_\epsilon}}{ k_\epsilon^{2s-1}}  - \frac{e^{i|x-y|k}}{ k^{2s-1}}\right|^2  \frac{dy}{(1+|y|^2)^\delta} \right)^{1/2} \|f\|_{L^{2,\delta}}
 \end{align}
 The above integral converges to $0$ as $\epsilon \to 0$ by the Dominated Convergence Theorem. 
Next, we consider  the term $J_1^{s,k}$ defined by (\ref{def:J1}). Simple calculations yield
\begin{align*}
   & |J_1^{s,k_\epsilon}(x) -J_1^{s,k}(x)| \leq \frac{C(s)}{|x|^{1-2s}}\int_0^\infty e^{-y} y^{2s}\left|\frac{P(y^{2s},k|x|) - P(y^{2s}, k_\epsilon |x|)}{P(y^{2s},k|x|) P(y^{2s},k_\epsilon|x|) }\right|\\
    &\, \leq \frac{C(k,s)|k^{2s}-k_\epsilon^{2s}|}{|x|^{1-2s}}\int_0^\infty e^{-y}y^{2s}|x|^{4s}\frac{|(k^{2s}+k_\epsilon^{2s})-2\cos(s\pi)y^{2s}|x|^{-2s}| }{|P(y^{2s},k|x|) P(y^{2s}, k_\epsilon|x|)|} dy,
\end{align*}
where $P(\cdot,\cdot)$  again is given by (\ref{Poly}). For some constant depending on $s,k$ we obtain
$$
|-2\cos(s\pi) |x|^{-2s}y^{2s} + (k^{2s}+k_\epsilon^{2s})| \leq C(k,s) P(y^{2s}|x|^{-2s}, k).
$$
Therefore, following the same arguments as in the proof of Theorem \ref{th:asymptotics-all}
\begin{align*}
|J_1^{s,k_\epsilon}(x) -J_1^{s,k}(x)| \leq \frac{O(\epsilon)}{|x|^{1-2s}}\int_0^\infty e^{-y}\frac{y^{2s}}{|P(y^{2s}, k_\epsilon|x|)|} dy \leq \frac{O(\epsilon)}{|x|^{1-2s}\left<x\right>^{4s}},
\end{align*}
We then conclude by applying Young's inequality for convolution
$$
\| (J_1^{s,k} - J_1^{s,k_\epsilon} )*f\|_{L^2} \leq O(\epsilon) \|f\|_{L^2}.
$$

\item For $n=2$ we have 
\begin{align}\label{ae_estimate_Ghelm2}
&|\mathcal{G}_{2,helm}^{k_\epsilon}f(x) - \mathcal{G}_{2,helm}^{k}f(x) |\nonumber\\
    &=C(s)\left|\int_{\R^2} \left(k_\epsilon^{2-2s}H^{(1)}_0(k_\epsilon|x-y|) -k^{2-2s} H^{(1)}_0(k|x-y|) \right)f(y) dy\right|\nonumber\\
    & \leq C(s)\left(\int_{\R^2} \frac{\left|k_\epsilon^{2-2s}H^{(1)}_0(k_\epsilon|x-y|) -k^{2-2s} H^{(1)}_0(k|x-y|) \right|^2}{(1+|y|^2)^{\delta}} dy\right)^{1/2}\|f\|_{L^{2,\delta}}.
 \end{align}
Continuity of the Hankel function implies that  for all $y\neq x$ 
\begin{align*}
    \lim_{\epsilon \to 0} (k_\epsilon^{2-2s}H^{(1)}_0(k_\epsilon|x-y|) -k^{2-2s} H^{(1)}_0(k|x-y|) = 0.
\end{align*}
Thanks to  rate of decay at infinity of $H^{(1)}_0$, the integrand in (\ref{ae_estimate_Ghelm2}) is uniformly bounded by an integrable function, that is
\begin{align}\label{bound_H1_0}
    \frac{\left|k_\epsilon^{2-2s}H^{(1)}_0(k_\epsilon|x-y|) -k^{2-2s} H^{(1)}_0(k|x-y|) \right|^2}{(1+|y|^2)^{\delta}} \leq C(k) \frac{|\ln(k|x-y|)|^2}{|x-y|(1+|y|^2)^{\delta}}.
\end{align}
Therefore, by the Dominated Convergence Theorem, $\mathcal{G}_{2,helm}^{k_\epsilon}f$ converges almost everywhere to $\mathcal{G}_{2,helm}^{k}f$. The terms of the form $|x|^{2-2s(j+1)}$  multiplied by $k_\epsilon^{2sj}$ are obviously continuous in $\epsilon$. Next, recall that from the computations in the proof of Theorem  \ref{th:asymptotics-all}, we have that in each case of $s\in (\frac{1}{2}, 1)$ and  $s\in(0,\frac{1}{2}]$ with $\frac{1}{2s}\in \mathbb{N}$ or $\frac{1}{2s}\not\in\mathbb{N}$, the functions $F_m(|\xi|,k)$ and $\tilde F_m(|\xi|,k)$ are in $L^2(\R^2)$ with $\frac{\partial F_m}{\partial r}(r,k), \frac{\partial \tilde F_m}{\partial r}(r,k) \in L^1(\R^2)$. Therefore
\begin{align*}
    \left| J_2^{s,k}(x) - J_2^{s,k_\epsilon}(x)\right| \leq 2|x|^{-1} \| \partial_r F_m(r,k) - \partial_r F_m(r,k_\epsilon)\|_{L^1(\R^2)}.
\end{align*}
Furthermore, $\partial_r F_m(r,k_\epsilon) - \partial_r F_m(r,k) \to 0$ as $\epsilon \to 0$ almost everywhere. Hence, by the Dominated Convergence Theorem $\| \partial_r F_m(r,k) - \partial_r F_m(r,k_\epsilon)\|_{L^1(\R^2)} \to 0$. By the Hardy-Littlewood-Sobolev inequality, this is implies almost everywhere convergence, since we can estimate
\begin{align*}
    \left\Vert \left( J_2^{s,k} - J_2^{s,k_\epsilon}\right)*f\right\Vert_{L^r(\R^2)} \leq C \| \partial_r F_m(r,k_\epsilon) - \partial_r F_m(r,k)\|_{L^1(\R^2)} \|f\|_{L^q(\R^2)}
\end{align*}
where $\frac{2}{\delta + 1}<q <2$ and $r = \frac{2q}{2-q} \in (\frac{2}{\delta}, \infty)$. The analysis for $\tilde F_m$ can be done in a similar way.
\item  For $n=3$ we fist estimate
\begin{align}\label{estimate_Ghelm3}
&|\mathcal{G}_{3,helm}^{k_\epsilon}f(x) - \mathcal{G}_{3,helm}^{k}f(x) | =\left|\int_{\R^3} \frac{e^{i|x-y|k_\epsilon}}{4\pi s k_\epsilon^{2s-2}|x-y|} f(y) dy-\int_{\R^3} \frac{e^{i|x-y|k}}{4\pi s k^{2s-2}|x-y|} f(y) dy\right| \nonumber\\
    &\qquad \qquad  \leq \frac{1}{4\pi s}  \left(\int_{\R^3} \left|\frac{e^{i|x-y|k_\epsilon}}{ k_\epsilon^{2s-2}}  - \frac{e^{i|x-y|k}}{ k^{2s-2}}\right|^2 \frac{dy}{|x-y|^2(1+|y|^2)^\delta} \right)^{1/2} \|f\|_{L^{2,\delta}}.
 \end{align}
 The above integral converges to $0$ as $\epsilon \to 0$ by the Dominated Convergence Theorem. Again, the terms $|x|^{-3+2s(j+1)}$ only depend on $\epsilon$ through the factors $k_\epsilon^{2sj}$ and hence convergence as $\epsilon\to 0$ is obvious.  We first consider the term $J_3^{s,k}$ when $s>1/2$. To this end,
\begin{align*}
   & |J_3^{s,k_\epsilon}(x) -J_3^{s,k}(x)| \leq \frac{C(s)}{|x|^{3-2s}}\int_0^\infty e^{-y} y^{1+2s}\left|\frac{ P(y^{2s},k|x|) - P(y^{2s}, k_\epsilon |x|)}{P(y^{2s},k|x|) P(y^{2s},k_\epsilon|x|) }\right|.\\
    &\qquad \leq \frac{C(k,s)|k^{2s}-k_\epsilon^{2s}|}{|x|^{3-2s}}\int_0^\infty e^{-y}y^{2s}|x|^{4s}\frac{|(k^{2s}+k_\epsilon^{2s})-2\cos(s\pi)y^{2s}|x|^{-2s}| }{|P(y^{2s},k|x|) P(y^{2s}, k_\epsilon|x|)|} dy. 
\end{align*}
By a similar analysis as in the one-dimensional case, we can conclude that  
\begin{align}\label{estimate_J3}
    |J_3^{k_\epsilon, s}(x) -J_3^{k,s}(x) |  \leq \frac{O(\epsilon)}{|x|^{3-2s}\left<x\right>^{4s}}.
\end{align}
We now turn to $J_3^{s,k}$ when $s\leq 1/2$. In this case, we have  
\begin{align*}
&\left| J_3^{s,k_\epsilon} (x) - J_3^{s,k} (x) \right| \nonumber \\
& \qquad \leq  \frac{C(s)}{ |x|^{3-2s(m+1)}}\int_0^\infty e^{-y}y^{1-2sm} \left| \frac{k_\epsilon^{2sm}(y^{2s}-k_\epsilon^{2s}|x|^{2s})}{P(y^{2s}, k_\epsilon|x|)} - \frac{k^{2sm}(y^{2s}-k^{2s}|x|^{2s})}{P(y^{2s}, k|x|)}\right|dy
\end{align*}
Adding  and subtracting  $\frac{k^{2sm}(y^{2s} - k^{2s}|x|^{2s})}{P(y^{2s}, k_\epsilon |x|)}$, and performing  similar computations as in the proof of Theorem \ref{th:asymptotics-all}, we get 
\begin{align*}
 \left| J_3^{s,k_\epsilon} (x) - J_3^{s,k} (x) \right|\leq    \frac{O(\epsilon)}{|x|^{3-2s(m+1)}}\left( \int_0^\infty e^{-y}y^{1-2sm} \frac{y^{2s}+|x|^{2s}+1}{|P(y^{2s}, k_\epsilon|x|)|} \right) \leq \frac{O(\epsilon)}{|x|^{3-2sm}}.
\end{align*}
Finally, an application of the Hardy-Littlewood-Sobolev inequality  gives
$$
\left\Vert \left( J_3^{s,k_\epsilon}  - J_3^{s,k}  \right)*f \right\Vert_{L^r(\R^3)} \leq O(\epsilon) \|f\|_{L^2(\R^3)}
$$
where $r = \frac{6}{3-4sm}$. 
 \end{itemize}
    \end{proof}
 
 \subsection{Fourier Multipliers : properties of \texorpdfstring{$M(\xi;z)$}{TEXT} and a useful lemma} \label{appendixA}
 The results proven in this section are fundamental to the analysis of Section \ref{3}.
\begin{lemma} \label{boundM} Let $z$ be in a bounded region $\mathcal{O}$ as described in Lemma \ref{lemma:decomposition_resolvent}. Recall that we defined the Fourier multiplier $M(\xi;z)$ to be 
\begin{align*}
    M(\xi;z) := \frac{z^{2s}|\xi|^{2-2s}-z^{2-2s}|\xi|^{2s}}{|\xi|^{2s}-z^{2s}}
\end{align*}
and we write $M(z)\phi := \F^{-1} M(\xi;z) \F \phi $ for all $\phi \in \Sw$. Then, for some positive constants $C, C'$ depending on $\mathcal{O}$, $M(\xi;z)$  and the operator $M(z)$ satisfy
\begin{itemize}
    \item[(i)] $
        |M(\xi;z)| \leq C \left<|\xi|\right>^{2-4s}$ for all $\xi \in \R^n$. In particular, when $s>1/2$, $M(\xi;z)$ is bounded. 
    \item[(ii)] $
        |\partial_{\xi_i} M(\xi;z) | \leq C' |\xi|^{-1+2s}\left<|\xi|\right>^{2-6s}
    $ for all $1\leq i \leq n$, $\xi\in\R^n$. In particular, when $s>1/2$, $|\nabla M(\xi;z)|$ is bounded. 
    \item[(iii)] Let $\phi \in \Sw(\R^n)$ and $\delta \leq 1$. Then $M(z)\phi \in L^{2,\delta}(\R^n)$ for $n=2,3$ when $s\in (0,1)$, and for $n=1$ when $s\in (1/4, 1)$. 
\end{itemize}
\end{lemma}
\begin{proof}
(i) Note that as $|\xi|\to \infty$, $M(\xi;z)$ grows like $|\xi|^{2-4s}$. To complete the proof of (i), we need to show that the singularity at $z = k\in \mathcal{O}\cap \R$ is removable. This can be seen by performing a Taylor expansion around $|\xi_0| = k$ for $h$ small, $\theta \in \mathbb{S}^1$, and  $0<\alpha<1$ 
\begin{align}\label{Taylor_alpha}
    |\xi_0 + h \theta|^{2\alpha} &= (k^2 + 2 \xi_0 \cdot \theta h + h^2)^\alpha \nonumber\\
    &=k^{2\alpha}\left(1 + 2\alpha \frac{\xi_0 \cdot \theta h}{k^2}+  \alpha \frac{h^2}{k^2} + 2\alpha (\alpha-1) \frac{(\xi_0\cdot \theta)^2 h^2}{k^4} + o(h^2)\right).
\end{align}
Then, using the above formula \ref{Taylor_alpha} with $\alpha = s$ and $1-s$, we obtain the following expansion of $M(\xi; k)$ around $|\xi_0 | = k$ 
\begin{align*}
    M(\xi_0 + h \theta;k) =  \frac{2(1-2s) \xi_0 \cdot \theta h+  (1-2s) h^2 + o(h^2)}{k^{2s-2}\left( 2s \xi_0 \cdot \theta h+  s h^2 + 2s(s-1) \frac{(\xi_0\cdot \theta)^2 h^2}{k^2} + o(h^2)\right)} \xrightarrow[h \to 0]{} \frac{(1-2s)}{k^{2s-2}s}.
\end{align*}
The limit holds regardless of  whether $\xi_0 \cdot \theta  = 0$ or not. \\
    (ii) We now compute $\partial_{\xi_i} M(\xi;z)$ and study its behavior at infinity and its singularities.  Define the function $H(x,z)$ by $H(|\xi|^{2s},z)=M(\xi,z)$, then  
\begin{align*}
    H(x,z) = \frac{z^{2s}x^{1/s - 1} - z^{2-2s} x}{x-z^{2s}} \mbox{ and } 
 \frac{\partial}{\partial x}H(x,z) = \frac{z^{2s}(1/s - 2) x^{1/s - 1} - z^{4s}(1/s - 1) x^{1/s-2} + z^2}{(x-z^{2s})^2}.
\end{align*}
Therefore 
\begin{align}\label{derivM}
    \partial_{\xi_i} M(\xi;z) &= \frac{2s \xi_i}{|\xi|^{2-2s}} \frac{\partial}{\partial x}H(|\xi|^{2s},z)\nonumber\\
    &= \frac{2s \xi_i}{|\xi|^{2-2s}}\frac{z^{2s}(1/s - 2) |\xi|^{2 - 2s} - z^{4s}(1/s - 1) |\xi|^{2-4s} + z^2}{(|\xi|^{2s}-z^{2s})^2}.
\end{align}
As before, we observe that the singularity around $z=k$ is removable by mean of a Taylor expansion around  $|\xi_0| = k $ with $\theta \in \mathbb{S}^1$, and use of (\ref{Taylor_alpha}). Expanding the powers in the numerator in (\ref{derivM}) by taking $\alpha = 1-s$ and $\alpha = 1-2s$ in (\ref{Taylor_alpha}) we obtain
\begin{align*}
    &k^{2s} \left(\frac{1}{s} - 2\right) |\xi|^{2-2s} \nonumber\\
    &\hskip 10pt =\,k^2\left(\frac{1}{s} - 2\right) \left(1 + 2(1-s)\ \frac{\xi_0 \cdot \theta h}{k^2}+   (1-s) \frac{h^2}{k^2} + 2(1-s) (-s) \frac{(\xi_0\cdot\theta)^2 h^2}{k^4} + o(h^2)\right),\\
     &k^{4s} \left(\frac{1}{s}- 1\right) |\xi|^{2-4s}  \nonumber\\
    &\hskip 10pt = k^2\left(\frac{1}{s} - 1\right) \left(1 + 2(1-2s)\ \frac{\xi_0 \cdot \theta h}{k^2}+  (1-2s) \frac{h^2}{k^2} + 2(1-2s) (-2s) \frac{(\xi_0\cdot\theta)^2 h^2}{k^4} + o(h^2)\right).
\end{align*}
We observe 
\begin{align*}
   \frac{ k^{2s}(1/s - 2) |\xi|^{2 - 2s} - k^{4s}(1/s - 1) |\xi|^{2-4s} + k^2}{(|\xi|^{2s} - k^{2s})} &= \frac{2(1/s -2)(1-s) \frac{s(\xi_0 \cdot \theta)^2}{k^2}h^2 + o(h^2)}{k^{4s}4s^2\frac{(\xi_0 \cdot \theta)^2}{k^4} h^2 + o(h)} \\
   \\
   &\xrightarrow[h\to 0]{} \frac{(1/s -2)(1-s) k^{2-4s}}{2s},
\end{align*}
when $\xi_0 \cdot \theta \neq 0$. If $\xi_0 \cdot\theta = 0$, the terms in the numerator still cancel up to order four, and the fourth order term  of $|\xi_0 + h \theta|^{2\alpha}$ takes the form $k^{2\alpha}\frac{\alpha(\alpha -1)}{2} \frac{h^4}{k^4}$. So when $\xi_0 \cdot\theta = 0$ we have 
\begin{align*}
    \frac{ k^{2s}(1/s - 2) |\xi|^{2 - 2s} - k^{4s}(1/s - 1) |\xi|^{2-4s} + k^2}{(|\xi|^{2s} - k^{2s})} &= \frac{1/2(1/s -2)(1-s) s k^{-2}  h^4 +o(h^4)}{k^{4s-4}s^2 h^4 + o(h^4)} \\
   \\
   &\xrightarrow[h\to 0]{} \frac{(1/s -2)(1-s) k^{2-4s}}{2s},
\end{align*}
which coincides with the case when $\xi_0 \cdot \theta \neq 0$.  Thus we deduce the bounds stated in (ii).

\noindent
(iii) Here it suffices to consider  the case  $\delta = 1$. To show  that $M(z)\phi$ is in $L^{2,1}(\R^n)$ we must  prove  that $M(z)\phi$ is in $L^2(\R^n)$ and  that $|x|M(z)\phi$ is in $L^2$. The former follows from use of Plancherel's formula, the result in Lemma \ref{boundM} (i), and  the fact that Schwarz functions are closed under Fourier transform. We estimate 
\begin{align}\label{proof:MphiL^2}
    \| M(z) \phi \|_{L^{2}(\R^n)}^2 
    &\leq C \int_{\R^n} \left<|\xi|\right>^{4-8s}|\F\phi(\xi)|^2 dx \nonumber\\
    &\leq C \left(\sup_{\xi\in \R^n} \left<|\xi|\right>^{2-4s + \frac{n+1}{2}} |\F \phi(\xi)|\right)^2 \int_{\R^n}\left <|\xi|\right >^{-n-1}dx <\infty.
\end{align} 
To show $|x|M(z)\phi$ is in $L^2$, we compute as follows  
\begin{align*}
    |x|^2 \left| \F^{-1} M(\xi;z) \F \phi (x)\right|^2 &= \left(\frac{1}{2\pi}\right)^{n} \sum_{i=1}^n\left| \int_{\R^n} x_i e^{i\xi \cdot x} M(\xi;z) \F \phi (\xi) d\xi \right|^2\\
    &=\left(\frac{1}{2\pi}\right)^{n} \sum_{i=1}^n\left| \int_{\R^n} \partial_{\xi_i} ( e^{i\xi \cdot x}) M(\xi;z) \F \phi (\xi) d\xi \right|^2\\
    &= \left(\frac{1}{2\pi}\right)^{n} \sum_{i=1}^n\left| \int_{\R^n}  e^{i\xi \cdot x} \left[ \partial_{\xi_i}M(\xi;z) \F \phi (\xi) + M(\xi;z)  \partial_{\xi_i}\F\phi(\xi) \right]d\xi \right|^2\\
    &\leq 2\sum_{i=1}^n |\F^{-1} \partial_{\xi_i}M(\xi;z) \F \phi(\xi)|^2 +  |\F^{-1} M(\xi;z) \partial_{\xi_i}\F \phi(\xi)|^2~.
\end{align*}
The term $|\F^{-1} M(\xi;z) \partial_{\xi_i}\F \phi(\xi)|^2$ is again in $L^2(\R^n)$ by the same argument we used in  (\ref{proof:MphiL^2}), because $\partial_{\xi_i} \F \phi$ is again a Schwarz function.  For the term involving the derivative of $M$, namely $|\F^{-1} \partial_{\xi_i}M(\xi;z) \F \phi(\xi)|^2 $, we make use of Plancherel's theorem and Lemma \ref{boundM} (ii)  to obtain 
\begin{align*}
 & \int_{\R^n} |\partial_{\xi_i}M(\xi;z) \F \phi (\xi)|^2 \leq  C\int_{\R^n}  \frac{\left<|\xi|\right>^{4-12s}}{|\xi|^{2-4s}}| \F \phi (\xi)|^2 \\
    &\qquad \leq C\left( \sup_{\xi\in \R^n} \left<|\xi|\right>^{2-6s + \frac{n+2}{2}}|\F\phi(\xi)|\right)^2\int_{\R^n} \frac{d\xi}{|\xi|^{2-4s}<|\xi|>^{n+2}}<\infty \mbox{ if } 2-4s<n. 
\end{align*}
This concludes the proof of item  (iii). 
\end{proof}

\begin{definition} A $C^\infty$-function $p(x, \xi)$ on $\R^n \times \R^n$ is said to be in the class
$\Sw^{\mu}_{0,0}$ $( \mu \in \R)$ if for any pair $\alpha$ and $\beta$ of multi-indices there exists a constant $C_{\alpha, \beta}\geq 0$ such that
\begin{align*}
    \left| \left(\frac{\partial}{\partial \xi} \right)^{\alpha}\left(\frac{\partial}{\partial x}\right)^{\beta} p(x,\xi) \right| \leq C_{\alpha,\beta} \langle \xi\rangle^\mu
\end{align*}
The class $\Sw^{\mu}_{0,0}$ is a Fr\'echet space equipped with the seminorm
\begin{align}\label{norm}
    |p|^{(\mu)}_\ell  := \max_{|\alpha|, |\beta|\leq \ell}\sup_{x,\xi} \left| \langle\xi \rangle^{-\mu} \left(\frac{\partial}{\partial \xi} \right)^{\alpha}\left(\frac{\partial}{\partial x}\right)^{\beta} p(x,\xi) \right|, \qquad \ell=0,1,2 \dots
\end{align}
\end{definition}
\noindent
Define  the pseudodifferential operator $p(x,D)u:={\mathcal F}^{-1}p(x,\xi){\mathcal F}u$.
\begin{lemma}[Lemma 3.1 in \cite{umeda2003generalized}]\label{lemma_31_Umeda}
Let $p(x,\xi)$ belong to $\Sw^{-m}_{0,0}$ for some integer $m \geq 0$, and let $\delta \in \R$. Then there exists a nonnegative constant $C = C_{m,\delta}$ and a positive integer $\ell=\ell_{m,\delta}$ such that 
\begin{align*}
    \|p(x, D)u\|_{H^{m,\delta}} \leq C |p|^{(-m)}_\ell \|u\|_{L^{2,\delta}} \quad \mbox{for all } u \in  \Sw(\R^n).
\end{align*}
\end{lemma}
\noindent
\begin{remark}\label{hoer}
The above Lemma holds true for any $m>0$ (not just integers) and we refer the reader to H{\"o}rmander \cite[Chapter 18.1]{HormanderIII}. The case of symbols in $\Sw^{0}_{0,0}$, i.e., bounded, yielding bounded operators in $L^{2,\delta}$ is due Calder{\'o}n–Vaillancourt.
\end{remark}

\subsection{Equivalence of SRC and GSRC for the Helmholtz Equation}\label{equiv}
Here we prove that for a solution to the homogeneous Helmholtz equation in the exterior of a bounded region,  the classical Sommerfeld Radiation Condition (SRC)  (\ref{def:SRC}) and the Generalized Sommerfeld Radiation Condition (GSRC) (\ref{def:generalized_SRC-I}) are equivalent.

\begin{theorem}
    Let $k>0$, $n=2,3$ and let $u\in H^1_{loc}(\Omega)$ be a solution of the Helmholtz equation $\Delta u +k^2u=0$ in $\Omega  = \R^n \setminus D$. Then the following radiation conditions are equivalent 
    \begin{align}\label{SRC}
        \lim_{r\to \infty} r^{\frac{n-1}{2}}|\partial_r u(r,\hat x) -ik u(r,\hat x)| = 0  \mbox{ uniformly in }\hat x = \frac{x}{|x|} \in \mathbb{S}^{n-1},\tag{SRC}
    \end{align}
    \begin{align}
        \int_{\Omega} (1+|x|^2)^{\delta-1} |\nabla u -ik u \hat x|^2 dx  <\infty   \mbox{ for some } \delta \in (1/2,1) \tag{GSRC} 
    \end{align}
    where $r= |x|$.
\end{theorem}
\begin{proof} We prove the result for $n=2$. The case $n=3$ is similar. First, we prove that (\ref{GSRC}) implies (\ref{SRC}). Note that since $u$ satisfies the Helmholtz equation in the exterior of a ball, we can always expand $u$ in normalized spherical Harmonics $\{Y_{m}(\hat x)\} = \{ \frac{1}{\sqrt{2\pi}},\frac{1}{\sqrt{\pi}} \cos(m\hat x), \frac{1}{\sqrt{\pi}} \sin(m \pi), m\geq 1\}$. By orthogonality, it suffices to consider only the cosine modes,  and hence  we write $u(r,\hat x) = \sum_{m =0}^{\infty} u_{m}(r) Y_m(\hat x)$, with
$Y_m(\hat x) = \frac{1}{\sqrt{\pi}}\cos(m\hat x)$,  $r\in(R, \infty)$, and $\hat x \in [0,2\pi)$. The coefficients will then satisfy 
$$
u_{m}''(r) + \frac{1}{r} u_{m}'(r) - \frac{m^2}{r^2} u_{m}(r) + k^2 u_{m}(r) = 0 ~.
$$
The solution to the above equation is given as a linear combination of Hankel functions $H^{(j)}_{m}$
 $$u_{m}(r) = \alpha_{m} H^{(1)}_{m}(kr) + \beta_{m}H^{(2)}_{m}(kr).$$ Therefore $u$ can be written as 
\begin{align}\label{decomposition_u}
u(r, \hat x) = \sum_{m=0}^{\infty} \left(\alpha_{m} H^{(1)}_{m}(kr) + \beta_{m}H^{(2)}_{m}(kr)\right)Y_m(\hat x).
\end{align}
Recall the asymptotics of the Hankel functions and their derivative when $r$ is large 
\begin{align*}
    H^{(1)}_{m}(kr) &= \sqrt{\frac{2}{\pi kr}} e^{i(kr- \frac{\pi m}{2} - \frac{\pi}{4})}\left(1+\frac{i(4m^2-1)}{8kr} +O(r^{-2})\right),\\
    H^{(2)}_{m}(kr) &= \sqrt{\frac{2}{\pi kr}} e^{-i(kr- \frac{\pi m}{2} - \frac{\pi}{4})}\left(1+O(r^{-1})\right)\\
     \frac{\partial}{\partial r}H^{(1)}_{m}(kr) &= ik \sqrt{\frac{2}{\pi kr}} e^{i(kr- \frac{\pi m}{2} - \frac{\pi}{4})} \left(1+\frac{i(4m^2+3)}{8kr}+O(r^{-2})\right),\\
      \frac{\partial}{\partial r}H^{(1)}_{m}(kr) &= -ik \sqrt{\frac{2}{\pi kr}} e^{-i(kr- \frac{\pi m}{2} - \frac{\pi}{4})} (1+O(r^{-1})).
\end{align*}
Because of the different sign in the exponential, a cancellation takes place for $H^{(1)}$ that does not occur for $H^{(2)}$ in the following
\begin{align}
    \partial_r H^{(1)}_{m}(kr) - ik H^{(1)}_{m}(kr) &= -\frac{1}{\sqrt{2\pi k}} r^{-\frac{3}{2}} e^{i (kr - \frac{\pi m}{2} - \frac{\pi}{4})} ( 1+ O(r^{-1}) ) \nonumber\\
    &= : c^1_{k} r^{-\frac{3}{2}} e^{i  (kr - \frac{\pi m}{2})} ( 1+ O(r^{-1}) ),  \label{H1_asympt}\\
    \nonumber\\
    \partial_r H^{(2)}_{m}(kr) - ik H^{(2)}_{m}(kr)&= - 2i \sqrt{\frac{2k}{\pi}} r^{- \frac{1}{2}} e^{-i(kr - \frac{\pi m}{2} - \frac{\pi}{4})}(1+O(r^{-1}))\nonumber\\
    &=: c^2_{k} r^{-\frac{1}{2}} e^{-i  (kr - \frac{\pi m}{2})} ( 1+ O(r^{-1}) ). \label{H2_asympt}
\end{align}
Using the orthonormality of $Y_{m}$ in $L^2([0,2\pi])$, we now compute 
\begin{align}\label{integrandSRC}
    &r\int_{0}^{2\pi}|\partial_r u (r,\hat x)- ik u(r,\hat x)|^2 d\hat x  \nonumber \\
 &\hskip 10pt \geq r \sum_{m=0}^M  \left|\alpha_{m}c^1_{k} r^{-\frac{3}{2}} e^{i (kr - \frac{\pi m}{2})} ( 1+ O(r^{-1}) )+ \beta_{m}c^2_{k} r^{-\frac{1}{2}} e^{-i (kr - \frac{\pi m}{2})} ( 1+ O(r^{-1}) )\right|^2 \nonumber\\
& \hskip 10pt = |c_{k}^2|^2\left(\sum_{m=0}^M   |\beta_{m}|^2 \right) + O_M(r^{-1})
\end{align}

for any $M \in \mathbb{N}$. Let now  $\frac{1}{2}<\delta< 1$. 
 By (\ref{integrandSRC}), we see that 
\begin{align*}
     \int_{\Omega} (1+|x|^2)^{\delta-1} |\nabla u -ik \hat xu|^2 dx &= \int_R^\infty (1+r^2)^{\delta-1} \int_0^{2\pi}|\partial_ru -ik u|^2 +\frac{1}{r^2}|\partial_\theta u|^2d\hat x rdr\\
     & \geq \int_R^\infty (1+r^2)^{\delta-1} \left(|c_{k}^2|^2 \left(\sum_{m=0}^M   |\beta_{m}|^2 \right) + O_M(r^{-1})\right) dr
\end{align*}
for any $M \in \mathbb{N}$. The first order term is not integrable at infinity since $ \delta > \frac{1}{2} \implies 2(\delta - 1) >-1$. The other terms cannot compensate since they are absolutely integrable, $\delta <1 \implies 2(\delta-1)-1 <-1$. We deduce that 
$$
u \mbox{ satisfies (\ref{GSRC}) } \implies \beta_{m} = 0 \quad \forall m\geq 0 .
$$
Thanks to this observation, we can now show that 
$$
\lim_{r\to 0} r\int_{0}^{2\pi}|\partial_ru -iku|^2 d\hat x = \lim_{r\to \infty} r\sum_{m=0 }^\infty |\alpha_m|^2 \left| (\partial_r -ik) H_m^{(1)}(kr)\right|^2=0
$$ 
which implies (\ref{SRC}). Let $\epsilon >0$. We split the sum into two parts
$$
r \int_{0}^{2\pi}|\partial_r u-ik u|^2 d\hat x = S_N(r) + S_N^\infty(r) 
$$
with 
\begin{align*}
    S_N(r) &:= r\sum_{m=0}^{N-1} |\alpha_m|^2 \left| (\partial_r -ik) H_m^{(1)}(kr)\right|^2\\
     S_N^\infty(r) &:= r\sum_{m=N}^{\infty} |\alpha_m|^2 \left| (\partial_r -ik) H_m^{(1)}(kr)\right|^2
\end{align*}
For the tail $S_N^\infty(r)$, we use the following property 
\begin{align*}
    \frac{|H^{(1)}_m(r)|^2}{|H^{(1)}_m(s)|^2} \leq \frac{s}{r}, \quad \mbox{for }0< s \leq r, m\not = 0
\end{align*}
which can be found in \cite{watson1922treatise} p. 446 (see also \cite{hansen2007asymptotically} estimate (24)). We then have the bound for $m \geq 2$
\begin{align*}
    r| (\partial_r-ik) H_m^{(1)}(kr)|^2 &= r\left|  k H^{(1)}_{m- 1}(kr) - ( \frac{m}{r}+ik) H^{(1)}_m(kr)\right|^2\\
    &\leq 2r\left(\frac{m^2}{r^2}+k^2 \right) \left(\left| H^{(1)}_{m-1}(kr)\right|^2 +\left| H^{(1)}_m(kr)\right|^2\right)\\
    &\leq 4 R k^2m^2 \left(\left| H^{(1)}_{m-1}(kR)\right|^2 +\left| H^{(1)}_m(kR)\right|^2\right)
\end{align*}
where, for the purpose of the  last inequality, we suppose $r$ sufficiently large such that $\max (R,\frac{1}{k})\le r$. The behavior of the Hankel function for large order $m$ is given by
$$
H^{(1)}_m (kR)  \sim_{m\to +\infty} -i \frac{(m-1)!}{\pi} \left( \frac{2}{kR}\right)^m.
$$
Therefore we observe that $|H_{m-1}^{(1)}(kR)| $ is not worse than $|H_{m}^{(1)}(kR)|$ asymptotically, so we only need to justify that $\sum_{m=0}^M m^2|\alpha_m|^2|H^{(1)}_m(kR)|^2$ converges. This is indeed the case since, by elliptic regularity, $u$ is smooth on the circle of radius $kR$
$$
\sum_{m=0}^\infty m^2|\alpha_m|^2 |H_m^{(1)}(kR)|^2 = \int_0^{2\pi}|\partial_\theta u(kR, \theta)|^2 d\theta <\infty ~. $$ 
For a fixed $R$ we can thus take $N\in\mathbb{N}$ sufficiently large that
\begin{align*}
    S_N^\infty(r) \leq 4Rk^2\sum_{m= N}^\infty m^2|\alpha_m|^2\left( \left| H^{(1)}_{m-1}(kR)\right|^2 +\left| H^{(1)}_m(kR)\right|^2\right)<\frac{\epsilon}{2} ,
\end{align*}
for all $r \ge \max (R,\frac1k)$. We then use (\ref{H1_asympt}) and choose $r$ sufficiently large such that $S_N(r) \leq \frac{\epsilon}{2}$. This concludes that 
$$
\lim_{r\to \infty} r \int_{0}^{2\pi}|\partial_r u -ik u |^2 d\hat x = 0 ~,
$$
hence we showed that (\ref{GSRC}) implies (\ref{SRC}). \\
For the other direction, let us assume that $u \in H^1_{loc}(\Omega)$ is a solution of the Helmholtz equation in $\Omega$ satisfying (\ref{SRC}). By enlarging $D$ if necessary, we can define a smooth extension of $u$
\begin{align*}
    \bar u := \begin{cases}
    u \mbox{ in } \R^n\setminus D\\
    w \mbox{ in } D
\end{cases} \qquad \mbox{ which satisfies } \qquad (\Delta + k^2 ) \bar u  =  f \quad \mbox{ in } \R^n
\end{align*} 
where $f$ is a smooth compactly supported function. Let $u^+ = \lim_{\epsilon \to 0}(\Delta + k^2 + i \epsilon)^{-1} f$ be given by the limiting absorption principle as in Theorem 1.4 of \cite{ikebe1972limiting}. By this same theorem, $u^+$ satisfies (\ref{GSRC}), and as we just saw, this in turn implies that $u^+$ satisfies (\ref{SRC}). Finally, $\bar u - u^+ $ satisfies the homogeneous Helmholtz equation in $\R^n$ together with (\ref{SRC}), therefore $\bar u = u^+$. We conclude that $u$ satisfies (\ref{GSRC}). 
\end{proof}
\begin{remark}
    In the one dimensional case, the solution has the form 
    $$
   u(x) = c_1 e^{ik x} + c_2 e^{-ikx}.
    $$
    When $x>0$, it is trivial to see that to satisfy either of the radiation conditions, $c_2 = 0$. Similarly when $x<0$, we must have $c_1 = 0$ for either of the conditions to be satisfied. Hence we conclude that either radiation condition is equivalent to
    $$
    u(x) = \begin{cases}
        c_1 e^{ikx} \hbox{ for  sufficiently large positive } x\\
        c_2 e^{-ikx} \hbox{ for  sufficiently small negative } x
    \end{cases}.
    $$
\end{remark}

\end{subappendices}

\section*{Acknowledgements}
The work of FC and DZ was partially  supported by NSF grant DMS-24-06313. The work of MSV was partially supported by NSF grant DMS-22-05912.
\bibliographystyle{plain}
\bibliography{sample}

@article{guan2023helmholtz,
  title={Helmholtz solutions for the fractional Laplacian and other related operators},
  author={Guan, Vincent and Murugan, Mathav and Wei, Juncheng},
  journal={Communications in Contemporary Mathematics},
  volume={25},
  number={02},
  pages={2250016},
  year={2023},
  publisher={World Scientific}
}

@article{bucur2015some,
  title={Some observations on the Green function for the ball in the fractional Laplace framework},
  author={Bucur, Claudia},
  journal={arXiv preprint arXiv:1502.06468},
  year={2015}
}

@inproceedings{umeda1995radiation,
  title={Radiation conditions and resolvent estimates for relativistic Schr{\"o}dinger operators},
  author={Umeda, Tomio},
  booktitle={Annales de l'IHP Physique th{\'e}orique},
  volume={63},
  number={3},
  pages={277--296},
  year={1995}
}

@inproceedings{ben1997remarks,
  title={Remarks on relativistic Schr{\"o}dinger operators and their extensions},
  author={Ben-Artzi, Matania and Nemirovsky, Jonathan},
  booktitle={Annales de l'IHP Physique th{\'e}orique},
  volume={67},
  number={1},
  pages={29--39},
  year={1997}
}

@article{hiltunen2024nonlocal,
  title={Nonlocal Partial Differential Equations and Quantum Optics: Bound States and Resonances},
  author={Hiltunen, Erik Orvehed and Kraisler, Joseph and Schotland, John C. and Weinstein, Michael I.},
  journal={SIAM Journal on Mathematical Analysis},
  volume={56},
  number={3},
  pages={3802--3831},
  year={2024},
  publisher={SIAM}
}

@article{di2012hitchhikers,
  title={Hitchhiker's guide to the fractional Sobolev spaces},
  author={Di Nezza, Eleonora and Palatucci, Giampiero and Valdinoci, Enrico},
  journal={Bulletin des sciences math{\'e}matiques},
  volume={136},
  number={5},
  pages={521--573},
  year={2012},
  publisher={Elsevier}
}

@article{garofalo2017fractional,
  title={Fractional thoughts},
  author={Garofalo, Nicola},
  journal={arXiv preprint arXiv:1712.03347},
  year={2017}
}

@article{agmon1975spectral,
  title={Spectral properties of Schr{\"o}dinger operators and scattering theory},
  author={Agmon, Shmuel},
  journal={Annali della Scuola Normale Superiore di Pisa-Classe di Scienze},
  volume={2},
  number={2},
  pages={151--218},
  year={1975}
}

@book{ben1987limiting,
  title={The limiting absorption principle for partial differential operators},
  author={Ben-Artzi, Matania and Devinatz, Allen},
  volume={364},
  year={1987},
  publisher={American Mathematical Soc.}
}

@article {CS-extesion,
    AUTHOR = {Caffarelli, Luis and Silvestre, Luis},
     TITLE = {An extension problem related to the fractional {L}aplacian},
   JOURNAL = {Comm. Partial Differential Equations},
  FJOURNAL = {Communications in Partial Differential Equations},
    VOLUME = {32},
      YEAR = {2007},
    NUMBER = {7-9},
     PAGES = {1245--1260},
      ISSN = {0360-5302,1532-4133},
   MRCLASS = {35J70},
  MRNUMBER = {2354493},
MRREVIEWER = {Francesco\ Petitta},
       DOI = {10.1080/03605300600987306},
       URL = {https://doi.org/10.1080/03605300600987306},
}

@article {inverse,
    AUTHOR = {Ghosh, Tuhin and Salo, Mikko and Uhlmann, Gunther},
     TITLE = {The {C}alder\'on problem for the fractional {S}chr\"odinger
              equation},
   JOURNAL = {Anal. PDE},
  FJOURNAL = {Analysis \& PDE},
    VOLUME = {13},
      YEAR = {2020},
    NUMBER = {2},
     PAGES = {455--475},
      ISSN = {2157-5045,1948-206X},
   MRCLASS = {35R11 (26A33 35J10 35J70 35R30)},
  MRNUMBER = {4078233},
       DOI = {10.2140/apde.2020.13.455},
       URL = {https://doi.org/10.2140/apde.2020.13.455},
}

@article{inverse-scat,
  title={Recovering asymptotics of potentials from the scattering of relativistic {S}chr\"odinger operators},
  author={Uhlmann, Gunther and Yiran, Wang},
  journal={arXiv:2508.12463},
  volume={},
  year={2025},
  publisher={}
}

@article{inverse-scat2,
  title={Inverse scattering for the fractional Schrödinger equation},
  author={Das, Saumyajit and Ghosh, Tuhin and Ma, Shiqi},
  journal={arXiv preprint arXiv:2509.12685},
  year={2025}
}

@article {john1,
    AUTHOR = {Orvehed Hiltunen, Erik and Kraisler, Joseph and Schotland,
              John C. and Weinstein, Michael I.},
     TITLE = {Nonlocal partial differential equations and quantum optics:
              bound states and resonances},
   JOURNAL = {SIAM J. Math. Anal.},
  FJOURNAL = {SIAM Journal on Mathematical Analysis},
    VOLUME = {56},
      YEAR = {2024},
    NUMBER = {3},
     PAGES = {3802--3831},
      ISSN = {0036-1410,1095-7154},
   MRCLASS = {81Q10 (35P15 35S15 81V80)},
  MRNUMBER = {4753510},
MRREVIEWER = {Moorad\ Alexanian},
       DOI = {10.1137/23M158142X},
       URL = {https://doi.org/10.1137/23M158142X},
}

@article{john2,
  title={Quantum field theory and inverse problems: imaging with entangled photons},
  author={Lassas, Matti and Nursultanov, Medet and Oksanen, Lauri and Schotland, John C.},
  journal={arXiv:2506.03653},
  volume={},
  year={2025},
  publisher={}
}

@article{john3,
  title={Dynamic one photon localization in a discrete model of quantum optics},
  author={Kraisler, Joseph and Schenker, Jeffrey and Schotland, John C.},
  journal={arXiv preprint arXiv:2407.14109},
  volume={},
  year={2024},
  publisher={arXiv:2407.14109}
}

@article {john4,
    AUTHOR = {Kraisler, Joseph and Schotland, John C.},
     TITLE = {Collective spontaneous emission and kinetic equations for
              one-photon light in random media},
   JOURNAL = {J. Math. Phys.},
  FJOURNAL = {Journal of Mathematical Physics},
    VOLUME = {63},
      YEAR = {2022},
    NUMBER = {3},
     PAGES = {Paper No. 031901, 22},
      ISSN = {0022-2488,1089-7658},
   MRCLASS = {81V80},
  MRNUMBER = {4388726},
       DOI = {10.1063/5.0055171},
       URL = {https://doi.org/10.1063/5.0055171},
}

@book {colton-book,
    AUTHOR = {Colton, David and Kress, Rainer},
     TITLE = {Inverse acoustic and electromagnetic scattering theory},
    SERIES = {Applied Mathematical Sciences},
    VOLUME = {93},
   EDITION = {Fourth},
 PUBLISHER = {Springer, Cham},
      YEAR = {(2019)},
     PAGES = {xxii+518},
      ISBN = {978-3-030-30350-1; 978-3-030-30351-8},
   MRCLASS = {35-02 (35P25 35Q60 35R30 45A05 65M30 76Q05 78A46)},
  MRNUMBER = {3971246},
       DOI = {10.1007/978-3-030-30351-8},
       URL = {https://doi.org/10.1007/978-3-030-30351-8},
}

@article {MR4491134,
    AUTHOR = {Ishida, Atsuhide and Lorinczi, J\'ozsef and Sasaki, Itaru},
     TITLE = {Absence of embedded eigenvalues for non-local {S}chr\"odinger
              operators},
   JOURNAL = {J. Evol. Equ.},
  FJOURNAL = {Journal of Evolution Equations},
    VOLUME = {22},
      YEAR = {2022},
    NUMBER = {4},
     PAGES = {Paper No. 82, 30},
      ISSN = {1424-3199,1424-3202},
   MRCLASS = {35P30 (35J10)},
  MRNUMBER = {4491134},
       DOI = {10.1007/s00028-022-00836-0},
       URL = {https://doi.org/10.1007/s00028-022-00836-0},
}

@book {HormanderIII,
    AUTHOR = {H\"ormander, Lars},
     TITLE = {The analysis of linear partial differential operators. {III}},
    SERIES = {Classics in Mathematics},
      NOTE = {Pseudo-differential operators,
              Reprint of the 1994 edition},
 PUBLISHER = {Springer, Berlin},
      YEAR = {2007},
     PAGES = {viii+525},
      ISBN = {978-3-540-49937-4},
   MRCLASS = {35-02 (35Sxx 47F05 47G30 58J40)},
  MRNUMBER = {2304165},
       DOI = {10.1007/978-3-540-49938-1},
       URL = {https://doi.org/10.1007/978-3-540-49938-1},
}

@article {MR312066,
    AUTHOR = {Ikebe, Teruo and Saito, Yoshimi},
     TITLE = {Limiting absorption method and absolute continuity for the
              {S}chr\"odinger operator},
   JOURNAL = {J. Math. Kyoto Univ.},
  FJOURNAL = {Journal of Mathematics of Kyoto University},
    VOLUME = {12},
      YEAR = {1972},
     PAGES = {513--542},
      ISSN = {0023-608X},
   MRCLASS = {35J10 (47F05)},
  MRNUMBER = {312066},
MRREVIEWER = {Stanly\ L.\ Steinberg},
       DOI = {10.1215/kjm/1250523478},
       URL = {https://doi.org/10.1215/kjm/1250523478},
}

@article {umeda2003generalized,
    AUTHOR = {Umeda, Tomio},
     TITLE = {Generalized eigenfunctions of relativistic {S}chr\"odinger
              operators. {I}},
   JOURNAL = {Electron. J. Differential Equations, arXiv preprint math/0310090},
  FJOURNAL = {Electronic Journal of Differential Equations},
      YEAR = {2006},
     PAGES = {No. 127, 46},
      ISSN = {1072-6691},
   MRCLASS = {35Q75 (35P05 47A40 47G30 47N50 81Q05)},
  MRNUMBER = {2255242},
}

@book {rudin,
    AUTHOR = {Rudin, Walter},
     TITLE = {Functional analysis},
    SERIES = {McGraw-Hill Series in Higher Mathematics},
 PUBLISHER = {McGraw-Hill Book Co., New York-D\"usseldorf-Johannesburg},
      YEAR = {1973},
     PAGES = {xiii+397},
   MRCLASS = {46-01},
  MRNUMBER = {365062},
MRREVIEWER = {F.\ Smithies},
}

@book {bill,
    AUTHOR = {Kaltenbacher, Barbara and Rundell, William},
     TITLE = {Inverse problems for fractional partial differential
              equations},
    SERIES = {Graduate Studies in Mathematics},
    VOLUME = {230},
 PUBLISHER = {American Mathematical Society, Providence, RI},
      YEAR = {[2023] \copyright 2023},
     PAGES = {xiii+505},
      ISBN = {[9781470472450]; [9781470472764]},
   MRCLASS = {35-02 (26A33 35K57 35R11 35R30 47F05 60K50 65M32)},
  MRNUMBER = {4604107},
       DOI = {10.1090/gsm/230},
       URL = {https://doi.org/10.1090/gsm/230},
}

@book{LiebLoss,
  author    = {Elliott H. Lieb and Michael Loss},
  title     = {Analysis},
  edition   = {2},
  publisher = {American Mathematical Society},
  address   = {Providence, RI},
  year      = {2001},
}

@article{lozier2003nist,
  title={NIST digital library of mathematical functions},
  author={Lozier, Daniel W.},
  journal={Annals of Mathematics and Artificial Intelligence},
  volume={38},
  number={1},
  pages={105--119},
  year={2003},
  publisher={Springer}
}

@book{watson1922treatise,
  title={A treatise on the theory of Bessel functions},
  author={Watson, George Neville},
  volume={3},
  year={1922},
  publisher={The University Press}
}

@article{hansen2007asymptotically,
  title={Asymptotically precise norm estimates of scattering from a small circular inhomogeneity},
  author={Hansen, Derek J. and Poignard, Clair and Vogelius, Michael S.},
  journal={Applicable Analysis},
  volume={86},
  number={4},
  pages={433--458},
  year={2007},
  publisher={Taylor \& Francis}
}

@article{ikebe1972limiting,
  title={Limiting absorption method and absolute continuity for the Schr{\"o}dinger operator},
  author={Ikebe, Teruo and Saito, Yoshimi},
  journal={Journal of Mathematics of Kyoto University},
  volume={12},
  number={3},
  pages={513--542},
  year={1972},
  publisher={Duke University Press}
}

@article {sphere,
    AUTHOR = {Frank, Rupert L. and Lenzmann, Enno and Silvestre, Luis},
     TITLE = {Uniqueness of radial solutions for the fractional {L}aplacian},
   JOURNAL = {Comm. Pure Appl. Math.},
  FJOURNAL = {Communications on Pure and Applied Mathematics},
    VOLUME = {69},
      YEAR = {2016},
    NUMBER = {9},
     PAGES = {1671--1726},
      ISSN = {0010-3640,1097-0312},
   MRCLASS = {35R11 (35A02 35B07)},
  MRNUMBER = {3530361},
MRREVIEWER = {Pablo\ Ra\'ul\ Stinga},
       DOI = {10.1002/cpa.21591},
       URL = {https://doi.org/10.1002/cpa.21591},
}

\end{document}